\documentclass[11pt, reqno]{amsart}

\usepackage{amsmath,amsthm,amsfonts,amssymb,mathrsfs}
\usepackage{inputenc,mathrsfs}

\numberwithin{equation}{section}
\usepackage[dvips]{graphicx}
\usepackage[colorlinks]{hyperref}
\hypersetup{allcolors=blue}

\newtheorem{theorem}{Theorem}[section]
\newtheorem{definition}{Definition}[section]

\newtheorem{question}{Question}[section]
\newtheorem{proposition}{Proposition}[section]
\newtheorem{lemma}{Lemma}[section]
\newtheorem{corollary}{Corollary}[section]

\theoremstyle{remark}
\newtheorem{remark}{Remark}[section]

\DeclareMathOperator{\diam}{\mathrm{diam}}
\DeclareMathOperator{\riem}{\mathrm{Rm}}
\DeclareMathOperator{\ric}{\mathrm{Ric}}
\DeclareMathOperator{\hess}{\mathrm{Hess}}
\DeclareMathOperator{\dive}{\mathrm{div}}
\DeclareMathOperator{\vol}{\mathrm{vol}}

\author[Panagiotis Gianniotis]{Panagiotis Gianniotis}
\thanks{\\Panagiotis Gianniotis\\ Department of Mathematics,
National and Kapodistrian University of Athens, University Campus 15784, Zografou Athens\\  email: pgianniotis@math.uoa.gr}
\address{Department of Mathematics\\
National and Kapodistrian University of Athens, University Campus 15784, Zografou Athens}
\email{pgianniotis@math.uoa.gr}

\begin{document}
\title[Splitting maps in Type I Ricci flows]{Splitting maps in Type I Ricci flows}
\begin{abstract}
We study the existence and small scale behaviour of almost splitting maps along a Ricci flow satisfying  Type I curvature bounds. These are special solutions of the heat equation that serve as parabolic analogues of harmonic almost splitting maps, which have proven to be an indespensable tool in the study of the structure of the singular set of non-collapsed Ricci limit spaces.

In this paper, motivated by the recent work of Cheeger--Jiang--Naber in \cite{CJN} in the Ricci limit setting, we construct sharp splitting maps on Ricci flows that are almost selfsimilar, and then investigate their small scale behaviour. We show that, modulo linear transformations, an almost splitting map at a large scale remains a splitting map even at smaller scales, provided that the Ricci flow remains sufficiently self-similar. Allowing these linear transformations means that a priori an almost splitting map might degenerate at small scales. However, we show that under an additional summability hypothesis such degeneration doesn't occur.
\end{abstract}
\maketitle
\tableofcontents

\section{Introduction}\label{sec:introduction}

It is well known that a smooth Ricci flow $(M,g(t))_{t\in [0,T)}$ on a compact manifold can exhibit singularities in finite time $T<+\infty$, in which case the curvature blows up as $t\rightarrow T$ and the flow can not be continued smoothly past time $T$. We need to understand the structure of these singularities well enough so that we are able to construct continuations of the flow in some weak sense.

In dimension three the structure of the possible singularities is well understood by the work of Perelman \cite{Perelman1,Perelman2,Perelman3}.  
High curvature regions of the Ricci flow in dimension three are locally, at the scale of the curvature, modelled on $\kappa$-solutions, which are ancient solutions of the Ricci flow with well understood structure. In higher dimensions  such fine understanding is not available, at least without further assumptions on the curvature, as in Brendle \cite{BrendlePIC} for positive isotropic curvature. It is however known that centered blow up sequences converge to, possibly singular, gradient shrinking Ricci solitons by the work of Sesum \cite{Sesum}, Naber \cite{Naber4d} and Enders-Buzano-Topping \cite{EMT} and Buzano-Mantegazza in \cite{ManteMull}, for Type I singularities, and Bamler in \cite{BamlerScalar} and \cite{Bamler3}, for singularities with scalar curvature blowing up at Type I rate and $\mathbb F$-limits of smooth Ricci flows respectively. A classification though of the possible gradient shrinking solitons in high dimensions seems out of reach.

Even in dimension three however, there are aspects of the singularity formation of Ricci flow that are still not well-understood. For instance, Perelman in \cite{Perelman1} states

``Suppose that $g_{ij}(t)$ is defined on $M \times [1,T)$, $T < +\infty$, and goes singular as $t\rightarrow T$. Then using 12.1 we see that, as $t \rightarrow T$, either the curvature goes to infinity everywhere, and then $M$ is a quotient of either $\mathbb S^3$ or $\mathbb S^2 \times \mathbb R$, or the region of high curvature in $g_{ij}(t)$ is the union of several necks and capped necks, which in the limit turn into horns (\textit{the horns most likely have finite diameter}, but at the moment I don’t have a proof of that)"

We can view this statement as a special case of the following, more general, question:

\begin{question} \label{question} Let $(M,g(t))_{t\in [0,T)}$ be a maximal smooth Ricci flow on a compact manifold, that becomes singular as $t\rightarrow T$. Is it possible that 
\begin{equation}\label{eqn:diameter_infinity}
\limsup_{t\rightarrow T} \diam_{g(t)}(M) = +\infty.
\end{equation}
If not, is there a limit, in some sense, of $(M,g(t))$ as $t\rightarrow T$? What is the regularity structure of that limit?\end{question}

Under an additional scalar curvature bound the behaviour of the Ricci flow as $t\rightarrow T$ has been extensively studied by Wang \cite{Wang}, Chen-Wang \cite{CW1,CW2a,CW2b,CW3}, Simon \cite{SimonScalar2,SimonScalar1}, Bamler-Zhang \cite{BZ1,BZ2} and Bamler \cite{BamlerScalar}. In particular, it has been shown by Bamler \cite{BamlerScalar} that, if the scalar curvature of $(M,g(t))_{t\in [0,T)}$ is bounded, the limit of $(M,g(t_j))$ as $t_j\rightarrow T$ exists in the Gromov-Hausdorff sense and is smooth away from singularities of codimension at least $4$. Bamler-Zhang also obtain in \cite{BZ1,BZ2} distance distortion estimates, relying only on a scalar curvature bound along the flow. Moreover, in dimension $4$, the limit can only have isolated orbifold singularities, by Bamler-Zhang \cite{BZ1} and Simon \cite{SimonScalar2,SimonScalar1}.  In general however, without a uniform scalar curvature bound, it is not known whether such limit exists at all, even in the Gromov-Hausdorff sense.

The study of similar issues in the closely related context of non-collapsed Ricci limit spaces, starting from the influential work of Cheeger-Colding \cite{CC1,CC2,CC3} and continuing until recently with the work of Cheeger-Jiang-Naber \cite{CJN}, has lead to a deep understanding of the structure of the singular stratification in this setting. In particular, in \cite{CJN} it is shown, among other things, that the singular set is $k$-rectifiable.

Fundamental to the study of the singular set of non-collapsed Ricci limit spaces in \cite{CC1,CC2,CC3}, as well as of Ricci flows under scalar curvature bounds in \cite{Bamler_singular, BamlerScalar}, are the so called $(k,\varepsilon)$-splitting maps. These are harmonic functions $u:B_1(p)\rightarrow\mathbb R^k$ that quantify the extent that the manifold locally splits $k$ Euclidean factors, by means of $L^2$ estimates on the Hessian and the gradient. In particular, in \cite{CJN}, splitting maps are used to construct bi-Lipschitz charts of the singular set, up to a set of $n-2$-Hausdorff measure zero, which implies its rectifiability.  For this, it is crucial to understand the small scale behaviour of splitting maps, the main issue being that splitting maps might degenerate at small scales,  their derivative tending to become more and more singular at small scales. At \cite{CJN} it is shown that splitting maps in a non-collapsed manifold with lower Ricci bounds can only degenerate in a controlled manner, essentially via linear transformations that blow up with at most H\"older rate. Moreover, around most points of an approximation of the singular set, splitting maps don't degenerate at all.

In this paper, we investigate similar issues for a parabolic analogue of the notion of a $(k,\varepsilon)$-splitting map in the setting of a Ricci flow $(M,g(t))_{t\in [0,T]}$. Such maps, at scale $r>0$ are solutions $v: M\times [T-r^2,T]\rightarrow \mathbb R^k$ to the heat equation along Ricci flow, whose Hessian is small and the gradient vector fields of the components of $v$ are almost orthonormal, in a space-time $L^2$ sense with respect to a conjugate heat kernel measure centered around $(p,T)$. In contrast to the Ricci limit case, these assumptions are global on $M$, not just in parabolic balls around $p$, see Definition \ref{def:splitting_map_intro} below for more details.

In forthcoming work, motivated by \cite{CJN}, we aim to address Question \ref{question} under the assumption that a Ricci flow $(M,g(t))_{t\in [0,T)}$ satisfies a Type I bound on the curvature, namely for every $t\in [0,T)$,
\begin{equation}\label{eqn:intro_type_i}
\max_M |\riem(g(t))|\leq \frac{C_I}{T-t},
\end{equation}
which may, or may not, be smooth up to $t=T$.  This paper sets the foundations for this work, addressing the degeneration behaviour of almost splitting maps. In particular, we investigate the following issues.
\begin{enumerate}
\item \underline{Existence of optimal splitting maps at each scale:} we will show that if a smooth compact Ricci flow $(M,g(t))_{t\in [0,T]}$ satisfying \eqref{eqn:intro_type_i}, is almost selfsimilar at scale $1$ around $p\in M$ and it almost splits $k$ Euclidean factors, then there exists a $(k,\varepsilon)$-splitting map whose Hessian is \textit{linearly} controlled in terms of a quantity involving the drop of the pointed $\mathcal W$-entropy of an appropriate collection of points around $p$. We call this quantity ``entropy $k$-pinching around $p$ at scale $1$", and it is a natural analogue of a similar quantity that plays a central role in \cite{CJN}. This is the content of Theorem \ref{intro_thm:sharp_splitting_maps} below.
\item \underline{Propagating splitting maps in small scales:} we will show that given a good enough splitting map $v:M\times [-1,0]\rightarrow \mathbb R^k$ at scale $1$ around $p\in M$, as long as the Ricci flow $(M,g(t))_{t\in [-1,0]}$ is almost selfsimilar, it almost splits $k$ Euclidean factors, and is far from splitting an additional Euclidean factor down to an arbitrarily small scale $r>0$, then at any scale $s\in [r,1]$, $v|_{M\times [-s^2,0]}$ can be linearly transformed to a slightly worse almost splitting map. Moreover, the Hessian of these small scale splitting maps is controlled by the sum of entropy $k$-pinching around $p$ at all  scales larger than $s$. This is the content of Theorem \ref{intro_thm:GTT} below.
\item \underline{Non-degeneration of almost splitting maps:} we show that if the sum of entropy $k$-pinching around $p$ is small down to arbitrarily small scales, then an almost splitting map at a large scale remains an almost splitting map at all smaller scales - possibly with some loss in the estimates. This is the content of Theorem \ref{intro_thm:non_degen} below.
\end{enumerate}

Theorems \ref{intro_thm:sharp_splitting_maps}, \ref{intro_thm:GTT} and \ref{intro_thm:non_degen} are consequences of Theorems \ref{thm:sharp_splitting}, \ref{thm:transformations} and \ref{thm:non_degen} respectively, which are stated in terms of a set of a priori assumptions introduced in Section \ref{sec:apriori}. These assumptions are valid for compact Ricci flows satisfying a Type I bound \eqref{eqn:intro_type_i} and a non-collapsing assumption, by Proposition \ref{prop:RF35}. Most of these a priori assumptions, such as upper and lower Gaussian bounds for conjugate heat kernels and lower distance distortion bound are still valid even under a weaker Type I curvature bound on the \textit{scalar} curvature, by \cite{BZ1,Hallgren_scalar}. It thus seems very likely that the results of the paper hold even under this weaker assumption, however this will be treated elsewhere in the future.

It is important to note that Bamler in \cite{Bamler3} introduces a similar notion of splitting maps for the Ricci flow, which plays a central role in the structure theory for smooth Ricci flows and $\mathbb F$-limits of smooth Ricci flows developed in \cite{Bamler3}, which is in the spirit of the theory developed in  \cite{CC1,CC2,CC3} and \cite{CheegerNaber} for non-collapsed Ricci limit spaces. In this context, Gaussian upper bounds for the conjugate heat kernel are available, however lower Gaussian bounds are not, at least in this generality.

\subsection{Results and main ideas}
 
Before we can state the main results in more detail, we need to fix some notation and terminology. Given a compact Riemannian manifold $(M^n,g)$, $f\in C^\infty(M)$ and $\tau>0$ let
$$\mathcal W(g,f,\tau)=\int_M (\tau(R+|\nabla f|^2) + f-n)(4\pi \tau)^{-n/2} e^{-f} d\vol_g,$$
and 
\begin{align*}
\mu(g,\tau)&=\inf\left\{ \mathcal W(g,f,\tau), f\in C^\infty(M) \quad \textrm{with} \quad\int_M (4\pi \tau)^{-n/2} e^{-f} d\vol_g=1\right\},\\
\nu(g,\tau) &= \inf \{ \mu(g,\tau'), 0<\tau'\leq \tau\}.
\end{align*}

Let $I\subset \mathbb R$ be an interval with $\max I=0$, $(M,g(t))_{t\in I}$ be smooth compact Ricci flow and $p\in M$. By Perelman \cite{Perelman1}, $\mu$ and $\nu$ are monotone under Ricci flow in the sense that for any $\tau>0$ and $t_1<t_2$
\begin{align*}
\mu(g(t_1), \tau +t_2-t_1) &\leq \mu(g(t_2),\tau),\\
\nu(g(t_1), \tau + t_2-t_1)&\leq \nu(g(t_2),\tau).
\end{align*}
Moreover, denoting by $u_{(p,0)} =(4\pi |t|)^{-n/2} e^{-f}$ the conjugate heat kernel starting at $(p,0)$, see Section \ref{sec:preliminaries}, and $d\nu_{(p,0),t} = u_{(p,0)}(\cdot,t) d\vol_{g(t)}$, we define the pointed entropy at $p$ as
$$\mathcal W_p(\tau)=\mathcal W(g(-\tau),f(\cdot,-\tau),\tau),$$
 which is also non-increasing with respect to $\tau$.

\begin{definition}[$(k,\delta)$-splitting map]
\label{def:splitting_map_intro}
Let $(M^n,g(t),p)_{t\in [-r^2,0]}$ be a smooth complete pointed Ricci flow and let $1\leq k\leq n$. Then $v=(v^1,\ldots,v^k): M\times (-r^2,0)\rightarrow \mathbb R^k$ is a $(k,\delta)$-splitting map around $p$ at scale $r$ if
\begin{enumerate}
\item Each $v^a$ solves the heat equation $\frac{\partial v^a}{\partial t} = \Delta_{g(t)} v^a$.
\item For any $a=1,\ldots, k$
\begin{equation}
\int_{-r^2}^{-\delta r^2}\int_{M} | \hess_{g(t)} v^a|^2 d\nu_{(p,0),t} dt \leq\delta.
\end{equation}
\item For any $a,b=1,\ldots,k$
\begin{equation}
r^{-2}\int_{-r^2}^{-\delta r^2}\int_{M} \left| \langle\nabla v^a,\nabla v^b\rangle -\delta^{ab} \right|^2  d\nu_{(p,0),t} dt\leq\delta.
\end{equation}
\end{enumerate}
\end{definition}
 A quantity which will prove crucial in controlling splitting maps in a Ricci flow is the \textit{entropy $k$-pinching} $\mathcal E^{(k,\mu,\delta,R)}_r(p)$ at a point $p\in M$ and some scale $r>0$. This quantity will eventually quantify, in terms of the pointed entropy, the extent that the flow around $p$ at scale $r>0$ is close to a selfsimilar Ricci flow splitting $k$ Euclidean factors.
 
  In more detail, denoting by $B(x,t,r)$ the open ball centred at $x$ and radius $r>0$ with respect to the metric $d_{g(t)}$, the entropy $k$-pinching is defined roughly as
$$\mathcal E^{(k,\mu,\delta,R)}_r(p) = \inf_{\begin{array}{c}\{x_i\}_{i=0}^k\subset B(p,-r^2,Rr)\\ \textrm{satisfying (1) and (2) below} \end{array}} \left\{ \sum_{i=0}^k \left(\mathcal W_{x_i}(r^2) - \mathcal W_{x_i} (T r^2)\right) \right\},
$$
where the infimum is taken over all subsets $\{x_i\}_{i=0}^k\subset B(p,-r^2,Rr)$ with the following properties
\begin{enumerate}
\item For each $i=0,\ldots,k$, either $\mathcal W_{x_i}(\delta r^2)-\mathcal W_{x_i}(\delta^{-1}r^2)<\delta$ or the flow around $x_i$, at scale $r$, is $\delta$-close to a selfsimilar Ricci flow induced by a shrinking Ricci soliton. 
\item The points $\{x_i\}_{i=0}^k$ can't be embedded into a product metric space of the form $K\times \mathbb R^l$ with $\diam(K)\leq D'r$ and distortion less than $(D'+\mu R)r$, when $l<k$.
\end{enumerate}
Here $D'<+\infty$ is meant to be a constant much smaller than $R$, while $T<+\infty$ is a large constant. Both will eventually depend on the dimension and the available upper and lower Gaussian bounds on conjugate heat kernels along the Ricci flow.

In general, we will say that a pointed Ricci flow  $(M,g(t),p)_{t\in(-\infty,0)}$ is $k$-selfsimilar, if it is induced by a gradient shrinking Ricci soliton splitting \textit{at least} $k$ Euclidean factors and $p$ is a minimum of a soliton function. We will also say that a pointed Ricci flow $(M,g(t),p)_{t\in (-2\delta^{-1}r^2,0)}$ is $(k,\delta)$-selfsimilar at scale $r$, if up to scaling by the factor $r^{-2}$ it is close to a $k$-selfsimilar Ricci flow, in a sense that is specified in Section \ref{sec:selfsimilar}.

Note that if a Ricci flow $(M,g(t),p)_{t\in (2\delta^{-1},0]}$ is $(k,\delta)$-selfsimilar at scale $1$, we could easily construct a $(k,\varepsilon)$-splitting map around $p$ at scale $1$. However, the control on the Hessian of such splitting map would be of order $\varepsilon$, as in Definition \ref{def:splitting_map_intro}. It turns out that we can do better, and our first result asserts the existence of splitting maps with the square of the $L^2$ norm of their Hessian linearly controlled in terms of the entropy $k$-pinching.

\begin{theorem}\label{intro_thm:sharp_splitting_maps}
Fix $\varepsilon>0$, $\mu\in (0,1/6)$. Let $(M^n,g(t),p)_{t\in [-10\delta^{-2},0]}$ be a smooth compact Ricci flow satisfying
\begin{align*}
\max_M |\riem|(\cdot,t) &\leq \frac{C_I}{|t|},\\
\nu(g(-10\delta^{-2}),20\delta^{-2})&\geq -\Lambda,
\end{align*}
for every $t\in [-10\delta^{-2},0)$, and let $R\geq \frac{D'}{\mu}$, where $D'=D'(n,C_I,\Lambda)<+\infty$. 

Suppose that $(M,g(t),q)_{t\in(-2\delta^{-2},0)}$ is $(k,\delta^2)$-selfsimilar at scale $1$ at $q\in B(p,-1,R)$, for some $1\leq k\leq n$, and if $q\not = p$ suppose that
$$\mathcal W_p(\delta)-\mathcal W_p(\delta^{-1})<\delta.$$

If $0<\delta\leq \delta(n,C_I,\Lambda|R,\mu,\varepsilon)$, then there is a $(k,\varepsilon)$-splitting map $v=(v^1,\ldots,v^k):M\times [-1,0]\rightarrow \mathbb R^k$ at scale $1$ around $p$ such that for any $a=1,\ldots,k$
\begin{equation*}
\int_M |\hess  v^a |^2 d\nu_{(p,0),t} \leq C(n,C_I,\Lambda|R,\mu,\varepsilon) \mathcal E^{(k,\mu,\delta,R)}_1(p),
\end{equation*}
for every $t\in [-1,-\varepsilon]$.
\end{theorem}

The main idea behind the proof of Theorem \ref{intro_thm:sharp_splitting_maps} is that given a Ricci flow $(M,g(t))_{t\in (-2\delta^{-2},0]}$ and $x\in M$, we can a construct solution $w$ to the equation $\left(\frac{\partial}{\partial t} - \Delta_{g(t)} \right) w=-\frac{n}{2}$ which can be considered as a parabolic regularization of the function $|t|f_{(x,0)}$, defined from the conjugate heat kernel $(4\pi|t|)^{-n/2} e^{-f_{(x,0)}}$ based at $(x,0)$. 

It turns out that such function provides a good approximation of a soliton function, once $(M,g(t),x)_{t\in (-2\delta^{-2},0]}$ is $(0,\delta)$-selfsimilar. In fact we will obtain an estimate of the form
\begin{equation}\label{eqn:intro_w}
\int_M \left| |t| \ric_{g(t)} + \hess_{g(t)} w - \frac{g(t)}{2}\right|^2 d\nu_{(x,0),t} \leq C \left(\frac{|T|}{|t|}\right)^E(\mathcal W_x(1)-\mathcal W_x(2T)),
\end{equation}
for any $t\in [-T,0)$, as opposed to the estimate
$$\int_{-2T}^{-1} \int_M |t| \left| \ric_{g(t)} + \hess_{g(t)} f_{(x,0)} - \frac{g(t)}{2|t|}\right|^2 d\nu_{(x,0),t} =\mathcal W_x(1)-\mathcal W_x(2T),$$
we would obtain by simply applying the monotonicity formula of $\mathcal W_x$.

When $(M,g(t),p)_{t\in (-2\delta^{-2},0]}$ is $(k,\delta^2)$-selfsimilar, it will turn out that there are points $\{x_i\}_{i=0}^k$, satisfying the properties in the definition of the $k$-pinching entropy, for which we construct $\{w^i\}_{i=0}^k$ solving  $\left(\frac{\partial}{\partial t} - \Delta_{g(t)} \right) w^i=-\frac{n}{2}$. The functions $v^a=w^a-w^0$, $a=1,\ldots,k$ then solve the heat equation, and  they are a good starting point to construct a $(k,\varepsilon)$-splitting map $v=(v^1,\ldots,v^k)$, controlled by entropy $k$-pinching.

A technical issue is how to relate the integral estimates \eqref{eqn:intro_w} of each of the functions $w_i$, which are in terms of the conjugate heat kernels $u_{(x_i,0)}$, with integral estimates with respect to the conjugate heat kernel based at $(p,0)$. This is necessary in order to obtain an integral hessian estimate with respect to the conjugate heat kernel based at $(p,0)$, for each of the functions $v^a$. Unfortunately, conjugate heat kernels based at different points are not comparable, even in the static Euclidean Ricci flow in $\mathbb R^n$. Namely, it is not true that $u_{(p,0)}\leq C u_{(x_i,0)}$. We will see however that an estimate of the form $u_{(p,0)}\leq C e^{\alpha f_{(x_i,0)}} u_{(x_i,0)}$ does hold, provided that $\alpha\in(0,1)$ is close enough to $1$. This method of comparing conjugate heat kernel measures at different points originates in \cite{Bamler3}. Our situation however is simpler, since the upper and lower conjugate heat kernel bounds available in our setting can be utilized to prove such an estimate directly, in Lemma \ref{lemma:compare_kernels}.

Promoting then estimate \eqref{eqn:intro_w} for each of the functions $w^i$ to a stronger estimate of the form
$$\int_M \left| |t| \ric_{g(t)} + \hess_{g(t)} w^i - \frac{g(t)}{2}\right|^2 e^{\alpha f_{(x_i,0)}}d\nu_{(x_i,0),t} \leq C |t|^{-E}(\mathcal W_x(1)-\mathcal W_x(2T))$$
will be due to the hypercontractivity property of solutions to the heat equation under Ricci flow.

Hallgren in \cite{Hallgren_Kahler} constructs similar regularizations of conjugate heat kernel potentials, and manages to show proximity to the potential of a conjugate heat kernel without assuming any Type I curvature bounds. Both in Theorem \ref{intro_thm:sharp_splitting_maps} as well as in \cite{Hallgren_Kahler}, the regularizations are constructed by using the conjugate heat kernel potentials as an initial condition to construct solutions to a forward parabolic equation. However, in our case, we need to establish the stronger estimate \eqref{eqn:intro_w}. Such estimate, for now, requires controlling the global contribution of the curvature terms in the evolution of a Lichnerowicz heat equation, which is one of the reason a Type I curvature bound is needed for our argument.

Now that we have established the existence of an optimal splitting map at a certain scale, the following theorem is concerned with its behaviour in smaller scales.

\begin{theorem}\label{intro_thm:GTT}
Fix $\varepsilon>0$, $\eta>0$, $\mu\in(0,1/6)$. Let $(M^n,g(t),p)_{t\in [-10\delta^{-2},0]}$ be a smooth compact Ricci flow satisfying \begin{align*}
\max_M |\riem|(\cdot,t) &\leq \frac{C_I}{|t|},\\
\nu(g(-10\delta^{-2}),20\delta^{-2})&\geq -\Lambda,
\end{align*}
for every $t\in [-10\delta^{-2},0)$, and let $R\geq \frac{D'}{\mu}$, where $D'=D'(n,C_I,\Lambda)<+\infty$. 

Let $v$ be a $(k,\delta)$-splitting map $v$ around $p$ at scale $1$, and suppose that there are $r\in (0,1)$ and $1\leq k\leq n$ such that at every scale $s\in [r,1]$, there is $q_s\in B(p,-s^2,Rs)$ such that $(M,g(t),q_s)_{t\in (-2\delta^{-2},0)}$ is $(k,\delta^2)$-selfsimilar but not $(k+1,\eta)$-selfsimilar at scale $s$, and if $q_s\not = p$
$$\mathcal W_p(\delta s^2)-\mathcal W_p(\delta^{-1} s^2)<\delta.$$

If $0<\delta\leq\delta(n, C_I,\Lambda,\eta|R,\mu,\varepsilon)$, then for every $s\in [r,1]$ there is a lower triangular $k\times k$ matrix $T_s$ such that
\begin{enumerate}
\item $v_s:=T_s v$ is a $(k,\varepsilon)$-splitting map around $p$ at scale $s$ normalized so that for any $a,b=1,\ldots,k$
\begin{equation*}
\frac{4}{3s^2} \int_{-s^2}^{-s^2/4} \int_M \langle \nabla v_s^a,\nabla v_s^b\rangle d\nu_{(p,0)} dt = \delta^{ab}.
\end{equation*}
\item Let $r_j= \frac{1}{2^j}$. There is $\theta=\theta(n,C_I,\Lambda)\in(0,1)$ such that for every $s\in [r,1]$ and $a=1,\ldots,k$
\begin{align*}
 \int_{-s^2}^{-\frac{s^2}{4}} \int_M |\hess v_s^a|^2 d\nu_{(p,0)} dt \leq C(n,C_I,\Lambda,\eta|R,\mu)  \left(\sum_{s\leq r_j\leq 1} \left( \frac{s}{r_j} \right)^\theta \mathcal E^{(k,\mu,\delta,R)}_{r_j} (p) +\delta s^\theta\right).
\end{align*}
\end{enumerate}
\end{theorem}

The proof of each part of Theorem \ref{intro_thm:GTT} is via contradiction, exploiting a spectral gap of the drift Laplacian on gradient shrinking Ricci solitons that split $k$ Euclidean factors, but are ``$\eta$-far" from splitting more than $k$ factors. In particular, the first non-trivial eigenvalue of the drift Laplacian appears with multiplicity exactly $k$, namely $\lambda_1=\cdots=\lambda_k=-\frac{1}{2}$ and there is a $\delta>0$ such that $\lambda_{i}\leq -\delta$ for every $i\geq k+1$, where $\delta$ depends on the curvature, the entropy of the soliton and $\eta$.   This spectral gap property leads to two consequences, Lemma \ref{lemma:growth_linear} and Lemma \ref{lemma:hessian_decay_solitons}. The first is a rigidity statement for solutions of the heat equation on an evolving gradient shrinking Ricci soliton, which forces solutions with slow growth as $t<<-1$ to be linear. This allows us to show that, should Assertion (1) of the theorem fail at some scale $s>r$, the maps $v_s$ become increasingly better splitting maps, which leads to a contradiction. The second asserts that the Hessian of a solution to the heat equation on an evolving shrinking Ricci soliton decays at a H\"older rate. This allows us to show that the Hessian of the ``non-linear" part of a solution to the heat equation decays at a H\"older rate, whereas the Hessian of the ``linear" part is controlled by entropy $k$-pinching, via Theorem \ref{thm:sharp_splitting}.

There are similarities between the proof of Theorem \ref{intro_thm:GTT} and the ``extend and improve" procedure exploited by Colding-Minicozzi in \cite{CM_singularities}, where they prove that generalized cylinders are isolated as gradient shrinking Ricci solitons, in the pointed smooth Cheeger-Gromov topology. In the setting of Theorem \ref{intro_thm:GTT}, ``extending" amounts to showing that an almost splitting map at some scale is still  an almost splitting map at some smaller scale (i.e. forward in time), with a possible loss in the estimates. On the other hand,  ``improving" amounts to that at the lowest scale of our contradiction argument we show that the splitting map is actually better than it is a priori assumed. Both here and in \cite{CM_singularities}, the mechanism at play is that a rough estimate outwards or backwards in time, in \cite{CM_singularities} and Theorem \ref{intro_thm:GTT} respectively, implies a better estimate at a slightly smaller scale, or forward in time. In our case, the rough estimate comes from assuming that $v$ is an almost splitting map at all scales larger than $s$, whereas in \cite{CM_singularities} by pseudolocality. Moreover, as in our case, the improvement mechanism in \cite{CM_singularities} relies on a spectral gap, but, unlike \cite{CM_singularities}, we don't need to impose any additional assumptions on the shrinking solitons involved - the spectral gap merely depends on the soliton being far from splitting an additional Euclidean factor.

A priori, the matrices $T_s$ in Theorem \ref{intro_thm:GTT} might satisfy $||T_s|| \rightarrow +\infty$, leading to the degeneration in small scales of the initial splitting map $v$. A situation such degeneration might develop is the following. Consider a rotationally symmetric Ricci flow on $\mathbb S^n$ forming a non-degenerate neckpinch singularity along a $n-1$-sphere, at $t=0$, with bounded diameter. An almost splitting map at a small scale around a singular point will be a solution to the heat equation with almost vanishing Hessian and gradient with norm $\approx 1$ along a neck of size $\approx \sqrt{|t|}$, that achieves its maximum/minimum values at the poles. One possibility, as the splitting map evolves under the heat equation and diffuses from one side of the neck-pinch to the other, is that the norm of the gradient in balls of radius $\approx \sqrt{|t|}$ around a singular point will decay to zero, forcing the corresponding linear transformations to blow up.

However, in Theorem \ref{intro_thm:non_degen} below, if the sum of the entropy $k$-pinching along all scales $s\in [r,1]$ is small, we can exploit the Hessian decay estimate of Theorem \ref{intro_thm:GTT} to show that in this case the transformations $T_s$ remain close to the identity at all scales $s\in [r,1]$, and thus a splitting map $v$ at a large scale remains a splitting map even at smaller scales.

\begin{theorem}\label{intro_thm:non_degen}
Fix $\varepsilon>0$, $\eta>0$, $\mu\in (0,1/6)$. Let $(M^n,g(t),p)_{t\in [-10\delta^{-2},0]}$ be a smooth compact Ricci flow satisfying
\begin{align*}
\max_M |\riem|(\cdot,t) &\leq \frac{C_I}{|t|},\\
\nu(g(-10\delta^{-2}),20\delta^{-2})&\geq -\Lambda
\end{align*}
for every $t\in [-10\delta^{-2},0]$, and let $R\geq \frac{D'}{\mu}$, where $D'=D'(n,C_I,\Lambda)<+\infty$. 

Let $v$ be a $(k,\delta)$-splitting map $v$ around $p$ at scale $1$, and suppose that there is $r\in (0,1)$ and $1\leq k\leq n$ such that at every scale $s\in [r,1]$, there is $q_s\in B(p,-s^2,Rs)$, such that $(M,g(t),q_s)_{t\in (-2\delta^{-2},0)}$ is $(k,\delta^2)$-selfsimilar but not $(k+1,\eta)$-selfsimilar at scale $s$, and if $q_s\not = p$
$$\mathcal W_p(\delta s^2)-\mathcal W_p(\delta^{-1} s^2)<\delta.$$

If for $r_j=2^{-j}$, $j\geq 0$ integer, we have
\begin{equation}\label{eqn:kpinch_sum}
\sum_{r \leq r_j\leq 1} \mathcal E^{(k,\mu,\delta,R)}_{r_j}(p) <\delta, 
\end{equation}
and $0<\delta\leq\delta(n,C_I,\Lambda,\eta|R,\mu,\varepsilon)$, then for every $s\in [r,1]$, $v: M\times [-s^2,0]\rightarrow \mathbb R^k$ is a $(k,\varepsilon)$-splitting map around $p$ at scale $s$.
\end{theorem}

An extrinsic analogy for Theorems \ref{intro_thm:sharp_splitting_maps},  \ref{intro_thm:GTT} and \ref{intro_thm:non_degen}, for the mean curvature flow, would be the following. Suppose that a mean curvature flow is close to a shrinking cylinder mean curvature flow at some scale. In this situation, Theorem \ref{intro_thm:sharp_splitting_maps} would correspond to finding the axis of a cylinder that best approximates the flow at this scale. Theorem \ref{intro_thm:GTT} can then be compared with being able to adjust the axes of the optimally approximating cylinders at each scale, while being able to control the accumulation of the total error. Finally,  Theorem \ref{intro_thm:non_degen} would mean that we can choose a fixed cylindrical mean curvature flow that approximates the mean curvature flow around a singular point down to arbitrarily small scales. 

Indeed, for the mean curvature flow this is exactly the behaviour one expects from the Lojasiewicz inequality of Colding--Minicozzi in  \cite{CM_uniqueness_MCF}, which in fact establishes the uniqueness of tangent flows at a point where the mean curvature flow exhibits a cylindrical singularity. There, the distance from the optimal cylinder and the evolution of the hypersurface are controlled in terms of Huisken's local monotone quantity for the mean curvature flow, as opposed to the entropy $k$-pinching here.  The monotonicity of Huisken's quantity immediately gives the required summability that leads to a control of the flow down to arbitrarily small scales. Unfortunately, in our case, the entropy $k$-pinching is not monotone with respect to scale, and thus it is not clear whether condition \eqref{eqn:kpinch_sum} always holds below some scale. However, we will see in forthcoming work that condition \eqref{eqn:kpinch_sum} \textit{does} hold around enough points to enable us  to transfer large scale information down to arbitrarily small scales, leading to a result similar to the neck structure theorem for non-collapsed Ricci limit spaces in \cite{CJN}.

\subsection{Structure of the paper}
In Section \ref{sec:preliminaries} we review some preliminary facts about the heat and conjugate heat equation on a Ricci flow, Perelman's entropy and gradient shrinking Ricci solitons.

In Section \ref{sec:apriori} we describe the set of a priori assumptions under which the results of this paper hold, and present some of their basic consequences that will be used throughout the paper. Then we discuss why these a priori assumptions hold in the setting of Ricci flows satisfying a Type I bound on the curvature.

In Section \ref{sec:selfsimilar} we set up a theory for selfsimilar Ricci flows. We introduce a notion of a \textit{spine} for such Ricci flows and investigate properties of the spine when a selfsimilar Ricci flow arises as a limit of Ricci flows.

In Section \ref{sec:heat_rf} we prove some technical estimates for the heat equation under a Ricci flow.  

In Section \ref{sec:heat_gsrs} we prove a spectral gap for the drift Laplacian on a gradient shrinking Ricci soliton, and investigate consequences in the behaviour of solutions of the heat equation on a Ricci flow induced by a gradient shrinking Ricci soliton. 

In Section \ref{sec:splitting_maps} we establish some preliminary properties of splitting maps.  

In Section \ref{sec:construction} we construct sharp splitting maps controlled by the entropy $k$-pinching and prove Theorem \ref{intro_thm:sharp_splitting_maps}.

In Section \ref{sec:GTT} we state and prove Theorem \ref{intro_thm:GTT}.

In Section \ref{sec:non_degeneration} we state and prove the non-degeneration Theorem \ref{intro_thm:non_degen}.

\section{Preliminaries} \label{sec:preliminaries}

\subsection{The heat and conjugate heat equation}  Let $(M,g(t))_{t\in I}$ be smooth compact Ricci flow and let $v,u\in C^\infty(M\times I)$. Integration by parts then gives that
\begin{equation}\label{eqn:integral_product}
\frac{d}{dt} \int_M v u d\vol_{g(t)} =\int_M  \left(\frac{\partial}{\partial t} -\Delta_{g(t)}\right) v  u d\vol_{g(t)} + \int_M v \left(\frac{\partial}{\partial t} +\Delta_{g(t)} - R\right) u d\vol_{g(t)}.
\end{equation}
In particular, if $v,u$ evolve by the heat equation 
\begin{equation}\label{eqn:heat_equation_prelim}
\left(\frac{\partial}{\partial t} -\Delta_{g(t)}\right) v =0,
\end{equation}
and the conjugate heat equation
\begin{equation}\label{eqn:conj_heat_equation_prelim}
\left(\frac{\partial}{\partial t} +\Delta_{g(t)}-R_{g(t)}\right) u =0
\end{equation}
respectively, then
\begin{equation*}
\frac{d}{dt} \int_M v u d\vol_{g(t)}=0.
\end{equation*}
Similarly, integrating \eqref{eqn:conj_heat_equation_prelim} we obtain that
\begin{equation}\label{eqn:chf_constant}
\frac{d}{dt} \int_M u d\vol_{g(t)} = 0.
\end{equation}

Moreover, a direct compuation shows that $u(\cdot,t)=(4\pi |t_0-t|^2) e^{-f(\cdot,t)}$, $t<t_0$, is a positive solution to the conjugate heat equation if and only if
\begin{equation}\label{eqn:f_che}
\frac{\partial f}{\partial t} = -\Delta_{g(t)} f + |\nabla^{g(t)} f|^2 - R_{g(t)} + \frac{n}{2|t_0-t|}.
\end{equation}

For every $(x,t)\in M\times I$ there is a unique positive solution to the conjugate heat equation $u_{(x,t)}(y,s)$, for $s<t$, such that for any smooth function $\varphi$
\begin{equation}\label{eqn:chk_dirac}
\lim_{s\rightarrow t} \int_M \varphi (y) u_{(x,t)}(y,s) d\vol_{g(s)}(y) = \varphi(x),
\end{equation}
see for instance \cite{RFpartIII}. Moreover, by \eqref{eqn:chf_constant}, it follows that
$$\int_M u_{(x,t)} (y,s)d\vol_{g(s)}(y)=1,$$ for every $s<t$. Thus, $d\nu_{(x,t),s} = u_{(x,t)}(\cdot,s) d\vol_{g(s)}$ defines a probability measure $\nu_{(x,t),s}$ on $M$, for every $s<t$.
We will call $u_{(x,t)}$ the conjugate heat kernel based at $(x,t)$ and the family  $\nu_{(x,t)}=(\nu_{(x,t),s})_{s<t}$ the conjugate heat kernel flow based at $(x,t)$.  

A direct consequence of \eqref{eqn:integral_product} and \eqref{eqn:chk_dirac} is that if $v$ satisfies $\left(\frac{\partial}{\partial t} -\Delta_{g(t)}\right) v=F$, for some  function $F$ in $M\times I$ then for every $(x,t)\in M\times I$ and $s<t$
\begin{equation}\label{eqn:duhamel}
v(x,t) = \int_M v(y,s) d\nu_{(x,t),s} + \int_s^t \int_M F(y,u) d\nu_{(x,t),u} du
\end{equation}

The conjugate heat kernel under Ricci flow satisfies the following Poincar\'e inequality, due to Hein--Naber \cite{HeinNaber}.

\begin{theorem}[Theorem 1.10 in \cite{HeinNaber}]\label{thm:poincare}
Let $(M,g(t))_{t\in I}$, $\max I=0$, be a smooth Ricci flow on a compact manifold and $p\in M$. Then, for every $s\in I\cap (-\infty,0)$ and $\varphi\in C^\infty(M)$ with $\int_M \varphi d\nu_{(p,0),s} =0$,
\begin{equation*}
\int_M \varphi^2 d\nu_{(p,0),s} \leq 2|s| \int_M |\nabla \varphi|^2 d\nu_{(p,0),s}.
\end{equation*}
\end{theorem}

\subsection{Entropy and no local collapsing}

\begin{definition}\label{def:W_entropy}
Given a complete Riemannian manifold $(M^n,g)$ with bounded curvature, $f\in C^\infty(M)$ and $\tau>0$ define
\begin{equation}
\mathcal W(g,f,\tau)=\int_M \left(\tau (R+|\nabla f|^2) + f-n\right) u d\vol_g,
\end{equation}
where $u=(4\pi \tau)^{-n/2}e^{-f}$. 

We also define
\begin{align*}
\mu(g,\tau)&=\inf\left\{ \mathcal W(g,f,\tau), f\in C^\infty(M) \quad \textrm{with} \quad\int_M (4\pi \tau)^{-n/2} e^{-f} d\vol_g=1\right\},\\
\nu(g,\tau) &= \inf \{ \mu(g,\tau'), 0<\tau'\leq \tau\}.
\end{align*}
\end{definition}
Let $(M,g(t))_{t\in I}$ be a Ricci flow on a closed manifold, $\max I =0$, and suppose that $u=(4\pi |t|)^{-n/2} e^{-f}$ evolves by the conjugate heat equation along $g(t)$. Then, by \cite{Perelman1}, for every $t\in I \cap \{t<0\}$
\begin{equation}\label{eqn:monotonicity_formula}
\frac{d}{dt} \mathcal W(g(t),f(t),|t|) = 2|t| \int_M \left| \ric_{g(t)} + \hess_{g(t)} f(\cdot,t) - \frac{g(t)}{2|t|} \right|^2 u(\cdot,t) d\vol_{g(t)} \geq 0.
\end{equation}
From this it follows that $\mu$ and $\nu$ are non-decreasing along Ricci flow, in the sense that for any $\tau>0$ and $t_1<t_2$
\begin{align*}
\mu(g(t_1),\tau+t_2-t_1)\leq \mu(g(t_2),\tau),\\
\nu(g(t_1),\tau+t_2-t_1)\leq \nu(g(t_2),\tau).
\end{align*}

\begin{definition}\label{def:p_entropy}
Let $(M,g(t))_{t\in I}$, $\max I=0$, be a smooth Ricci flow and $p\in M$. For any $\tau>0$ such that $-\tau\in I$, we define
 \begin{equation*}
\mathcal W_p (\tau) = \mathcal W(g(-\tau),f(\cdot,-\tau),\tau),
\end{equation*}
where  the function $u_{(p,0)}(x,t)=(4\pi |t|)^{-n/2} e^{-f(x,t)}$ is the backwards conjugate heat kernel starting at $(p,0)$.
\end{definition}

In the setting of Definition \ref{def:p_entropy}, \eqref{eqn:monotonicity_formula} implies that the function $\tau\mapsto \mathcal W_p(\tau)$ is monotone decreasing and
\begin{equation}\label{eqn:pointed_monotonicity}
\frac{d}{d t} \mathcal W_p(|t|) = -2|t|\int_M \left| \ric_{g(t)} + \hess_{g(t)} f(\cdot,t) - \frac{g(t)}{2|t|} \right|^2 d\nu_{(p,0),t}.
\end{equation}
Moreover, $\mathcal W_p(\tau)\leq 0$, by Perelman's differential Harnack inequality \cite{Perelman1}.

\begin{theorem}\label{thm:entropy_noncollapsing}
Let $(M^n,g)$ be a compact Riemannian manifold and let $0<s\leq r$. Suppose that $R\leq r^{-2}$ inside the ball $B(p,r)$. Then 
\begin{equation}
\frac{\vol_g(B(p,s))}{s^n} \geq C(n) e^{\nu(g,r^2)}.
\end{equation}
\end{theorem}

\begin{definition}
We say that $(M^n,g)$ is $\kappa$-noncollapsed below scale $r>0$ if for every $x\in M$ such that $R\leq r^{-2}$ in $B(x,r)$ we have
$$\frac{\vol(B(x,s))}{s^n} \geq \kappa,$$
for every $0<s\leq r$.
\end{definition}

\begin{corollary}
Let $(M,g)$ is a compact Riemannian manifold such that $\nu(g,r^2)\geq -\Lambda$. Then there is $\kappa=\kappa(n,\Lambda)$ such that $(M,g)$ is $\kappa$-noncollapsed below scale $r$.
\end{corollary}

\subsection{Gradient shrinking Ricci solitons}\label{subsection:gsrs}
\begin{definition}
Let $(M,g)$ be a complete Riemannian manifold, $\tau>0$ and $f\in C^\infty(M)$ such that 
\begin{equation}\label{def:gsrs}
\ric_g +\hess_g f = \frac{g}{2\tau}.
\end{equation}
Then $(M,g,f)$ is called a gradient shrinking Ricci soliton at scale $\tau>0$. We will refer to such $f$ as a soliton function at scale $\tau$.
\end{definition}

\begin{proposition}\label{prop:soliton_identities}
Let $(M^n,g,f)$ be a gradient shrinking soliton at scale $\tau>0$. Then,
\begin{align}
R+\Delta f &=\frac{n}{2\tau}\label{eqn:traced_soliton_eqn},\\
\tau ( R +|\nabla f|^2)-f &= c,\quad \textrm{for some $c\in \mathbb R$}, \label{eqn:soliton_constant}\\
\ric(\nabla f) &=\frac{1}{2}\nabla R,\\
R&\geq 0.
\end{align}
\end{proposition}
\begin{proof}
This is standard, see for instance \cite{RFpartI}.
\end{proof}
The soliton function $f$ of a gradient shrinking Ricci soliton grows quadratically with respect to the distance from a fixed point.  In fact, the following lemma from  \cite{HasMull} gives precise bounds on the growth of $f$.
\begin{lemma}\label{lemma:soliton_growth}
There is a constant $a(n)$ such that if $(M^n,g,f)$ is a complete gradient shrinking Ricci soliton at scale $\tau>0$ then there is $p\in M$ where $f$  attains a minimum and for any $x\in M$
\begin{equation}\label{eqn:soliton_growth}
\frac{d_g(p,x)^2}{a(n)\tau} - a(n)\leq f(x) +c\leq a(n) \left(\frac{d_g(p,x)^2}{\tau}+1\right).
\end{equation}
Moreover, there is a constant $b(n)$ such that for any two minimum points $p_1,p_2$ of $f$ we have $d_g(p_1,p_2)\leq b(n)\sqrt\tau$.
\end{lemma}

As a consequence, since $(M,g)$ has polynomial volume growth \cite{CarrilloNi}, then 
$$\int_M (4\pi \tau)^{-n/2} e^{-f} d\vol_g <+\infty.$$

Therefore, we can normalize $f$ as follows.
\begin{definition}
If $(M^n,g,f)$ is a gradient shrinking Ricci soliton at scale $\tau>0$ and
\begin{equation}\label{eqn:soliton_normalized}
\int_M (4\pi \tau)^{-n/2} e^{-f} d\vol_g = 1,
\end{equation}
then we will call $(M,g,f)$ a normalized shrinking Ricci soliton at scale $\tau>0$.
\end{definition}

From equation \eqref{def:gsrs} it follows that if $f_0,f_1\in C^\infty(M)$ are such that both $(M,g,f_0)$, $(M,g,f_1)$ are gradient shrinking Ricci solitons, then $L=f_1-f_0$ satisfies $\hess_g L=0$. A consequence of this is the following lemma, which asserts that the constant in \eqref{eqn:soliton_constant} does not depend on the normalized soliton function. For a proof, see \cite{CarrilloNi}.

\begin{lemma}\label{lemma:aux_constant_ngsrs}
Suppose that $f_i\in C^\infty(M)$, $i=0,1$ are such that $(M,g,f_i)$, $i=0,1$ are both normalized gradient shrinking soliton at scale $\tau>0$. Then
\begin{equation}\label{eqn:soliton_cs_same}
\tau(R + |\nabla f_0|^2) -f_0=\tau(R+|\nabla f_1|^2) -f_1.
\end{equation}
\end{lemma}

The lemma below clarifies the nature of the constant in \eqref{eqn:soliton_constant}, expressing it in terms of the $\mathcal W$ functional on a gradient shrinking Ricci soliton. We refer the reader to  \cite{HasMull} for a proof.
\begin{lemma}\label{lemma:soliton_identity_entropy}
Let $(M,g,f)$ be a normalized gradient shrinking Ricci soliton at scale $\tau>0$. Then $\mathcal W(g,f,\tau)$ is well defined and
\begin{equation*}
\tau (R+|\nabla f|^2) - f= -\mathcal W(g,f,\tau).
\end{equation*}
\end{lemma}

In light of Lemma \ref{lemma:aux_constant_ngsrs} and Lemma \ref{lemma:soliton_identity_entropy}, given $(M,g,f)$, a normalized gradient shrinking Ricci soliton at scale $\tau$, the quantity $\mathcal W(g,f,\tau)$ depends only on $g$ and $\tau$, but not $f$. We see below that the scale of a normalized gradient shrinking soliton is also uniquely determined by $g$, unless it is a Gaussian soliton. From this it follows that $\mathcal W(g,f,\tau)$ in fact depends only on $g$.
 
Suppose that $(M,g,f_1)$ and $(M,g,f_2)$ are two normalized gradient shrinking Ricci solitons at scales $\tau_1$ and $\tau_2$ respectively. Then, for $c=\frac{1}{\tau_2}-\frac{1}{\tau_1}$, $\psi=f_2-f_1$ satisfies $\hess w =  c g$. Without loss of generality, we may assume that $c>0$, so $w$ is strictly convex and it attains a unique minimum $p\in M$, and we may also assume that $w(p)=0$. It follows that along any unit speed geodesic $\gamma$ starting at $p$, $\frac{d^2}{dr^2}w(\gamma(r)) = c$, so $w(\gamma(r))=\frac{c}{2} r^2$.  Therefore, the distance function $r(x)=d(x,p)$ satisfies $r^2=\frac{2}{c} w$ and $\hess r^2= 2 g$. From this, we can easily conclude that $(M,g)$ has to be isometric to the Euclidean space $(\mathbb R^n, g_{\mathbb R^n})$. In fact, $(M,g,f_i)$ are Gaussian solitons in different scales  and a direct computation shows that $\mathcal W(g,f_1,\tau_1)=\mathcal W(g,f_2,\tau_2)$.

 Therefore the following definition is natural.

\begin{definition}
Let $(M^n,g,f)$ be a normalized gradient shrinking Ricci soliton at scale $\tau>0$. We define the entropy $\bar\mu(g)$ of $(M,g,f)$ as
\begin{equation}
\bar\mu(g):= \mathcal W(g,f,\tau) = \int_M \left(\tau (|\nabla f|^2 +R) +f -n\right) \frac{e^{-f}}{(4\pi \tau)^{n/2}} d\vol_g.
\end{equation}
\end{definition}

\begin{remark}
From \cite{CarrilloNi} we know that $\bar\mu(g)\leq 0$, if $(M,g,f)$ has bounded curvature - or just a lower bound on Ricci curvature by \cite{Yokota}.
\end{remark}

\begin{lemma}[Lemma 2.3 in \cite{HasMull}]\label{lemma:shrinker_noncollapsing}
Let $(M^n,g,f)$ be a normalized gradient shrinking Ricci soliton at scale $\tau$ and suppose that $f$ attains a minimum at $p\in M$ and that $\bar \mu(g)\geq -\Lambda$. For every $r>0$, there is a constant $\kappa(r)=\kappa(n,\Lambda|r)>0$ such that if $B(x,s)\subset B(p,r\sqrt\tau)$, $0<s\leq \sqrt\tau$, then $\vol(B(x,s)) \geq \kappa(r) s^n$. \end{lemma}

Let $(M,g,f)$ be shrinking Ricci soliton at scale $1$. Since $(M,g)$ is complete with non-negative scalar curvature $|\nabla f|^2 = -R - f +c\leq -f+c$, by Proposition \ref{prop:soliton_identities}, it follows that $\nabla f$ is a complete vector field, since $f$ grows quadratically by \eqref{eqn:soliton_growth}. Hence,  we can define the one parameter family of diffeomorphisms $\varphi_t: M\rightarrow M$, $t\in (-\infty,0)$ so that
\begin{equation}\label{eqn:self_similar_phi}
\begin{aligned}
\frac{d}{dt} \varphi_t &= \frac{1}{|t|} \nabla f \circ\varphi_t\\
\varphi_{-1} &= id_M.
\end{aligned}
\end{equation}
Set $g(t)=|t|\varphi_t^* g$ and $f_t=\varphi_t^*f=f\circ\varphi_t$. Then we know that 
\begin{equation*}
\frac{\partial}{\partial t} g(t)=-2\ric(g(t)),
\end{equation*}
and that $(M,g(t),f_t)$ is a gradient shrinking Ricci soliton at scale $|t|$. We will call $(M,g(t))_{t\in (-\infty,0)}$ the self-similar Ricci flow induced by the gradient shrinking Ricci soliton $(M,g,f)$.

Moreover, \eqref{eqn:self_similar_phi} gives
\begin{equation*}
\frac{d}{dt} \varphi_t = (\varphi_t)_*(\nabla^{g(t)} f_t )\quad\textrm{and}\quad \frac{d}{dt}\varphi_t^{-1} = -\nabla^{g(t)} f_t \circ \varphi_t^{-1},
\end{equation*}
hence
\begin{equation}\label{eqn:ft_grad_evolution}
\frac{\partial f_t}{\partial t}=  |\nabla^{g(t)} f_t|_{g(t)}^2.
\end{equation}

Moreover, by \eqref{eqn:traced_soliton_eqn},
\begin{equation}\label{eqn:sol_conjugate_heat}
\frac{\partial f_t}{\partial t} = -\Delta_{g(t)} f_t +|\nabla^{g(t)} f_t|^2_{g(t)} -R_{g(t)} +\frac{n}{2|t|}, 
\end{equation}
hence $u=(4\pi |t|)^{-n/2} e^{-f}$ is a solution to the conjugate heat equation, by \eqref{eqn:f_che}.

\begin{proposition}\label{prop:backwards_forwards}
Let $(M^n,g(t))_{t\in (-\infty,0)}$ be a smooth complete Ricci flow, with bounded curvature at each time. Suppose that $u=(4\pi |t|)^{-n/2} e^{-f}$ is a solution to the conjugate heat equation and that, for each $t<0$, $(M,g(t),f(\cdot,t))$ is a normalized gradient shrinking Ricci soliton at scale $|t|$. Then
\begin{enumerate}
\item $\frac{\partial}{\partial t} f=|\nabla^{g(t)} f|_{g(t)}^2$, and there is a family of diffeomorphisms $\varphi_t: M\rightarrow M$, satisfying \eqref{eqn:self_similar_phi} such that $g(t)=|t|\varphi_t^* g(-1)$ and $f(x,t)= f(\varphi_t(x),-1)$, for every $x\in M$. In particular, $(M,g(t))_{t\in (-\infty,0)}$ is the selfsimilar Ricci flow associated to the normalized gradient shrinking Ricci soliton at scale $1$ $(M,g(-1),f(\cdot,-1))$, and if $p\in M$ is a critical point of $f(\cdot,-1)$ then it is a critical point of $f(\cdot,t)$, for every $t<0$.
\item $\bar\mu(g(t))$ is constant.
\item $w=|t|(f - \bar\mu(g(-1)))$ satisfies 
$\frac{\partial w}{\partial t}= \Delta w -\frac{n}{2}$.
\end{enumerate}
\end{proposition}
\begin{proof}
Assertion 1 follows by the evolution equation of $f$
$$\frac{\partial f}{\partial t}= -\Delta f +|\nabla f|^2 -R +\frac{n}{2|t|}$$
by substituting the traced soliton equation \eqref{eqn:traced_soliton_eqn}, of Proposition \ref{prop:soliton_identities}. Then, pulling back by the flow of $-\nabla f$ and rescaling we obtain a stationary flow $\tilde g(s)=g(-1)$, from which the remaining of Assertion 1 follows easily.

Assertion 2 then follows by the scale and diffeomorphism invariance of the entropy $\mathcal W$.

To prove Assertion 3, we  use \eqref{eqn:traced_soliton_eqn} and that $|t|(R+|\nabla f|^2)-f=-\bar\mu(g(t))=-\bar\mu(g(-1))$ to compute
\begin{equation*}
\frac{\partial w}{\partial t}=|t| |\nabla f|^2 - f+\bar\mu(g(-1))=-|t| R =|t|\Delta f- \frac{n}{2}=\Delta w -\frac{n}{2}.
\end{equation*}

\end{proof}

\section{A priori assumptions on Ricci flows}\label{sec:apriori}

Let $(M^n,g(t))_{t\in I}$, $t_{\textrm{sup}}=\sup I$, be a smooth complete Ricci flow.
\begin{itemize}
\item[(RF1)] $(M,g(t))_{t\in I}$ satisfies
\begin{equation*}
\sup_M |\riem|(\cdot,t) \leq \frac{C_I}{t_{\textrm{sup}}-t}\end{equation*}
\item[(RF2)] If $M$ is compact, and for every $t\in I$, $t<t_{\textrm{sup}}$,  $\nu\left(g(t),|t_{\textrm{sup}}-t|\right) \geq -\Lambda$.
\item[(RF3)] For every $x,y\in M$ and $s,t\in I$, $s<t$ the conjugate heat kernel $$u_{(x,t)}(y,s)=(4\pi |t-s|)^{-n/2} e^{-f_{(x,t)}(y,s)}$$ satisfies
\begin{equation*}
f_{(x,t)}(y,s) \geq \frac{d_{g(s)}(x,y)^2}{H|t-s|} - H.
\end{equation*}
Equivalently,
\begin{equation*}
u_{(x,t)}(y,s)\leq C_1(4\pi |t-s|)^{-n/2} e^{-\frac{d_{g(s)}(x,y)^2}{C_1|t-s|}}
\end{equation*}
\item[(RF4)] For every $x,y\in M$ and $s,t\in I$, $s<t$ the conjugate heat kernel $$u_{(x,t)}(y,s)=(4\pi |t-s|)^{-n/2} e^{-f_{(x,t)}(y,s)}$$ satisfies
\begin{equation*}
f_{(x,t)}(y,s) \leq H\left(\frac{d_{g(s)}(x,y)^2}{|t-s|} + 1\right).
\end{equation*}
Equivalently,
\begin{equation*}
u_{(x,t)}(y,s)\geq C_2(4\pi |t-s|)^{-n/2} e^{-\frac{d_{g(s)}(x,y)^2}{C_2|t-s|}}
\end{equation*}
\item[(RF5)] For every $x,y\in M$ and $s,t\in I$, $s<t$,  
\begin{equation*}
d_{g(s)}(x,y) \leq d_{g(t)}(x,y) + K\sqrt{|s|}.
\end{equation*}
\end{itemize}

\begin{definition}[Conjugate heat flows]
 Let $(M^n,g(t))_{t\in I}$, $I\subset (-\infty,0)$, be a smooth complete Ricci flow and let $u=(4\pi|t|)^{-n/2} e^{-f}\in C^\infty(M\times I)$ be a positive solution to the conjugate heat equation so that for every $t\in I$, $$\int_M u(\cdot,t)d\vol_{g(t)}=1.$$ 
 For every $t\in I$, we will denote by $\nu_{f,t}$ the probability measure on $M_t:=M\times \{t\}$ with $d\nu_{f,t}=(4\pi|t|)^{-n/2}e^{-f(\cdot,t)} d\vol_{g(t)}$. The $1$-parameter family $\nu_f=(\nu_{f,t})_{t\in I}$ is called a conjugate heat flow on $(M,g(t))_{t\in I}$.
 \end{definition}
  
 We will say that a conjugate heat flow  $(\nu_{f,t})_{t\in I}$  satisfies
\begin{itemize}
\item[(CHF1)] if for every $(x,t)\in M\times I$
\begin{equation}\label{CHF1}
\frac{d_{g(t)}(p,x)^2}{H |t|} - H\leq f(x,t).
 \end{equation}
\item[(CHF2)] if for every $(x,t)\in M\times I$
\begin{equation}\label{CHF2}
f(x,t) \leq H\left(\frac{d_{g(t)}(p,x)^2}{|t|}+1 \right).
 \end{equation}
\end{itemize}

\begin{definition}
Consider a sequence $I_j\subset \mathbb R$ of intervals that converge to the interval $I_\infty\subset \mathbb R$ in the sense that for any closed interval $J\subset I_\infty$, $J\subset I_j$ for large $j$.
\begin{enumerate}
\item We say that a pointed sequence $(M_j,g_j(t),p_j)_{t\in I_j}$ of Ricci flows converges to a pointed Ricci flow $(M_\infty,g_\infty(t),p_\infty)_{t\in I_\infty}$ if for some, and thus for any,  $t_0\in I_\infty$ the following holds: for every $R<+\infty$ there are smooth maps $F_j: B(p_\infty,t_0,R)\rightarrow M_j$, diffeomorphisms onto their image, such that $F_j(p_\infty) = p_j$ and $F_j^* g_j$ converges to $g_\infty$ in the smooth and uniform in compact subsets of $B(p_\infty,t_0,R)\times I_\infty$ topology.
\item Given a convergent sequence $(M_j,g_j(t),p_j)_{t\in I_j}\rightarrow (M_\infty,g_\infty(t),p_\infty)_{t\in I_\infty}$, we will say that $x_j\in M_j$ converges to $x_\infty\in M_\infty$, and denote by $x_j\rightarrow x_\infty$, if there is $R<+\infty$ and $F_j :B(p_\infty,t_0,R)\rightarrow M_j$ as in (1) such that $F_j^{-1}(x_j)\rightarrow x_\infty$.
\item Given a convergent sequence $(M_j,g_j(t),p_j)_{t\in I_j}\rightarrow (M_\infty,g_\infty(t),p_\infty)_{t\in I_\infty}$, we will say that a sequence of functions $f_j\in C^\infty(M_j \times I_j)$ converges to $f_\infty \in C^\infty(M_\infty\times I_\infty)$  if for any $R<+\infty$ there is $F_j :B(p_\infty,t_0,R)\rightarrow M_j$ as in (1) such that $F_j^* f_j$ converges smoothly uniformly in compact subsets of $B(p_\infty,t_0,R)\times I_\infty$ to $f_\infty$.
\item Given a convergent sequence $(M_j,g_j(t),p_j)_{t\in I_j}\rightarrow (M_\infty,g_\infty(t),p_\infty)_{t\in I_\infty}$, we will say that a sequence $(\nu_{j,t})_{t\in I_j}$ of conjugate heat flows converges to a conjugate heat flow $(\nu_{\infty,t})_{t\in I_\infty}$ if the associated solutions $u_j\in C^\infty(M_j \times I_j)$ to the conjugate heat equation converge to $u_\infty\in C^\infty(M_\infty\times I_\infty)$, associated to the conjugate heat flow $\nu_{\infty,t}$.
\end{enumerate}
\end{definition}

\begin{proposition}[Compactness of Ricci flows paired with conjugate heat flows]\label{prop:compactness_rf}
Let $(M^n_j,g_j(t),p_j)_{t\in I}$, $\sup I = 0$,  be a sequence of smooth complete Ricci flows satisfying (RF1). If $M_j$ are compact, suppose that $(M_j,g_j(t),p_j)_{t\in I}$ satisfies (RF2), while if $(M_j,g_j(t))_{t\in (-\infty,0)}$ are induced by gradient shrinking Ricci solitons $(M_j,\bar g_j,\bar f_j)$ suppose that $\bar\mu(\bar g_j)\geq -\Lambda$, for all $j$. Moreover, suppose that  there is sequence $(\nu_{j,t})_{t\in I}$ of conjugate heat flows that satisfy (CHF1) with respect to $p_j\in M_j$. Then there is a pointed smooth complete Ricci flow $(M_\infty,g_\infty(t),p_\infty)_{t\in I}$ satisfying (RF1), and a conjugate heat flow $(\nu_{\infty,t})_{t\in I}$ on $(M_\infty,g_\infty(t))_{t\in I}$, satisfying (CHF1) with respect to $p_\infty$, such that a subsequence of $(M_j,g_j(t),p_j)_{t\in I}$ converges to $(M_\infty,g_\infty(t),p_\infty)_{t\in I}$ and $(\nu_{j,t})_{t\in I}$ converges to $(\nu_{\infty,t})_{t\in I}$.
\end{proposition}

\begin{proof}
By  Theorem \ref{thm:entropy_noncollapsing} and Lemma \ref{lemma:shrinker_noncollapsing}, the standard smooth compactness theory for sequences of Ricci flows we can assume, by passing to a subsequence, that $(M_j,g_j(t),p_j)_{t\in I}$  smoothly converges to a complete Type I Ricci flow $(M_\infty,g_\infty(t),p_\infty)_{t\in I}$ with constant $C_I$.

Since the conjugate heat flows $(\nu_{j,t})_{t\in I}$ uniformly satisfy (CHF1) with respect to $p_j\in M_j$, it follows that the associated solutions $u_j>0$ to the conjugate heat equation satisfy uniform Gaussian $C^0$ bounds. Thus, by parabolic regularity we can also assume, passing to a further subsequence if necessary, that $u_j$ converges to a smooth solution $u_\infty$ of the conjugate heat equation on $M_\infty \times I$. The Gaussian bounds (CHF1) on $u_j$ also give that 
$$\int_{M_\infty} u_\infty (\cdot,t) d\vol_{g_\infty(t)}=1$$
for every $t\in I$, so $(\nu_{\infty,t})_{t\in I}$ with $d\nu_{\infty,t} = u_\infty(\cdot,t) d\vol_{g_\infty(t)}$ is a conjugate heat flow on $(M_\infty,g_\infty(t))_{t\in I}$,  that satisfies (CHF1) with respect to $p_\infty$.
\end{proof}

Very often we will need to convert integral estimates with respect to a conjugate heat kernel measure $\nu_{(p_1,t_1),t}$ to estimates with respect to a conjugate heat kernel measure $\nu_{(p_2,t_2),t}$ based at a space time point $(p_2,t_2)\not = (p_1,t_1)$. Unfortunately, these two probability measures do not satisfy an inequality of the form $\nu_{(p_2,t_2),t}\leq C\nu_{(p_1,t_1),t}$, even in the static Ricci flow in Euclidean space. The following lemma will be crucial in such situations. It will allow us to switch basepoints in our estimates, provided we can enhance a given integral estimate to an estimate that involves an additional weight of the from $e^{\alpha f}$, for an appropriate $\alpha\in (0,1)$. This is inspired by Bamler \cite{Bamler3}, where this idea is exploited to deal with general Ricci flows. Here, for Ricci flows satisfying (RF3) and (RF4), the upper and lower Gaussian bounds for the conjugate heat kernels provide much more flexibility, and essentially the following lemma is a direct consequence of these bounds.

 \begin{lemma}[Change of basepoint]\label{lemma:compare_kernels}
Let $(M^n,g(t))_{t\in [-1,0]}$ be a smooth complete Ricci flow satisfying (RF3) and (RF4), and points $p_1,p_2\in M$. Let $u_{(p_i,t_i)}(x,t)= (4\pi |t_1-t|)^{-n/2} e^{-f_i(x,t)}$, $i=1,2$,
be the conjugate heat kernels based at $(p_i,t_i)$. If $\frac{2C_1 - \beta C_2}{2C_1} \leq \alpha<1$, then for every $-1\leq t<0$ such that 
\begin{equation} \label{eqn:ratio}
\beta|t|\leq |t_i - t| \leq |t|,
\end{equation}
 for every $i=1,2$, we have that
\begin{equation*}
 u_{(p_1,t_1)}(x,t)\leq C (n,H|\beta) e^{\frac{(d_{g(t)}(p_1,p_2))^2}{C_1|t|}} e^{\alpha f_2} u_{(p_2,t_2)}(x,t).
\end{equation*}
\end{lemma}

\begin{proof}
Recall that assumptions (RF3) and (RF4) imply that for every $x\in M$ and $t<t_i$, 
\begin{equation}\label{eqn:p_i_t_i_gaussian}
 \frac{C_2e^{-\frac{(d_{g(t)}(p_i,x))^2}{C_2|t_i-t|}}}{(4\pi |t_i-t|)^{n/2}}
\leq \frac{e^{-f_i(x,t)}}{(4\pi |t_i-t|)^{n/2}} \leq \frac{C_1e^{-\frac{(d_{g(t)}(p_i,x))^2}{C_1|t_i-t|}}}{(4\pi |t_i-t|)^{n/2}}
\end{equation}

It follows that, for 
$$\frac{1-\alpha}{C_2}\leq \frac{\beta}{2C_1} \Longleftrightarrow  \frac{2C_1 - \beta C_2}{2C_1} \leq \alpha,$$
we obtain
\begin{equation}\label{eqn:lower_bound_weight}
\begin{aligned}
e^{\alpha f_2} u_{(p_2,t_2)}(x,t) &=(4\pi |t_2-t|)^{-n/2} e^{-(1-\alpha)f_2(x,t)}\\
&\geq \frac{C_2^{1-\alpha}}{(4\pi |t_2-t|)^{n/2}} e^{-\frac{(1-\alpha)(d_{g(t)}(p_2,x))^2}{C_2 |t_2-t|}} \\
&\geq \frac{C_2}{(4\pi |t_2-t|)^{n/2}} e^{-\frac{\beta(d_{g(t)}(p_2,x))^2}{2C_1 |t_2-t|}}.
\end{aligned}
\end{equation}

Now, the triangle inequality gives
$(d_{g(t)}(p_2,x))^2\leq 2(d_{g(t)}(p_1,p_2))^2 + 2(d_{g(t)} (p_1,x))^2$, hence
\begin{equation}\label{eqn:exp_product}
e^{-\frac{\beta (d_{g(t)}(p_2,x))^2}{2C_1 |t_2-t|}} \geq e^{-\frac{\beta (d_{g(t)}(p_1,x))^2}{C_1 |t_2-t|}} e^{-\frac{\beta(d_{g(t)}(p_1,p_2))^2}{C_1|t_2-t|}}.
\end{equation}

To estimate the first factor on the right-hand side of \eqref{eqn:exp_product}, note that \eqref{eqn:ratio} implies
\begin{equation}
\frac{\beta}{|t_2-t|} \leq \frac{1}{|t|}  \leq \frac{1}{|t_1-t|}
\end{equation}
therefore
\begin{equation}\label{eqn:exp_d_p1x}
e^{-\frac{\beta(d_{g(t)}(p_1,x))^2}{C_1 |t_2-t|}} \geq e^{-\frac{(d_{g(t)}(p_1,x))^2}{C_1 |t_1-t|}}.
\end{equation}
Similarly, we can estimate the second factor of \eqref{eqn:exp_product} as
\begin{equation}\label{eqn:exp_d_p1p2}
e^{-\frac{\beta (d_{g(t)}(p_1,p_2))^2}{C_1|t_2-t|}} \geq e^{-\frac{(d_{g(t)}(p_1,p_2))^2}{C_1 |t|}}.
\end{equation}
Finally, using once more \eqref{eqn:p_i_t_i_gaussian}, \eqref{eqn:lower_bound_weight} becomes 
\begin{equation}
\begin{aligned}
e^{\alpha f_2} u_{(p_2,t_2)}(x,t) &\geq \frac{C_2}{(4\pi |t_2-t|)^{n/2}} e^{-\frac{(d_{g(t)}(p_1,x))^2}{C_1 |t_1-t|}} e^{-\frac{(d_{g(t)}(p_1,p_2))^2}{C_1 |t|}}\\
&\geq \frac{C_2}{C_1}  \beta^{n/2} e^{-\frac{(d_{g(t)}(p_1,p_2))^2}{C_1 |t|}}u_{(p_1,t_1)} (x,t),
\end{aligned}
\end{equation}
which proves the lemma.
\end{proof}

The following lemma is again a direct consequence of the Gaussian upper bounds for the conjugate heat kernel (RF3). 

\begin{lemma}\label{lemma:int_ker_bounds}
Let $(M^n,g(t))_{t\in I}$, $\max I=0$, be smooth complete  Ricci flow satisfying 
(RF1) and (RF3), and let $\alpha\in [0,1)$ and $m\geq 0$. Then, for every $t\in I$, $t<0$,
\begin{align}
\int_M e^{\alpha f} d\nu_{(p,0),t}&\leq F_1(n,C_I,C_1|\alpha),\label{eqn:eaf}\\
\int_M \left(\frac{d_{g(t)}(p,\cdot)}{|t|^{1/2}}\right)^m d\nu_{(p,0),t}&\leq F_2(n,C_I,C_1|m).
\end{align}
Moreover, for every $\varepsilon>0$ there is $r=r(C_I,C_1,\alpha,m)$ such that
\begin{equation}
\int_{M\setminus B(p,t,r\sqrt{|t|})} e^{\alpha f(\cdot,t)}d\nu_{(p,0),t} + \int_{M\setminus B(p,t,r\sqrt{|t|})}\left(\frac{d_{g(t)}(p,\cdot)}{|t|^{1/2}}\right)^m d\nu_{(p,0),t}<\varepsilon.
\end{equation}
\end{lemma}
\begin{proof}
By rescaling it suffices to prove the estimates assuming that $t=-1\in I$. Then, by the conjugate heat kernel bound (RF3), the functions $e^{\alpha f(\cdot,-1)} u_{(p,0)}(\cdot,-1)$ and $
(d_{g(-1)}(p,\cdot))^m u_{(p,0)}(\cdot,-1)$ are both bounded above by $Ce^{-\frac{(d_{g(-1)}(p,\cdot))^2}{C}}$, where $C$ depends on $C_1$ and $\alpha$ or $m$ accordingly. On the other hand, it follows by  assumption (RF1) ithe volume of balls of radius $r$ grow at most exponentially. This suffices to prove the result.
\end{proof}
\begin{remark}
Bamler-Zhang show in \cite{BZ2} that the volume of balls of radius $r>0$ still grows at most exponentially under a scalar curvature and entropy bound for the Ricci flow, although one needs to let the flow run for a bit, see Lemma 2.1 in \cite{BZ2}. Thus, a variant of Lemma \ref{lemma:int_ker_bounds} under a weaker scalar curvature bound does hold and would be sufficient for our purposes. Since in our setting we will have the (RF1) bound at our disposal, we will just use Lemma \ref{lemma:int_ker_bounds} for the sake of simplicity.
\end{remark}

Finally,  Proposition \ref{prop:RF35} below shows that any Ricci flow with a Type I scalar curvature bound and a lower entropy bound essentially satisfies assumptions (RF3-5). 

\begin{proposition}\label{prop:RF35}
Let $(M^n,g(t))_{t\in [-T, 0)}$ be a smooth compact Ricci flow satisfying 
\begin{align*}
\sup_M |R|(\cdot,t) &\leq \frac{C_I}{|t|},\\
\nu(g(-T), 2T) &\geq -\Lambda.
\end{align*}
Then, $(M,g(t))_{t\in [-\frac{T}{4},0)}$ satisfies (RF2-5), for some constants $\Lambda,C_1,C_2,H,K$ depending only on $n,C_I$ and $\Lambda$.
\end{proposition}
\begin{proof}
By a time shift and rescaling, it suffices to prove the proposition assuming that $(M,g(t))_{t\in [-2,0)}$ satisfies
\begin{align*}
\sup_M |R|(\cdot,t)&\leq \frac{C_I}{|t|},\\
\nu(g(-2),4)&\geq -\Lambda.
\end{align*}
By the definition of $\nu$ we know that
\begin{equation}\label{eqn:nu}
\nu(g(t),|t|)\geq \nu(g(-2),2) \geq \nu(g(-2),4)\geq -\Lambda
\end{equation}
for every $t\in [-2,0)$, so $(M,g(t))_{t\in [-2, 0 )}$ satisfies (RF2).

The Gaussian upper and lower conjugate heat kernel bounds (RF3) and (RF4) follow directly from Lemma 3.1 in \cite{Hallgren_scalar}.

The lower distance distortion estimate (RF5) follows from Theorem 1.1 in \cite{BZ2} as follows. Let $t_i=-2^{-i+1}$, $i\geq 0$ integer. Then $t_{i+1}-t_i =2^{-i}$ and the rescaled flows $(M,g_i(t))_{t\in [0,1]}$, defined as
$$g_i(t) = 2^i g( 2^{-i} t +t_i)$$
satisfy 
\begin{align*}
\sup_M R_{g_i}(\cdot,t) &= 2^{-i} \sup_M R_g(\cdot, 2^{-i} t +t_i) \leq 2^{-i} \frac{C_I}{|t_{i+1}|} \leq C_I,\\
\nu(g_i(0),2)&=\nu(2^i g(t_i),2) =\nu(g(t_i),2^{-i+1})=\nu(g(t_i),|t_i|)\geq -\Lambda.
\end{align*}
Therefore, by Theorem 1.1 in \cite{BZ2}, for any $x,y\in M$ and $t\in [t_i,t_{i+1}]$
\begin{equation}
\begin{aligned}
d_{g(t_i)}(x,y)&= 2^{-i/2} d_{g_i(0)} (x,y)\\
&\leq 2^{-i/2}( d_{g_i( 2^i(t-t_i))} (x,y) + C\sqrt{2^i(t-t_i)}) \\
&=d_{g(t)}(x,y) + C \sqrt{t-t_i}.
\end{aligned}
\end{equation}
and similarly
$$d_{g(t)}(x,y) \leq d_{g(t_{i+1})}(x,y) + C\sqrt{t_{i+1}-t}.$$
In particular, 
$$d_{g(t_i)}(x,y) \leq d_{g(t_{i+1})}(x,y) + C\sqrt{t_{i+1}-t_i}.$$
It follows that for any $-2\leq s<t<0$, with $s\in (t_i,t_{i+1}]$, $t\in [t_{i+1+m},t_{i+2+m})$
\begin{equation}
\begin{aligned}
d_{g(s)}(x,y)&\leq d_{g(t_{i+1})}(x,y) + C\sqrt{t_{i+1}-s}\\
&\leq d_{g(t_{i+2})}(x,y) + C \sqrt{t_{i+2}-t_{i+1}} + C\sqrt{t_{i+1}-s}\\
&\leq d_{g(t_{i+1+m})}(x,y) + C \sum_{k=1}^m \sqrt{t_{i+1+k}-t_{i+k}} + C\sqrt{t_{i+1}-s}
 \\
&\leq d_{g(t)}(x,y) + C\left(\sqrt{t-t_{i+1+m}}+  \sum_{k=1}^m \sqrt{t_{i+k+1}-t_{i+k}} + \sqrt{t_{i+1}-s}\right)\\
&\leq d_{g(t)}(x,y) + C\left( 2^{-\frac{i+2+m}{2}}+ \sum_{k=1}^m 2^{-\frac{i+k}{2}} + 2^{-\frac{i}{2}}\right)\\
&\leq d_{g(t)}(x,y)  +C2^{-\frac{i}{2}}\left( 2^{-\frac{m+2}{2}}+\sum_{k=1}^m 2^{-\frac{k}{2}}+1\right)\\
&\leq d_{g(t)}(x,y) + K \sqrt{|t_{i+1}|}\leq d_{g(t)}(x,y) +K\sqrt{|s|}.
\end{aligned}
\end{equation}

\end{proof}

\section{Self-similar Ricci flows}\label{sec:selfsimilar}

\begin{definition}\label{def:k_selfsimilar}
Let $(M^n,g(t))_{t\in (-\infty,0)}$ be a Ricci flow induced by a gradient shrinking Ricci soliton. 
\begin{enumerate}
\item We define the spine $\mathcal S$ of $(M,g(t))_{t\in (-\infty,0)}$ so that $\nu_f=(\nu_{f,t})_{t\in (-\infty,0)}\in \mathcal S$ if and only if $(\nu_{f,t})_{t\in (-\infty,0)}$ is a conjugate heat flow with $d\nu_{f,t}=(4\pi |t|)^{-n/2}e^{-f(\cdot,t)} d\vol_{g(t)}$ such that  $(M,g(t),f(\cdot,t))$ is a normalized gradient shrinking Ricci soliton at scale $|t|$, for each $t<0$.
\item We define the point-spine $\mathcal S_{\textrm{point}}$ of $(M,g(t))_{t\in (-\infty,0)}$ as
$$\mathcal S_{\textrm{point}} =\{x\in M, \textrm{there is $\nu_f\in\mathcal S$ and $f(\cdot,t)$ attains a minimum at $x$, for every $t<0$}\}.$$
\end{enumerate}
\end{definition}

\begin{proposition}[Spine structure] \label{prop:spine_structure}
Let $(M^n,g(t))_{t\in (-\infty,0)}$ be the  Ricci flow induced by a gradient shrinking Ricci soliton. Then there is a maximal $0\leq k\leq n$ such that $(M,g(t))=(M'\times \mathbb R^k, g'(t)\oplus g_{\mathbb R^k})$, for some smooth complete Ricci flow $(M',g'(t))_{t\in (-\infty,0)}$. Moreover,
\begin{enumerate}
\item $(M',g'(t))_{t\in(-\infty,0)}$ is the Ricci flow induced by a gradient shrinking Ricci soliton with spine $\mathcal S'$ consisting of a unique conjugate heat flow $\nu_{f'}$. Moreover, for every $\nu_f\in\mathcal S$, there is a unique $a_f\in\mathbb R^k$ such that
$$f((z,b),t)=\frac{|b-a_f|^2}{4|t|} + f'(z,t)$$
for every $t<0$ and $(z,b)\in M'\times \mathbb R^k$, and all elements of $\mathcal S$ are of this form.

\item If $\nu_f\in\mathcal S$ satisfies (CHF2) for some $H<+\infty$ with respect to $p\in M$, then $p=(q,a_f)$ for some $q\in M'$. Moreover, for any $p\in\mathcal S_{\textrm{point}}$, there is a unique $\nu_f\in\mathcal S$ such that $f(\cdot,t)$ attain a minimum at $p$.

\item There is a non-empty $\mathcal K\subset M'$ such that $\mathcal S_{\textrm{point}}= \mathcal K\times \mathbb R^k$ and 
$$\diam_{g'(t)}(\mathcal K) \leq A(n)\sqrt{|t|},$$
for every $t<0$.
\end{enumerate}
\end{proposition}
\begin{proof}
Let $0\leq k \leq n$ be the largest integer with the property that there is a Ricci flow $(M'\times \mathbb R^k, g'(t)\oplus g_{\mathbb R^k})_{t<0}$ which is isometric to $(M,g(t))_{t<0}$. In what follows, we will identify $M=M'\times \mathbb R^k$ and $g(t)=g'(t)\oplus g_{\mathbb R^k}$.

\begin{enumerate}
\item
Let $\nu_f \in \mathcal S$, which exists since $(M,g(t))_{t\in (-\infty,0)}$ is induced by a gradient shrinking Ricci soliton. By Proposition \ref{prop:backwards_forwards} we have that $\frac{\partial f}{\partial t} = |\nabla f|^2$ and $f(x,t)=f(\varphi_t(x),-1)$ for some $1$-parameter family of diffeomorphisms that fixes the critical points of $f(\cdot,-1)$, since it satisfies $\frac{d}{dt}\varphi_t = -\frac{1}{|t|}\nabla^{g(-1)} f(\varphi_t(\cdot),-1)$. 
In particular, any minimum point of $f(\cdot,-1)$ is also a minimum point of   $f(\cdot,t)$, for any $t<0$.

Thus, let $p=(q,a)\in M=M'\times \mathbb R^k$ be a minimum point of $f(\cdot,t)$, for every $t<0$. Define $f'\in C^\infty(M' \times (-\infty,0))$ as 
$$f'(z,t)=f((z,a),t),$$
and let $x^i$, $i=1,\ldots,k$ denote the coordinate functions that correspond to the $\mathbb R^k$ factor. 

Since $g(t)=g' (t)\oplus g_{\mathbb R^k}$, we have that $\nabla x^i=0$ on $M\times (-\infty,0)$. Moreover, since for every $t<0$, $(M,g(t),f(\cdot,t))$ we can easily check that the function $\langle\nabla f,\nabla x^i\rangle$ is constant in each level set of the map $(x^1,\ldots,x^k): M\rightarrow \mathbb R^k$. Therefore, by $\nabla f(\cdot,t)|_{(q,a)}=0$, it follows that $\nabla f(\cdot,t)|_{M'\times \{a\}}$ is tangent to $M'\times \{a\}$. Therefore, $\nabla^{g'(t)} f' = \nabla f(\cdot,t)$ on $M'\times \{a\}$, once we identify $M'$ with $M'\times \{a\}$. Consequently, for any two vector fields $V,W$ of $M$ tangent  to $M'\times \{a\}$ we compute
\begin{align*}
\left(\ric_{g'(t)} + \hess_{g'(t)} f' (\cdot,t)\right)(V,W)&=\left(\ric_{g(t)} + \hess_{g(t)} f(\cdot,t)\right)(V,W)\\
&=\frac{g(t)(V,W)}{2|t|}\\
&=\frac{g'(t)(V,W)}{2|t|},
\end{align*}
which proves that $(M',g'(t),f'(\cdot,t))$ is a gradient shrinking Ricci soliton at scale $|t|$.

Moreover, integrating the soliton equation \eqref{def:gsrs} along  geodesics of the form $$t\mapsto (q,a+tb),$$ for any $b\in\mathbb R^k$, we obtain that
\begin{equation}\label{eqn:sol_fcn_splitting}
f((z,b),t)=\frac{|b-a|^2}{4|t|} +f'(z,t).
\end{equation}
That $(M',g'(t),f'(\cdot,t))$ is normalized follows from that $(M,g(t),f(t))$ is normalized.

For the uniqueness of $f'$, recall that $f'(\cdot,t)$ is uniquely determined, up to adding a constant and linear function. Since $(M',g'(t),f'(\cdot,t))$ is normalized and does not split any more Euclidean factors, due to the maximality of $k$, this shows that $f'(\cdot,t)$ is in fact unique, depending only on $(M',g'(t))$. The uniqueness of $a$, in terms of the data $f$ and $g$ then follows from \eqref{eqn:sol_fcn_splitting}.
\item

 Now suppose that $\nu_f\in\mathcal S$ satisfies (CHF2) with respect to $p=(q,a')\in M'\times \mathbb R^k=M$, namely there is a constant $H<+\infty$ such that for every $t\in (-\infty,0)$
\begin{equation}\label{eqn:C_rep}
f(x,t) \leq H\left( \frac{d_{g(t)}(p,x)^2}{|t|} + 1\right).
\end{equation}
Therefore, setting $p=(z,b)=(q,a')$ in \eqref{eqn:sol_fcn_splitting} and \eqref{eqn:C_rep}, we obtain
\begin{equation}\label{eqn:sol_spine_bound}
f((q,a'),t)=\frac{|a'-a|^2}{4|t|}+f'(q,t) \leq H
\end{equation}
On the other hand, since for every $t<0$, $(M',g'(t),f'(\cdot,t))$ is a normalized gradient shrinking Ricci soliton at scale $|t|$, we know from Proposition \ref{prop:soliton_identities} and Proposition \ref{prop:backwards_forwards} that for every $t<0$
\begin{align*}
-f' &= -\mathcal W(g'(t),f'(\cdot,t),|t|) - |t|(R_{g'(t)}+|\nabla f'|^2(\cdot,t))\\
&\leq - \mathcal W(g'(t),f'(\cdot,t),|t|)\\
&=- \mathcal W(g'(-1),f'(\cdot,-1),1).
\end{align*}
Therefore, by \eqref{eqn:sol_spine_bound}, we obtain that
$$\sup_{t<0}\frac{|a'-a|^2}{4|t|} <+\infty,$$
which is only possible if $a'=a$.

If $p\in \mathcal S_{\textrm{point}}$, there is $\nu_f \in\mathcal S$ such that $f(\cdot,t)$ attains a minimum at $p$. By Lemma \ref{lemma:soliton_growth}, $\nu_f$ satisfies (CHF2) with the constant $a(n)$ of that lemma, with respect to $p$. It follows that $p=(q,a_f)$, where $a_f\in \mathbb R$ is the constant that determines $f$ via Assertion 1. In particular $a_f$ is determined uniquely from $p\in\mathcal S_{\textrm{point}}$, which proves Assertion 2.

\item

Since $\mathcal S$ is invariant under translation along the $\mathbb R^k$ factor we can easily verify that $\mathcal S_{\textrm{point}}$ is also invariant under translation along $\mathbb R^k$, hence $\mathcal S_{\textrm{point}}=\mathcal K\times\mathbb R^k$ for some $\mathcal K\subset M'$. By Lemma \ref{lemma:soliton_growth},  every soliton function attains a minimum, hence $\mathcal K$ is non-empty.

To estimate the diameter of $\mathcal K$, let $y_i=(z_i,0)\in\mathcal K\times \mathbb R^k$, $i=1,2$, be two points in $\mathcal S_{\textrm{point}}$ such that $d_{g(t)}(y_1,y_2)=d_{g'(t)}(z_1,z_2)$. 

By the definition of $\mathcal S_{\textrm{point}}$ there is are functions $f_i\in C^\infty(M\times (-\infty,0))$, $i=1,2$, such that $(M,g(t),f_i(\cdot,t))$ is a normalized gradient shrinking Ricci soliton at scale $|t|$ and $f_i(\cdot,t)$ attains a minimum at $y_i$, for every $t<0$. By Assertion 1 of this proposition, there are $a_i\in\mathbb R^k$, $i=1,2$ such that
$$f_i((z,b),t)=\frac{|b-a_i|^2}{4|t|} + f'(z,t)$$
and $(M',g'(t),f'(\cdot,t))$ is normalized gradient shrinking soliton at scale $|t|$, and $f'$ is uniquely determined by $(M',g'(t))$.

Since each $f_i$ attains a minimum at $y_i=(z_i,0)$, it follows that $a_i=0$ and both $z_i$ are points where $f'$ attains a minimum. Thus, by Lemma \ref{lemma:soliton_growth},
$$d_{g'(t)}(z_1,z_2)\leq b(n) \sqrt{|t|},$$
which suffices to prove the assertion.
\end{enumerate}
\end{proof}

\subsection{Convergence and selfsimilarity}

\begin{proposition}[Theorem 1.6 in \cite{ManteMull}]\label{prop:entropy_convergence}
Let $(M^n_j,g_j(t),p_j))_{t\in [A_j,0]}$, $A_j\rightarrow -\infty$, be a sequence of smooth compact Ricci flows satisfying (RF3). Suppose that it converges to the Ricci flow $(M,g(t),p)_{t\in (-\infty,0)}$ induced by a gradient shrinking Ricci soliton, and that the conjugate heat kernels $\nu_{(p_j,0)}$ converge to a conjugate heat flow $\nu_f \in\mathcal S$. Then, for every $t<0$, 
$$\mathcal W_{p_j}(|t|)\rightarrow \bar\mu(g(-1)).$$
\end{proposition}
\begin{proof}
We give a brief account of the proof for the sake of completeness.

First of all, if $d\nu_{(p_j,0),t}=(4\pi |t|)^{-n/2} e^{-f_j(\cdot,t)} d\vol_{g_j(t)}$, integrating by parts we obtain
\begin{align*}
\mathcal W_{p_j}(|t|)&=\int_{M_j} \left(|t|\left(R_{g_j(t)}+|\nabla^{g_j(t)} f_j(\cdot,t)|_{g_j(t)}^2 \right)+f_j(\cdot,t) -n\right) d\nu_{(p_j,0),t} \\
&=\int_{M_j} \left(|t|\left(R_{g_j(t)} +2\Delta_{g_j(t)} f_j(\cdot,t)-|\nabla^{g_j(t)} f_j(\cdot,t)|_{g_j(t)}^2\right) +f_j(\cdot,t)-n \right) d\nu_{(p_j,0),t} \leq 0.
\end{align*}
and the integrand in the second equation above is non-positive by Perelman's Harnack inequality \cite{Perelman1}.

Thus, for every $r>0$,
\begin{equation}\label{eqn:limsupW}
\begin{aligned}
&\limsup_j \mathcal W_{p_j}(|t|) \leq\\
&\leq \lim_j \int_{B(p_j,t,r)} \left(|t|\left(R_{g_j(t)} +2\Delta_{g_j(t)} f_j(\cdot,t)-|\nabla^{g_j(t)} f_j(\cdot,t)|_{g_j(t)}^2\right) +f_j(\cdot,t)-n \right) d\nu_{(p_j,0),t} \\
&=\int_{B(p,t,r)}\left(|t|\left(R_{g(t)} +2\Delta_{g(t)} f(\cdot,t)-|\nabla^{g(t)} f(\cdot,t)|_{g(t)}^2\right) +f(\cdot,t)-n \right) d\nu_{f,t}.
\end{aligned}
\end{equation}
By Lemma 4.2 in \cite{ManteMull}
$$\bar \mu(g(t))=\mathcal W(g(t),f(\cdot,t),|t|)= \int_M \left(|t|\left(R_{g(t)} +2\Delta_{g(t)} f(\cdot,t)-|\nabla^{g(t)} f(\cdot,t)|_{g(t)}^2\right) +f(\cdot,t)-n \right) d\nu_{f,t} \leq 0, $$
therefore, by \eqref{eqn:limsupW},
\begin{equation}
\begin{aligned}
&\bar \mu(g(-1))=\bar\mu(g(t))=\mathcal W(g(t),f(\cdot,t),|t|)\\
&=\lim_{r\rightarrow +\infty}\int_{B(p,t,r)}\left(|t|\left(R_{g(t)} +2\Delta_{g(t)} f(\cdot,t)-|\nabla^{g(t)} f(\cdot,t)|_{g(t)}^2\right) +f(\cdot,t)-n \right) d\nu_{f,t},
\end{aligned}
\end{equation}
it follows that $\limsup_j \mathcal W_{p_j}(|t|) \leq \bar\mu(g(-1))$.

On the other hand,
\begin{equation}\label{eqn:W_two_integrals}
\begin{aligned}
\mathcal W_{p_j}(|t|)&= \int_{B(p_j,t,r\sqrt{|t|})} \left(|t|\left(R_{g_j(t)}+|\nabla^{g_j(t)} f_j(\cdot,t)|_{g_j(t)}^2 \right)+f_j(\cdot,t) -n\right) d\nu_{(p_j,0),t}\\
&+\int_{M_j\setminus B(p_j,t,r\sqrt{|t|})} \left(|t|\left(R_{g_j(t)}+|\nabla^{g_j(t)} f_j(\cdot,t)|_{g_j(t)}^2 \right)+f_j(\cdot,t) -n\right) d\nu_{(p_j,0),t}.
\end{aligned}
\end{equation}
By the standard lower bound for the scalar curvature $R_{g_j(t)}\geq -\frac{n}{2(t-A_j)}$ on $M_j$, and the lower bound (RF3) on $f_j(\cdot,t)$, it follows that for large enough $r=r(n,H)$ the second term in \eqref{eqn:W_two_integrals} is non-negative, for every $j$ large enough. Therefore,
\begin{equation*}
\begin{aligned}
&\liminf_j \mathcal W_{p_j}(|t|)=\\
&= \lim_j \int_{B(p_j,t,r\sqrt{|t|})} \left(|t|\left(R_{g_j(t)}+|\nabla^{g_j(t)} f_j(\cdot,t)|_{g_j(t)}^2 \right)+f_j(\cdot,t) -n\right) d\nu_{(p_j,0),t}\\
&+\liminf_j \int_{M_j \setminus B(p_j,t,r\sqrt{|t|})} \left(|t|\left(R_{g_j(t)}+|\nabla^{g_j(t)} f_j(\cdot,t)|_{g_j(t)}^2 \right)+f_j(\cdot,t) -n\right) d\nu_{(p_j,0),t}\\
&\geq \int_{B(p,t,r\sqrt{|t|})} \left(|t|\left(R_{g(t)}+|\nabla^{g(t)} f(\cdot,t)|_{g(t)}^2 \right)+f(\cdot,t) -n\right) d\nu_{f,t}\\
&+\int_{M\setminus B(p,t,r\sqrt{|t|})} \left(|t|\left(R_{g(t)}+|\nabla^{g(t)} f(\cdot,t)|_{g(t)}^2 \right)+f(\cdot,t) -n\right) d\nu_{f,t}\\
&=\bar\mu(g(t))=\bar\mu(g(-1)).
\end{aligned}
\end{equation*}
This proves the proposition.
\end{proof}

The following proposition can be seen as a uniqueness assertion of solutions to the conjugate heat equation satisfying (CHF1) with respect to some point of the spine $\mathcal S_{\textrm{point}}$ of a selfsimilar Ricci flow.

\begin{proposition}\label{prop:Spoint_S}
Let $(M^n_j,g_j(t),p_j)_{t\in [A_j,0]}$, $A_j\rightarrow -\infty$, be a sequence of smooth compact Ricci flows satisfying (RF2) and (RF3), that converges to a Ricci flow $(M,g(t),p)_{t\in (-\infty,0)}$ induced by a gradient shrinking Ricci soliton. Suppose that the conjugate heat flows $\nu_{(p_j,0)}$ smoothly converge to a conjugate heat flow $\nu_\infty$ and that  $p\in \mathcal S_{\textrm{point}}$.   
Then $\nu _\infty \in\mathcal S$. 
Moreover, passing to a subsequence, we may assume that $\mathcal W_{p_j}(|t|)\rightarrow\bar \mu(g(-1))$, for every $t<0$.
\end{proposition}
\begin{proof}
Let  $d\nu_{(p_j,0),t} = (4\pi |t|)^{-n/2} e^{-f_j(\cdot,t)} d\vol_{g_j(t)}$ and $d\nu_{\infty,t} = (4\pi |t|)^{-n/2} e^{-f(\cdot,t)} d\vol_{g(t)}$, where  $f_j$ and $f$ are smooth functions on $M_j\times [A_j,0]$ and $M\times (-\infty,0)$ respectively. By assumption, $f_j\rightarrow f$. We need to show that $(M,g(\bar t),f(\cdot,\bar t))$ is a normalized gradient shrinking Ricci soliton at scale $|\bar t|$, for every $\bar t<0$. 
\\ 

Define $t_k=4^{-k} \bar t$ and $T_m = 4^m \bar t$, for integers $k,m\geq 0$. By assumption (RF2) and the monotonicity of $\mathcal W$, we know that for every $j$,
\begin{equation}\label{eqn:W_bounds} 
-\Lambda\leq\mu(g_j(t),|t|)\leq \mathcal W(g_j(t),f_j(\cdot,t),|t|) \leq 0.
\end{equation}
Thus, passing to a subsequence if necessary, we can always find $\tilde t_i = t_{k(i)}$ and $\tilde T_i = T_{m(i)}$ so that for every $j$
\begin{equation}\label{eqn:sed}
\mathcal W_{p_j}(|\tilde t_i|) - \mathcal W_{p_j}(4|\tilde t_i|) <1/i\quad \textrm{and} \quad
\mathcal W_{p_j}(|\tilde T_i|) - \mathcal W_{p_j}(4|\tilde T_i|) <1/i.
\end{equation}
To see this, note that each of the inequalities in \eqref{eqn:sed} can only fail for finitely many $t_k$ and $T_m$, by \eqref{eqn:W_bounds} and the monotonicity of $\mathcal W_{p_j}$. Therefore, there is a $\delta_i>0$ such that for each $j$, \eqref{eqn:sed} holds for some $t_k,T_m\in [-\delta_i^{-1},-\delta_i]$, with $k,m$ depending on both $i$ and $j$. By the compactness of  $[-\delta_i^{-1},-\delta_i]$, we obtain the claim by passing to a subsequence.
\\

Since $p\in\mathcal S_{\textrm{point}}$, there is a smooth function $\hat f\in C^\infty(M\times (-\infty,0))$ such that $p$ is a minimum point of $\hat f(\cdot,t)$ and $(M,g(t),\hat f(\cdot,t))$ is a normalized gradient shrinking Ricci soliton at scale $|t|$, for every $t<0$. 

By Proposition \ref{prop:backwards_forwards}, there is a $1$-parameter family $\varphi_t$ of diffeomorphisms of $M$ such that $g(t)=|t| \varphi_t^* g(-1)$, $\hat f(x,t)=\hat f(\varphi_t(x),t)$.

 Define $\tilde g_i(t)=|\tilde t_i|^{-1} g(t|\tilde t_i|)=\varphi_{\tilde t_i}^* g(t)$ and $\tilde G_i(t)=|\tilde T_i|^{-1} g(t|\tilde T_i|)=\varphi_{\tilde T_i}^* g(t)$. Notice that $\varphi_{\tilde t_i}(p)=\varphi_{\tilde T_i}(p)= p$, since $p$ is a critical point of $\hat f$. It follows that the sequences $(M,\tilde g_i (t),p)_{t\in (-\infty,0)}$ and $(M,\tilde G_i (t),p)_{t\in (-\infty,0)}$ both converge to $(M,g(t),p)_{t\in (-\infty,0)}$, as $i\rightarrow +\infty$, in fact they are isometric to $(M,g(t),p)_{t\in(-\infty,0)}$ with an isometry that fixes $p$. 
 
On the other hand, for each $i,j$, define $g_{i,j}(t)=|\tilde t_i|^{-1} g_j(t|\tilde t_i|)$ and $G_{i,j}(t)=|\tilde T_i|^{-1} g_j(t|\tilde T_i|)$. Then the pointed sequences $(M_j,g_{i,j}(t),p_j)_{t\in [A_j |\tilde t_i|^{-1},0]}$ and $(M_j,G_{i,j}(t),p_j)_{t\in [A_j |\tilde T_i|^{-1},0]}$ converge to \linebreak $(M,\tilde g_i(t),p)_{t\in(-\infty,0)}$ and $(M,\tilde G_i(t),p)_{t\in(-\infty,0)}$ respectively, as $j\rightarrow +\infty$. Moreover, by \eqref{eqn:sed},
\begin{equation}
\mathcal W_{g_{i,j},p_j}(1)-\mathcal W_{g_{i,j},p_j}(4) <1/i \quad \textrm{and}\quad \mathcal W_{G_{i,j},p_j}(1)-\mathcal W_{G_{i,j},p_j}(4) <1/i.
\end{equation}

We can thus construct diagonal subsequences $(M_l,\bar g_l(t),p_l)_{t\in [A'_l,0]}$ and $(M_l,\bar G_l(t),p_l)_{t\in [B'_l,0]}$, with \linebreak $A'_l,B'_l\rightarrow -\infty$,  such that
\begin{equation}\label{eqn:sed_diag}
\mathcal W_{\bar g_l,p_l}(1)-\mathcal W_{\bar g_l,p_l}(4) <1/l \quad \textrm{and}\quad \mathcal W_{\bar G_l,p_l}(1)-\mathcal W_{\bar G_l,p_l}(4) <1/l,
\end{equation}
both converging to $(M,g(t),p)_{t\in (-\infty,0)}$.
\\

Let $\bar u_l=(4\pi |t|)^{-n/2} e^{-\bar f_l}$ and $\bar U_l = (4\pi |t|)^{-n/2} e^{-\bar F_l}$ the conjugate heat kernels starting at $(p_l,0)$ of $(M_l, \bar g_l(t))_{t\in [A'_l,0]}$ and $(M_l, \bar G_l(t))_{t\in [B'_l,0]}$ respectively. Using the convergence of $(M_j,g_j(t),p_j)_{t\in [A_j,0]}$ to $(M,g(t), p)_{t\in(-\infty,0)}$ and assumption (RF3), by Proposition \ref{prop:compactness_rf} we can pass to a further subsequence so that $\bar f_l$ and $\bar F_l$ converge to smooth functions $\bar f_\infty$ and $\bar F_\infty$ in $M\times (-\infty,0)$, respectively.

Moreover, by \eqref{eqn:sed_diag} and the monotonicity formula \eqref{eqn:pointed_monotonicity} we obtain
\begin{equation}
\begin{aligned}
\int_{-4}^{-1}\int_{M_l} |t| \left| \ric_{\bar g_l} +\hess_{\bar g_l} \bar f_l -\frac{\bar g_l}{2|t|}\right|^2 \bar u_l d\vol_{\bar g_l(t)}dt&<1/l,\\
\int_{-4}^{-1}\int_{M_l} |t| \left| \ric_{\bar G_l} +\hess_{\bar G_l} \bar F_l -\frac{\bar G_l}{2|t|}\right|^2 \bar U_l d\vol_{\bar G_l(t)}dt&<1/l.
\end{aligned}
\end{equation} 
It follows that both $(M,g(t),\bar f_\infty(\cdot,t))$ and $(M,g(t),\bar F_\infty(\cdot,t))$ are gradient shrinking Ricci solitons at scale $|t|$, for every $t\in [-4,-1]$. Therefore, for every $t\in [-4,-1]$
\begin{equation}
\mathcal W(g(t),\bar f_\infty(\cdot, t),|t|) =\mathcal W(g(t),\bar F_\infty(\cdot, t),|t|)=\bar\mu(g(-1)),
\end{equation}
where $\mu(g(-1))$ is the entropy of the soliton $(M,g(-1),\hat f(\cdot,-1))$.

Therefore, applying Proposition \ref{prop:entropy_convergence}, using (RF4), we obtain that
\begin{equation}\label{eqn:WgG0}
|\mathcal W_{\bar g_l,p_l}(1) - \mathcal  W_{\bar G_l,p_l}(1)|\leq |\mathcal W_{\bar g_l,p_l}(1)  - \bar\mu(g(-1))| + |\mathcal W_{\bar G_l,p_l}(1)  - \bar\mu(g(-1))| \rightarrow 0.
\end{equation}
\\

Now, by passing to a subsequence of $(M_j,g_j(t),p_j)_{t\in [A_j,0]}$, we can assume that  there are  $\alpha_j,\beta_j>0$ such that $\alpha_j\rightarrow +\infty$, $\beta_j\rightarrow 0$, $\mathcal W_{g_j,p_j}(\alpha_j) =\mathcal W_{\bar g_l,p_l}(1)$ and
$\mathcal W_{g_j,p_j}(\beta_j) =\mathcal W_{\bar G_l,p_l}(1)$. 

Therefore, using \eqref{eqn:WgG0} we obtain
\begin{equation}
\int_{4\bar t}^{\bar t} \int_{M_j} |t| \left| \ric_{g_j} +\hess_{g_j} f_j -\frac{g_j}{2|t|}\right|^2 d\nu_{(p_j,0)} dt \leq  \mathcal W_{g_j,p_j}(\alpha_j) -\mathcal W_{g_j,p_j}(\beta_j) \rightarrow 0.
\end{equation}
Passing to the limit, we conclude that $(M,g(\bar t), f(\cdot,\bar t))$ is a gradient shrinking Ricci soliton at scale $|\bar t|$. Moreover, for $j$ large,
\begin{equation*}
\mathcal W_{g_j,p_j}(\alpha_j)\leq\mathcal W_{g_j,p_j}(|t|) \leq \mathcal W_{g_j,p_j}(\beta_j),
\end{equation*}
thus $\mathcal W_{g_j,p_j}(|t|)\rightarrow \bar\mu(g(-1))$.
\end{proof}

\subsection{Almost selfsimilar Ricci flows}
We will say that a pointed smooth complete Ricci flow $(M,g(t),p)_{t\in (-\infty,0)}$ is $k$-selfsimilar if $(M,g(t))_{t\in (-\infty,0)}$ is induced by a gradient shrinking Ricci soliton splitting $k$ Euclidean factors and $p\in \mathcal S_{\textrm{point}}$.

\begin{definition}
A pointed Ricci flow $(M^n,g(t),p)_{t\in (-2\delta^{-1} r^2,0)}$  is said to be $(k,\delta)$-selfsimilar at scale $r>0$ if $g_r(t)= r^{-2}g(r^2 t)$ satisfies the following: there is a $k$-selfsimilar Ricci flow $(\tilde M^n, \tilde g(t),q)_{t\in(-\infty,0)}$, and a diffeomorphism onto its image $F:B(q,-1,\delta^{-1}) \rightarrow M$ such that $F(q)=p$ and 
\begin{enumerate}
\item For every integer $0\leq l< \delta^{-1}$, we have $ \left|\tilde \nabla^l (F^* g_r - \tilde g)\right|_{\tilde g}<\delta$
in $B(q,-1,\delta^{-1})\times [-\delta^{-1},-\delta]$.
\item For every $(x,t)\in B(q,-1,\delta^{-1}) \times [-\delta^{-1},-\delta]$, we have $|d_{\tilde g(t)}(q,x) - d_{g_r(t)}(p,F(x)) |<\delta$.
\end{enumerate}

We define $\mathcal L_{p,r}=F( B(q,-1,\delta^{-1}) \cap \mathcal S_{\textrm{point}})$ and say that $(M,g(t),p)_{t\in (-2\delta^{-1} r^2,0]}$ is $(k,\delta)$-selfsimilar around $p$ at scale $r$ with respect to $\mathcal L_{p,r}$.
\end{definition}

\begin{remark}
Note that a $k$-selfsimilar Ricci flow may split more than $k$ Euclidean factors, so in general we know that $\mathcal S_{\textrm{point}}=\mathcal K\times \mathbb R^l$ for some $l\geq k$ and $\mathcal K$ satisfying the estimate of Proposition \ref{prop:spine_structure}. Thus, if a pointed Ricci flow is $(k,\delta)$-selfsimilar at scale $r$ with respect to $\mathcal L_{p,r}$, it is not necessarily true that the large scale geometry of $\mathcal L_{p,r}$ is  $k$-dimensional.
\end{remark}

\begin{lemma}\label{lemma:k_delta_convergence}
Let $\delta_j\rightarrow 0$ and $(M^n_j,g_j(t),p_j)_{t\in (-2\delta_j^{-1},0)}$  be a pointed sequence of complete Ricci flows converging smoothly to a complete pointed Ricci flow $(M_\infty,g_\infty(t),p_\infty)_{t\in (-\infty,0)}$.  Suppose that each $(M_j,g_j(t),p_j)_{t\in (-2\delta_j^{-1},0)}$ is $(k,\delta_j)$-selfsimilar at scale $1$. Then $(M_\infty,g_\infty(t),p_\infty)_{t\in (-\infty,0)}$ is $k$-selfsimilar with spine $\mathcal S$.

Moreover, if $M_j$ are compact and each $g_j(t)$ is defined in $(-2\delta_j^{-1},0]$ satisfying (RF2) and (RF3), then the conjugate heat flows $\nu_{(p_j,0)}$ converge to a conjugate heat flow $\nu_\infty\in\mathcal S$ and 
$$\mathcal W_{p_j}(|t|)\rightarrow \bar\mu(g_\infty(-1)),$$ for every $t<0$.
\end{lemma}
\begin{proof}
Since each $(M_j,g_j(t),p_j)_{t\in (-2\delta_j^{-1},0)}$ is $(k,\delta_j)$-selfsimilar, there are 
\begin{enumerate}
\item $k$-selfsimilar Ricci flows $(\bar M_j,\bar g_j(t),\bar p_j)_{t\in (-\infty,0)}$ with spine $\mathcal S_j$ and point spine $\mathcal S_{j,\textrm{point}}$, with $\bar p_j\in\mathcal S_{j,\textrm{point}}$.
\item maps $F_j: B(\bar p_j,-1,\delta_j^{-1}) \rightarrow M_j$, diffeomorphisms onto their image such that $F_j(\bar p_j) = p_j$ and $|\bar\nabla^l(F_j^*g_j-\bar g_j)|_{\bar g_j} < \delta_j$, for every $0\leq l<\delta_j^{-1}$, and 
$$|d_{g_j(t)}(p_j,F_j(x)) - d_{\bar g_j(t)}(\bar p_j,x)|<\delta_j$$
in $B(\bar p_j,-1,\delta_j^{-1})\times [-\delta_j^{-1},-\delta_j]$.
\item $\nu_{\bar f_j}\in\mathcal S_j$ such that $\bar f_j(\cdot,t)$ attain a minimum at $\bar p_j$.
\end{enumerate}
Thus, for large $j$, we can define on $B(p_j,-1,(2\delta_j)^{-1})\times [-(2\delta_j)^{-1},-2\delta_j]$ the smooth functions $f_j=(F_j^{-1})^*\bar f_j$. Observe that $f_j(\cdot,t)$ attain a minimum at $p_j$ and for every $r>0$ 
\begin{align}
\min_{B(p_j,-1,r)\times [-r^2,-r^{-2}]} |t| \min(R_{g_j(t)},0)& \longrightarrow 0,\label{eqn:scalar}\\
\max_{B(p_j,-1,r)\times [-r^2,-r^{-2}]}\left|\nabla^l\left(\ric_{g_j(t)} + \hess_{g_j(t)} f_j - \frac{g_j(t)}{2|t|} \right)\right|_{g_j(t)} &\longrightarrow 0,\quad \textrm{for any $l\geq 0$},\label{eqn:d_soliton_eqn}\\
\max_{B(p_j,-1,r)\times [-r^2,-r^{-2}]} \left|\left( \frac{\partial}{\partial t} +\Delta_{g_j(t)} - R_{g_j(t)}\right)\left( (4\pi |t|)^{-n/2} e^{-f_j} \right) \right|&\longrightarrow 0\label{eqn:conj_h},
\end{align}

Moreover, by the growth bounds on soliton functions of Lemma \ref{lemma:soliton_growth}, we can ensure that for large $j$ and any $(x,t)\in  B(p_j,-1,\delta_j^{-1})\times [-\delta_j^{-1},\delta_j]$, $\hat f_j = f_j-\bar \mu(\bar g_j(-1))$ satisfies
\begin{equation}\label{eqn:hatf_growth}
\frac{d_{g_j(t)}(p_j,x)^2}{2a(n) |t|} - 2a(n) \leq \hat f_j(x,t) \leq 2a(n) \left( \frac{d_{g_j(t)}(p_j,x)^2}{|t|} + 1\right)
\end{equation}
and for every $r>0$
\begin{equation*}
\max_{B(p_j,-1,r)\times [-r^2,-r^{-2}]} \left| |t|(|\nabla^{g_j(t)} f_j |^2 +R_{g_j(t)}) - \hat f_j \right|\rightarrow 0,
\end{equation*}
by Lemma \ref{lemma:soliton_identity_entropy}.

In particular, for any $r<+\infty$ and large $j$ we have the estimates
\begin{align*}
|t| |\nabla^{g_j(t)} \hat f_j |^2 \leq \hat f_j-|t| R_{g_j(t)}&\leq 2a(n) \left( \frac{d_{g_j(t)}(p_j,x)^2}{|t|} + 1\right)\\
|\nabla^k\hess_{g_j(t)} \hat f_j|_{g_j(t)}&\leq C_k
\end{align*}
in $B(p_j,-1,r)\times [-r^2,-r^{-2}]$.

Therefore, passing to a further subsequence we can assume that $\hat f_j$ smoothly converge to a limit $\hat f_\infty$ on $(M_\infty,g_\infty(t),p_\infty)_{t\in (-\infty,0)
}$. Moreover, due to \eqref{eqn:hatf_growth},
$$\int_{B(p_j,-1,(2\delta_j)^{-1})} (4\pi |t|)^{-n/2} e^{-\hat f_j(\cdot,t)} d\vol_{g_j(t)} \longrightarrow \int_{M_\infty}  (4\pi |t|)^{-n/2} e^{-\hat f_\infty(\cdot,t)} d\vol_{g_\infty(t)}\in (0,+\infty)$$

Therefore, there is a constant $C>0$ such that for large $j$ 
\begin{equation*}
C^{-1}\leq e^{\bar \mu(\bar g_j(-1))}\int_{B(p_j,-1,(2\delta_j)^{-1})} (4\pi |t|)^{-n/2} e^{-f_j} d\vol_{g_j(t)}.
\end{equation*}
from which we obtain that
$$e^{-\bar \mu(\bar g_j(-1))} C^{-1} \leq \int_{B(p_j,-1,(2\delta_j)^{-1})}  (4\pi |t|)^{-n/2} e^{-f_j} d\vol_{g_j(t)} \leq 1.$$

This gives that $\inf_j \bar\mu(\bar g_j(-1))>-\infty$. Therefore, passing to a further subsequence we can assume that $\bar \mu(\bar g_j(-1))\rightarrow \bar \mu_\infty$ and $f_j \rightarrow f_\infty=\hat f_\infty +\bar \mu_\infty$. 

By \eqref{eqn:d_soliton_eqn} and the growth bounds \eqref{eqn:hatf_growth}, it follows that $(M_\infty, g_\infty(t),f_\infty(\cdot,t))$ is a normalized gradient shrinking Ricci soliton at scale $|t|$, since 
$$(4\pi |t|)^{-n/2} \int_{\bar M_j} e^{-\bar f_j(\cdot,t)} d\vol_{\bar g_j(t)} = 1.$$
 Moreover, by \eqref{eqn:conj_h} we obtain
$$\left( \frac{\partial}{\partial t} +\Delta_{g_j(t)} - R_{g_j(t)}\right)\left( (4\pi |t|)^{-n/2} e^{-f_j} \right)=0.$$
Thus, $(M_\infty,g_\infty(t))_{t\in(-\infty,0)}$ is induced by a normalized gradient shrinking soliton with  $\nu_{f_\infty}\in\mathcal S$ and $f_\infty$ attains a minimum at $p_\infty\in \mathcal S_{\textrm{point}}$, hence  $(M_\infty,g_\infty(t),p_\infty)_{t\in(-\infty,0)}$ is $0$-selfsimilar. 

In a similar manner, if $(M_j,g_j(t),p_j)_{t\in(-2\delta_j^{-1},0)}$ is $(k,\delta_j)$-selfsimilar, we can obtain limit functions $v_\infty^1,\ldots,v_\infty^k$ on $(M_\infty,g_\infty(t),p_\infty)_{t\in(-\infty,0)}$ satisfying
$$\hess_{g_\infty(t)} v^a = 0, \quad \langle \nabla v^a,\nabla v^b\rangle = \delta^{ab},$$
for any $a,b=1,\ldots,k$. This proves that $(M_\infty,g_\infty(t),p_\infty)_{t\in(-\infty,0)}$ isometrically splits $k$ Euclidean factors and thus it is $k$-selfimilar.

The remaining assertions, assuming (RF2), (RF3), and that $M_j$ is compact, follow by applying Proposition \ref{prop:Spoint_S}.

\end{proof}

\begin{corollary}
Fix $s>0$ and $\varepsilon>0$, and let $(M^n,g(t),p)_{t\in (-2\delta^{-1},0)}$ be a smooth compact Ricci flow satisfying (RF1-2) which is $(k,\delta)$-selfsimilar with respect to $\mathcal L_{p,1}$. Suppose that $x\in B(p,-1,s)$ is such that $(M,g(t),x)_{t\in (-2\delta^{-1},0)}$ is $(0,\delta)$-selfsimilar. If $0<\delta\leq \delta(n,C_I,\Lambda|s,\varepsilon)$ then 
$$d_{g(t)}(x,\mathcal L_{p,1}) \leq \varepsilon $$
for every $t\in [-\varepsilon^{-1},-\varepsilon]$.
\end{corollary}
\begin{proof}
Consider any sequence $\delta_j\rightarrow 0$ and smooth compact Ricci flows $(M_j,g_j(t),p_j)_{t\in (-2\delta_j^{-1},0)}$ satisfying (RF1-2) which are $(k,\delta_j)$-selfsimilar with respect to $\mathcal L_{p_j,1}$, and $x_j \in B(p_j,-1,s)\subset M_j$ such that $(M_j,g_j(t),x_j)_{t\in (-2\delta_j^{-1},0)}$ is $(0,\delta_j)$-selfsimilar. By Proposition \ref{prop:compactness_rf} and Lemma \ref{lemma:k_delta_convergence}, passing to a subsequence, we may assume that $(M_j,g_j(t),p_j)_{t\in (-2\delta_j^{-1},0)}$ converges to a smooth complete Ricci flow $(M_\infty,g_\infty(t),p_\infty)_{t\in (-\infty,0)}$, which is $k$-selfsimilar.

Again, passing to a subsequence we may assume that there is $x_\infty\in B(p_\infty,-1,s)\subset M_\infty$ such that $x_j\rightarrow x_\infty$, so $(M_j,g_j(t),x_j)_{t\in (-2\delta_j^{-1},0)}$ converges to $(M_\infty,g_\infty(t),x_\infty)_{t\in (-\infty,0)}$, which by Lemma \ref{lemma:k_delta_convergence} is also $0$-selfsimilar.

In particular, $p_\infty,x_\infty \in \mathcal S_{\textrm{point}}$, which proves that $d_{g_j(t)} (x_j,\mathcal L_{p_j,1})\leq \varepsilon $, for any $t\in [-\varepsilon^{-1},-\varepsilon]$ if $j$ is large. This suffices to prove the result.
\end{proof}

\subsubsection{$\mathcal W$-entropy drop and almost selfsimilar Ricci flows}

\begin{proposition}\label{prop:L}
Let $(M,g(t),p)_{t\in (-2\delta^{-2},0]}$ be a smooth compact Ricci flow satisfying (RF1-3) and suppose that it is $(0,\delta^2)$-selfsimilar at scale $1$ with respect to $\mathcal L_{p,1}$. If $0<\delta\leq \delta(n,C_I,\Lambda,C_1|r,\varepsilon)$, then for any $q\in \mathcal L_{p,1}\cap B(p,-1,r)$, 
$$\mathcal W_q(\varepsilon)-\mathcal W_q(\varepsilon^{-1})<\varepsilon$$
and $(M,g(t),q)_{t\in (-2\delta^{-1},0)}$ is $(0,\delta)$-selfsimilar.

Moreover, there is a product metric space $(\mathcal K\times \mathbb R^l, d_{\mathcal K\times \mathbb R^l})$, $0\leq l\leq n$, with $\diam(\mathcal K)\leq A(n)$, such that
$$d_{GH} \left(\mathcal L_{p,1}\cap B(p,-1,r), (\mathcal K\times \mathbb R^l )\cap B(\tilde p,r) \right) < \delta,$$
where $\mathcal L_{p,1}\cap B(p,-1,r)$ is assumed to carry the induced metric from $d_{g(-1)}$.
\end{proposition}
\begin{proof}
Suppose that for some $\varepsilon>0$ there is a sequence $\delta_j\rightarrow 0 $ and a sequence of smooth compact Ricci flows $(M_j,g_j(t),p_j)_{t\in (-2\delta_j^{-1},0)}$ satisfying  satisfying (RF1), (RF2) and (RF3) which are $(0,\delta_j^2)$-selfsimilar at scale $1$, and suppose that there exist $q_j\in \mathcal L_{p_j,1}\cap B(p_j,-1,r)\subset M_j$ such that
\begin{equation}\label{eqn:entropy_contradiction}
\mathcal W_{q_j}(\varepsilon) - \mathcal W_{q_j}(\varepsilon^{-1}) \geq \varepsilon,
\end{equation}
for every $j$. Observe that since $q_j\in \mathcal L_{p_j,1}\cap B(p_j,-1,r)$ we also know that $(M_j,g_j(t),q_j)_{t \in (-2\delta_j^{-1},0)}$ is $(0,\delta_j)$-selfsimilar, for large $j$.

By (RF1), (RF2)  and Proposition \ref{prop:compactness_rf} we can assume, by passing to a subsequence, that \linebreak $(M_j,g_j(t),q_j)_{t\in (-2\delta_j^{-2},0)}$ converges to a smooth complete Ricci flow $(M_\infty,g_\infty(t),q_\infty)_{t\in(-\infty,0)}$. Thus, using Assumptions (RF2) and (RF3) we can apply Lemma \ref{lemma:k_delta_convergence} to conclude that passing to a further subsequence that $(M_\infty,g_\infty(t),q_{\infty})_{t\in (-\infty,0)}$ is $0$-selfsimilar and 
$$\mathcal W_{q_j}(|t|) \rightarrow \bar\mu(g_\infty(-1)).$$
In particular, for large $j$
$$\mathcal W_{q_j}(\varepsilon^{-1}) - \mathcal W_{q_j}(\varepsilon)<\varepsilon,$$
which contradicts \eqref{eqn:entropy_contradiction}.

The remaining assertion follows from $\mathcal L_{p,1} \cap B(p,-1,r)$ being the image of $(\mathcal K\times \mathbb R^l )\cap B(\tilde p,r)$ under a map with $\delta$-small distortion, due to $(M,g(t),p)_{t\in (-2\delta^{-2},0)}$ being $(0,\delta^2)$-selfsimilar.
\end{proof}

\begin{proposition}\label{prop:W_drop_small}
Let $(M^n_j,g_j(t),p_j)_{t\in [-2\delta_j^{-1},0]}$, $\delta_j\rightarrow 0$, be a sequence of compact Ricci flows satisfying (RF1-3) such that
\begin{equation}\label{eqn:prop_W_drop_small}
\mathcal W_{p_j}(\delta_j) - \mathcal W_{p_j}(\delta_j^{-1}) <\delta_j. 
\end{equation}
Then, passing to a subsequence, we may assume that $(M_j,g_j(t),p_j)_{t\in (-2\delta_j^{-1},0]}$ converges to a smooth complete Ricci flow $(M_\infty,g_\infty(t),p_\infty)_{t\in (-\infty,0)}$ which is induced by a gradient shrinking Ricci soliton and satisfies (RF1). Moreover, there is $D=D(n,H)<+\infty$ such that
\begin{equation}\label{eqn:Sc_distance}
d_{g(t)}(p_\infty,\mathcal S_{\textrm{point}}) \leq D \sqrt{|t|}.
\end{equation}
\end{proposition}
\begin{proof}
The convergence to the Ricci flow induced by a gradient shrinking Ricci soliton follows from assumptions (RF1), (RF2) and (RF3), using Proposition \ref{prop:compactness_rf}, the monotonicity formula for $\mathcal W_{p_j}$ and \eqref{eqn:prop_W_drop_small}.

In particular, by assumption (RF3), the conjugate heat kernels $\nu_{(p_j,0)}$ converge to a conjugate heat flow $\nu_f\in \mathcal S$ on $(M,g_\infty(t))_{t\in (-\infty,0)}$ satisfying
\begin{equation}\label{eqn:f_lower_bound}
f(x,t) \geq \frac{d_{g_\infty(t)}(p,x)^2}{H|t|} - H, 
\end{equation}
where $H$ is the constant in assumption (RF3).

On the other hand, if $\hat p\in \mathcal S_{\textrm{point}}$ is a point where $f$ attains a minimum, by Proposition \ref{prop:soliton_identities} 
\begin{equation}\label{eqn:f_upper_bound}
\begin{aligned}
f(\hat p,t)&= |t|\left(R_{g_\infty(t)}(\hat p,t) +|\nabla^{g_\infty(t)} f|_{g_\infty(t)}^2(\hat p,t) \right)+\bar\mu(g_\infty(-1))\\
&\leq |t| R_{g_\infty(t)}(\hat p,t) = - |t|\Delta_{g_\infty(t)} f(\hat p,t) +\frac{n}{2} \leq \frac{n}{2},
\end{aligned}
\end{equation}
since $\bar\mu(g_\infty(-1))\leq 0$, $|\nabla^{g_\infty(t)} f|_{g_\infty(t)}^2(\hat p,t)=0$ and $\Delta_{g_\infty(t)} f(\hat p,t)\geq 0$.

Therefore, combining \eqref{eqn:f_lower_bound} with \eqref{eqn:f_upper_bound} gives
$$\frac{d_{g_\infty(t)}(p,\hat p)^2}{H|t|} - H\leq \frac{n}{2},$$
which proves \eqref{eqn:Sc_distance}. 
\end{proof}

\begin{corollary}\label{cor:distance_2D}
Let $(M^n,g(t),p)_{t\in (-2\delta^{-1},0)}$ be a smooth compact Ricci flow satisfying (RF1-3).  If $0<\delta\leq \delta(n,C_I,\Lambda,H|\varepsilon)$ and $\mathcal W_p(\delta)-\mathcal W_p(\delta^{-1}) <\delta$, then there is $q\in M$ with $d_{g(t)}(p,q)\leq 2D(n,H)\sqrt{|t|}$ for every $t\in [-\varepsilon^{-1},-\varepsilon]$ such that $(M,g(t),q)_{t\in (-2\varepsilon^{-1},0)}$ is $(0,\varepsilon)$-selfsimilar, where $D=D(n,H)<+\infty$ is the constant provided by Proposition \ref{prop:W_drop_small}. In particular, 
$$d_{g(t)}(p,\mathcal L_{q,1})\leq 2D \sqrt{|t|}$$
for every $t\in [-\varepsilon^{-1},-\varepsilon]$.
\end{corollary}
\begin{proof}
Fix $\varepsilon>0$ and consider a sequence of counterexamples, namely a sequence of smooth compact Ricci flows $(M_j,g_j(t),p_j)_{t\in (-2\delta_j^{-1},0)}$, $\delta_j\rightarrow 0$, satisfying (RF1-3), $\mathcal W_{p_j}(\delta_j)-\mathcal W_{p_j}(\delta_j^{-1}) <\delta_j$, but for each $j$ there is no $q\in M_j$ such that $(M_j,g_j(t),q)_{t\in (-2\varepsilon^{-1},0)}$ is $(0,\varepsilon)$-selfsimilar and $$d_{g_j(t)}(p_j,q)\leq 2D \sqrt{|t|}$$
for every $t\in [-\varepsilon^{-1},-\varepsilon]$

Let $D=D(n,H)$ be the constant provided by Proposition \ref{prop:W_drop_small}. Then, by Proposition \ref{prop:W_drop_small}, we may assume that $(M_j,g_j(t),p_j)_{t\in (-2\delta_j^{-1},0)}$ converges to the Ricci flow $(M_\infty,g_\infty(t),p_\infty)_{t\in(-\infty,0)}$ induced by a gradient shrinking Ricci soliton and $$d_{g_\infty(t)}(p_\infty,\mathcal S_{\textrm{point}})\leq D\sqrt{|t|},$$
for every $t<0$. This means that there is a $q_\infty\in \mathcal S_{\textrm{point}}$ such that $d_{g_\infty(t)}(p_\infty,q_\infty)\leq D\sqrt{|t|}$ for every $t<0$ and $q_j\in M_j$ such that $q_j\rightarrow q_\infty$. But then, for every $t\in [-\varepsilon^{-1},-\varepsilon]$
$$d_{g_j(t)}(p_j,q_j)\leq 2D\sqrt{|t|}$$
and $(M_j,g_j(t),q_j)_{t\in(-2\varepsilon^{-1},0)}$ is $(0,\varepsilon)$-selfsimilar, for large $j$,  which is a contradiction.

\end{proof}

\section{The heat equation under Ricci flow}\label{sec:heat_rf}

Let $(M,g(t))_{t\in I}$ be a Ricci flow and $v:M\times I \rightarrow \mathbb R$ be a solution to the heat equation \eqref{eqn:heat_equation_prelim}. Standard computations give
\begin{align}
\frac{\partial}{\partial t} v^2 &=\Delta v^2 - 2|\nabla v|^2,\\
\frac{\partial}{\partial t} |\nabla v|^2 &= \Delta |\nabla v|^2-2|\hess v|^2\\
\frac{\partial}{\partial t}\hess v&= \Delta_L \hess v,\label{eqn:evol_hess}
\end{align}
where $\Delta_L h_{ij}= \Delta h_{ij}+2R_{ikmj}h_{ij}-R_{ik}h_{kj}-h_{ik}R_{kj}$  denotes the Lichnerowicz Laplacian.

\begin{lemma}\label{lemma:Lich_type_I}
Let $(M,g(t))_{t\in I}$, $\sup I=0$ be smooth Ricci flow satisfying (RF1) and a family of symmetric $2$-tensors $S$ satisfying
\begin{equation}\label{eqn:Lich_subsolution}
\frac{\partial}{\partial t} S =\Delta_L S.
\end{equation}
Then, there is a positive integer $E=E(n,C_I)$ such that
\begin{equation}
\left(\frac{\partial}{\partial t} -\Delta\right)(|t|^E |S|^2)\leq -2|\nabla (|t|^{\frac{E}{2}} S)|^2.
\end{equation}
\end{lemma}
\begin{proof}
By \eqref{eqn:evol_hess} and assumption (RF1)  we obtain that
\begin{equation*}
\frac{\partial}{\partial t} |S|^2 \leq \Delta |S|^2 - |\nabla S|^2+ \frac{E(n,C_I)|S|^2}{|t|}.
\end{equation*}
Therefore, 
\begin{equation*}
\begin{aligned}
\frac{\partial}{\partial t}(|t|^E |S|^2) &\leq -E |t|^{E-1} |S|^2 + |t|^E \left(\Delta |S|^2 -|\nabla S|^2) +\frac{E |S|^2}{|t|}\right)\\
&=\Delta (|t|^E |S|^2) - |\nabla (|t|^{E/2} S)|^2.
\end{aligned}
\end{equation*}
\end{proof}

\begin{corollary}\label{cor:subsolutions}
Let $(M,g(t))_{t\in I}$ be a smooth Ricci flow satisfying (RF1) and $v\in C^\infty(M\times I)$ be a solution to the heat equation $\frac{\partial}{\partial t}v= \Delta v$. Then, there is $E=E(n,C_I)$ such that, in the viscosity sense,
\begin{equation}\label{eqn:absolute_evolution}
\begin{aligned}
\left(\frac{\partial}{\partial t}-\Delta\right) |v|&\leq 0,\\
\left(\frac{\partial}{\partial t}-\Delta \right)|\nabla v|&\leq 0,\\
\left(\frac{\partial}{\partial t}-\Delta\right)( |t|^{E/2} |\hess v|)&\leq 0.
\end{aligned}
\end{equation}
\end{corollary}

\subsection{Polynomial growth estimates}

\begin{lemma}\label{lemma:bound_heat}
Let $(M,g(t),p)_{t\in [0,T)}$ be smooth complete Ricci flow, satisfying (RF1) and (RF3). Consider a non-negative subsolution $v$ of the heat equation, namely $(\partial_t-\Delta) v\leq0$ and suppose that for some constant $A_1<+\infty$
 $$0\leq v(x,0)\leq A_1 \left((d_{g(0)}(p,x))^k+1\right).$$ 
 Then there is a constant $A_2=A_2(n,C_I,C_1)<+\infty$ such that for every $t\in [0,T)$
\begin{equation*}
0\leq v(x,t) \leq A_2 \left(d_{g(0)}(p,x)^k +t^{k/2}+1\right).
\end{equation*}
\end{lemma}
\begin{proof}
Fix $(x,t)\in M\times (0,T)$ and set $R=d_{g(0)}(p,x)$.

 We then have
\begin{align*}
&v(x,t)\leq \int_M v(y,0) d\nu_{(x,t),0}(y)\leq A_1\int_M  \left((d_{g(0)}(p,y))^k+1\right) d\nu_{(x,t),0}(y).
\end{align*}
By the triangle inequality we have that $d_{g(0)}(p,y) \leq d_{g(0)}(p,x) +d_{g(0)}(x,y)$ so it follows that if $d_{g(0)}(x,y) \geq R$ then
$$(d_{g(0)}(p,y))^k +1\leq (d_{g(0)}(p,x) +d_{g(0)}(x,y))^k +1\leq 2^k (d_{g(0)}(x,y))^k +1 $$
whereas if $d_{g(0)}(x,y) < R$ then
$$(d_{g(0)}(p,y))^k +1\leq (d_{g(0)}(p,x) +d_{g(0)}(x,y))^k +1 \leq 2^k R^k +1 = 2^k (d_{g(0)}(p,x))^k + 1.$$
Therefore,
\begin{align*}
&v(x,t) \leq  A_1\int_M ( (d_{g(0)}(p,y))^k+1) d\nu_{(x,t),0}(y)\\
&\leq A_1 \left(\int_{M\setminus B(x,0, R)} ( (d_{g(0)}(p,y))^k+1) d\nu_{(x,t),0}(y) + \int_{B(x,0,R)} ( (d_{g(0)}(p,y))^k+1 ) d\nu_{(x,t),0}(y) \right)\\
&\leq A_1\left(\int_M (2^k (d_{g(0)}(x,y))^k +1 ) d\nu_{(x,t),0}(y)+ \int_M (2^k (d_{g(0)}(p,x))^k + 1) d\nu_{(x,t),0}(y)\right)\\
&\leq 2^k A_1 \int_M (d_{g(0)}(x,y))^k d\nu_{(x,t),0}(y))+2^k A_1(d_{g(0)}(p,x))^k +A_1\\
&\leq 2^k A_1 t^{k/2} \int_M \left(\frac{d_{g(0)}(x,y)}{t^{1/2}}\right)^k d\nu_{(x,t),0}(y)+2^k A_1((d_{g(0)}(p,x))^k +1)\\
&\leq A_2((d_{g(0)}(p,x))^k + t^{k/2}+1),
\end{align*}
where the last inequality used Lemma \ref{lemma:int_ker_bounds}.
\end{proof}

\begin{corollary}\label{cor:local_heat_estimates}
Let $(M,g(t),p)_{t\in[t_1,0]}$ be a smooth complete Ricci flow, satisfying (RF1) and (RF3), and $v: M\times [t_1,0]\rightarrow \mathbb R$ a solution to $\left(\frac{\partial}{\partial t} -\Delta\right)v =a$. Suppose that there are constants such that for any $x\in M$
\begin{equation}
|v(x,t_1)|^2 \leq A_1 ((d_{g(t_1)}(p,x))^k+1).
\end{equation}
Then, there are constants $E=E(n,C_I)$ and $A_2=A_2(n,C_I,C_1,t_1,a)<+\infty$ such that for any $(x,t)\in M\times [t_1,0]$
\begin{align*}
|v(x,t)|^2 + (t-t_1)|\nabla v|^2(x,t) +(t-t_1)^2 |t_1|^{-E} |t|^E |\hess v|^2 &\leq A_2 ((d_{g(t_1)}(p,x))^k+1)
\end{align*}
\end{corollary}
\begin{proof}
We have, for $t\in [t_1,0]$ 
\begin{align*}
\left(\frac{\partial}{\partial t}-\Delta\right)v^2&=-2|\nabla v|^2+2av\leq -2|\nabla v|^2 +v^2 +a^2\\
\left(\frac{\partial}{\partial t}-\Delta\right) \left( (t-t_1) |\nabla v|^2 \right)&=-2(t-t_1)|\hess v|^2 + |\nabla v|^2\\
\left(\frac{\partial}{\partial t}-\Delta\right)\left((t-t_1)^2 |t_1|^{-E} |t|^E |\hess v|^2\right)&\leq  2(t-t_1) |t_1|^{-E}|t|^E |\hess v|^2\leq 2(t-t_1) |\hess v|^2
\end{align*}
Hence, for $w=v^2+a^2+(t-t_1)|\nabla v|^2 + (t-t_1)^2 |t_1|^{-E} |t|^E|\hess v|^2$ and $t\in [t_1,0]$
\begin{align*}
\left(\frac{\partial}{\partial t}-\Delta\right)w&\leq -2|\nabla v|^2 +v^2+a^2-2(t-t_1)|\hess v|^2 + |\nabla v|^2+2(t-t_1)  |\hess v|^2\\
&=-|\nabla v|^2 +v^2+a^2\leq w
\end{align*}
Thus, $\tilde w=e^{-(t-t_1)}w$ is a subsolution of the heat equation and
$$\tilde w(x,t_1) = v^2(x,t_1)+a^2\leq \tilde A_1 ((d_{g(t_1)}(p,x))^k+1)$$
 Applying Lemma \ref{lemma:bound_heat} we obtain that
 $$\tilde w(x,t)\leq  \bar A_1(( d_{g(t_1)}(p,x))^k +(t-t_1)^{k/2}+1),$$
 from which it follows that
 $$v^2+(t-t_1)|\nabla v|^2 +(t-t_1)^2 |t_1|^{-E} |t|^E |\hess v|^2 \leq e^{t-t_1} \tilde w\leq \bar A_1 e^{t-t_1} ( (d_{g(t_1)}(p,x))^k +(t-t_1)^{k/2}+1),$$
 which suffices to prove the result.
 \end{proof}

\subsection{Concentration estimates}
The following theorem from \cite{bamler2021entropy} describes a hypercontractivity property of solutions to the heat equation along a Ricci flow, and it is a consequence of the logarithmic Sobolev inequality  for  conjugate heat kernel measures in \cite{HeinNaber}. We will make extensive use of it, and its consequence Corollary \ref{cor:L2toL2weighted}, when dealing with integral estimates weighted by conjugate heat kernels. In particular, Theorem \ref{thm:hypercontractivity} will allow us to pass such estimates to smooth limits of Ricci flows, whereas Corollary \ref{cor:L2toL2weighted} will allow us to utilize Lemma \ref{lemma:compare_kernels} to change base points in the conjugate heat kernel weights.

\begin{theorem}[Theorem 12.1 in \cite{bamler2021entropy}]\label{thm:hypercontractivity}
Let $(M,g(t),p)_{t\in [t_1,0]}$ be a smooth Ricci flow on a compact manifold, and let $v:M\times [t_1,0]\longrightarrow \mathbb R$ be such that $v\geq 0$ and $(\partial_t -\Delta)v\leq 0$, in the viscosity sense. Then for $1<q\leq p<+\infty$ and $t_2\in [t_1,0]$ such that
$$\frac{|t_1|}{|t_2|} \geq \frac{p-1}{q-1}$$ we have
\begin{equation*}
\left(\int_M v^p(\cdot, t_2) d\nu_{(p,0),t_2}\right)^{1/p} \leq \left(\int_M v^q(\cdot, t_1) d\nu_{(p,0),t_1}\right)^{1/q}.
\end{equation*}
\end{theorem}

\begin{corollary}\label{cor:L2toL2weighted}
Let $(M,g(t),p)_{t\in [t_1,0]}$ be smooth Ricci flow on a compact manifold, and let $d\nu_{(p,0),t}=(4\pi|t|)^{-n/2} e^{-f(\cdot,t)} d\vol_{g(t)}$. Let $v:M\times [t_1,0]\rightarrow \mathbb R$ be such that $v\geq 0$ and $(\frac{\partial}{\partial t} -\Delta)v\leq 0$, in the viscosity sense. Then for every $\alpha\in [0,1)$ there is $\beta\in (0,1)$ such that for every $t\in [\beta t_1,0]$ and measurable $X\subset M$
\begin{equation}
\int_X v^m(\cdot,t) e^{\alpha f(\cdot,t)} d\nu_{(p,0),t} \leq  \left(\int_X e^{\alpha s f(\cdot,t)} d\nu_{(p,0),t} \right)^{1/s} \left(\int_M v^2(\cdot,t_1) d\nu_{(p,0),t_1}\right)^{m/2},
\end{equation}
where $m\geq 2$ and $s=\frac{t_1/t +1}{t_1/t-m+1}$. In particular, if $(M,g(t))_{t\in [t_1,0]}$ satisfies (RF1) and (RF3), then
\begin{equation}
\int_M v^m(\cdot,t) e^{\alpha f(\cdot,t)} d\nu_{(p,0),t} \leq C(n,C_I,C_1|\alpha) \left(\int_M v^2(\cdot,t_1) d\nu_{(p,0),t_1}\right)^{m/2},
\end{equation}
for $t\in [\beta t_1,0]$.
\end{corollary}
\begin{proof}
By H\"older's inequality for any $p\geq m\geq 2$ and $s > 1$ such that $\frac{m}{p}+\frac{1}{s}=1$ we obtain
\begin{equation}\label{eqn:Holder_weight}
\begin{aligned}
&\int_X v^m(\cdot,t) e^{\alpha f(\cdot,t)} d\nu_{(p,0),t} \leq \left( \int_M v^p(\cdot,t) d\nu_{(p,0),t} \right)^{m/p}\left(\int_X e^{\alpha s f(\cdot,t)} d\nu_{(p,0),t} \right)^{1/s}.
\end{aligned}
\end{equation}
Now, for any $t\in (t_1,0)$ choose $p>2$ such that 
$$\frac{t_1}{t} = p-1 \quad \Longleftrightarrow \quad p=1+\frac{t_1}{t}\quad \Longleftrightarrow \quad s=\frac{1+t_1/t}{t_1/t-m+1}.$$
Then, by Theorem \ref{thm:hypercontractivity}
\begin{equation*}
\left( \int_M v^p(\cdot,t) d\nu_{(p,0),t}\right)^{m/p} \leq  \left(\int_M v^2(\cdot,t_1) d\nu_{(p,0),t_1}\right)^{m/2}, 
\end{equation*}
hence \eqref{eqn:Holder_weight} gives
\begin{equation}
\begin{aligned}
&\int_M v^m(\cdot,t) e^{\alpha f(\cdot,t)} d\nu_{(p,0),t}\leq \left(\int_M v^2(\cdot,t_1) d\nu_{(p,0),t_1}\right)^{m/2}\left(\int_X e^{\alpha s f(\cdot,t)} d\nu_{(p,0),t}\right)^{1/s}.
\end{aligned}
\end{equation}
In particular, if $(M,g(t))_{t\in [t_1,0]}$ satisfies (RF1) and (RF3), by Lemma \ref{lemma:int_ker_bounds} we have
$$\int_M e^{\alpha s f} d\nu_{(p,0),t}\leq C(n,C_I,C_1|\alpha s)$$
as long as $\alpha s<1$.

Thus, it suffices to choose $\beta\in (0,1]$ small enough so that for every $t\in [\beta t_1,0]$
\begin{equation*}
 s=\frac{1+t_1/t}{t_1/t-m+1} \leq \frac{1}{\alpha^{1/2}} \quad\Longleftrightarrow \quad |t|\leq \frac{\alpha^{-1/2}-1}{(m-1)\alpha^{-1/2}+1} \cdot |t_1|.
\end{equation*}
The result follows by choosing $\beta=\frac{\alpha^{-1/2}-1}{(m-1)\alpha^{-1/2}+1}$, since then $\alpha s\leq \alpha^{1/2}<1$.
\end{proof}

\subsection{Gradient estimates}
\begin{lemma}\label{lemma:space_time_Poincare}
Let $(M,g(t),p)_{t\in[-10,0]}$ be a smooth complete Ricci flow satisfying (RF1), and let  $\varepsilon>0$. 
Let $v=(v^1,\ldots,v^k):M \times [-10,0]\rightarrow \mathbb R^k$, $k\geq 1$,  be a solution to the heat equation with
\begin{equation}\label{eqn:poinc_tel_assumption1}
\frac{4}{3}\int_{-1}^{-1/4} \int_M  \langle\nabla v^a, \nabla v^b\rangle  d\nu_{(p,0),t} dt =\delta^{ab},
\end{equation}
for every $a,b=1\ldots k$ and suppose that for any $a=1,\ldots,k$
\begin{equation}\label{eqn:poin_tel_assumption4}
\int_{-10}^{-\varepsilon} \int_M |\hess v^a|^2 d\nu_{(p,0),t} dt <\delta \leq \delta(n,C_I|\varepsilon).
\end{equation}

Then, for any $a,b=1,\ldots,k$ 
\begin{equation}
\int_{-1}^{-\varepsilon} \int_M \left| \langle \nabla v^a,\nabla v^b \rangle - \delta^{ab} \right|^2 d\nu_{(p,0),t} dt <\varepsilon.
\end{equation}
\end{lemma}
\begin{proof}
By the mean value theorem, for each $a,b$, there is $\bar t_{ab}\in [-1,-1/4]$ so that
$$\int_M \langle \nabla v^a,\nabla v^b\rangle (\cdot,\bar t_{ab}) d\nu_{(p,0),\bar t_{ab}} =\delta^{ab}.$$
Since, 
$$\left(\frac{\partial}{\partial t}-\Delta\right)\langle \nabla v^a,\nabla v^b\rangle = -2\langle \hess v^a,\hess v^b\rangle,$$
we obtain
\begin{align*}
\frac{d}{dt}\int_M \langle \nabla v^a,\nabla v^b\rangle d\nu_{(p,0),t}= -2\int_M \langle \hess v^a,\hess v^b\rangle d\nu_{(p,0),t}.
\end{align*}
Integrating, we conclude that for every $t\in [-10,-\varepsilon]$
\begin{align*}
c^{ab}(t):=&\int_M \langle \nabla v^a,\nabla v^b \rangle(\cdot,t) d\nu_{(p,0),t}=\\
&=\int_M \langle \nabla v^a,\nabla v^b \rangle(\cdot, \bar t_{ab}) d\nu_{(p,0),t_{ab}} - 2\int_{\bar t_{ab}}^t \int_M \langle\hess v^a,\hess v^b\rangle d\nu_{(p,0),s} ds,\\
&=\delta^{ab} - 2\int_{\bar t_{ab}}^t \int_M \langle\hess v^a,\hess v^b\rangle d\nu_{(p,0),s} ds.
\end{align*}
hence, by Young's inequality and  \eqref{eqn:poin_tel_assumption4},  for every $t\in [-10,-\varepsilon]$
\begin{equation}
\begin{aligned}\label{eqn:ct_1}
|c^{ab}(t)-\delta^{ab}|&\leq 2\int_{-1}^{-\varepsilon} \int_M  \left|\langle\hess v^a,\hess v^b\rangle \right| d\nu_{(p,0),s}ds\\
& \leq \int_{-1}^{-\varepsilon} \int_M |\hess v^a|^2 d\nu_{(p,0),s} ds +\int_{-1}^{-\varepsilon} \int_M |\hess v^b|^2 d\nu_{(p,0),s}ds\\
&\leq 20\delta,
\end{aligned}
\end{equation}
for every $t\in [-10,-\varepsilon]$.

Using the mean value inequality, we can also find $\tilde t\in [-10,-9]$, such that for each $a$
\begin{equation}
\int_M |\hess v^a|^2(\cdot,\tilde t) d\nu_{(p,0),\tilde t} \leq \int_{-10}^{-9}\int_M |\hess v^a|^2 d\nu_{(p,0),t}dt < \delta
\end{equation}

Therefore, by (RF1), Corollary \ref{cor:subsolutions} and the hypercontractivity Theorem \ref{thm:hypercontractivity}, for every $t\in [\tilde t,-\varepsilon]$ and $\sigma(t)=1+\frac{\tilde t}{t}\geq 2$ we obtain, if $0<\delta\leq 1/20$
\begin{equation}\label{eqn:sigma_estimates}
\begin{aligned}
\left(\int_M |\nabla v^a|^{\sigma(t)}(\cdot,t) d\nu_{(p,0),t}\right)^{1/\sigma(t)}& \leq  \left(\int_M |\nabla v^a|^2(\cdot,\tilde t) d\nu_{(p,0),\tilde t}\right)^{1/2} \leq 2, \\
\left(\int_M |\hess v^a|^{\sigma(t)}(\cdot,t) d\nu_{(p,0),t}\right)^{1/\sigma(t)} &\leq \frac{|\tilde t|^{E/2}}{|t|^{E/2}} \left(\int_M |\hess v^a|^{2}(\cdot,\tilde t) d\nu_{(p,0),\tilde t}\right)^{1/2} \leq C(n,C_I|\varepsilon).
\end{aligned}
\end{equation}

\noindent \textbf{Claim:}  If $0<\delta\leq 1/20$ we have
 \begin{equation}
\int_{-1}^{-\varepsilon} \int_M |\nabla \langle\nabla v^a,\nabla v^b \rangle |^2  d\nu_{(p,0),t}dt \leq  C \delta^{1/4},
\end{equation}
where $C=C(n,C_I | \varepsilon)<+\infty$.
\begin{proof}
Observe that $\nabla_k \langle\nabla v^a,\nabla v^b \rangle =\nabla_k\nabla_m v^a \nabla_m v^b + \nabla_m v^a \nabla_k \nabla_m v^b$, which gives for each $a,b=1,\ldots,k$
$$|\nabla \langle\nabla v^a,\nabla v^b\rangle |^2 \leq 2 ||\hess v||^2 ||\nabla v||^2,$$
where $||\hess v||=\max_a |\hess v^a|$ and $||\nabla v||=\max_a |\nabla v|$.

Take $p=4$ and define $q,\bar q>1$ by the relations $\frac{1}{p}+\frac{1}{q}=1$ and $\frac{2}{p}+\frac{1}{\bar q}=1$, so that $q=\frac{p}{p-1}=\frac{4}{3}$ and $\bar q = \frac{p}{p-2}=2$. Applying H\"older's inequality twice, we then estimate for each $a,b=1,\ldots,k$
\begin{align*}
&\int_{-1}^{-\varepsilon}\int_M |\nabla \langle\nabla v^a,\nabla v^b \rangle |^2 d\nu_{(p,0),t} dt\leq \\
&\leq 4 \int_{-1}^{-\varepsilon}\int_M ||\hess v||^2 ||\nabla v||^2  d\nu_{(p,0),t} dt \\
&= 4\int_{-1}^{-\varepsilon}\int_M ||\hess v||^{1/2} ||\hess v||^{3/2} ||\nabla v||^2 d\nu_{(p,0),t}dt\\
&\leq 4 \left( \int_{-1}^{-\varepsilon}\int_M ||\hess v||^2 d\nu_{(p,0),t}dt \right)^{1/4} \left( \int_{-1}^{-\varepsilon}\int_M ||\hess v||^2 ||\nabla v||^{8/3} d\nu_{(p,0),t}dt\right)^{3/4}\\
&\leq 4 \left( \int_{-1}^{-\varepsilon}\int_M ||\hess v||^2 d\nu_{(p,0),t}dt \right)^{1/4} \\& \left( \int_{-1}^{-\varepsilon} \int_M ||\hess v||^4 d\nu_{(p,0),t} dt \right)^{1/2} 
\left(\int_{-1}^{-\varepsilon} \int_M ||\nabla v||^{16/3} d\nu_{(p,0),t}dt \right)^{3/8}.
\end{align*}
Since $|\tilde t|\geq 9$, for any $t\in [-1,0)$, $\sigma(t)=1+\frac{|\tilde t|}{|t|}\geq 1+|\tilde t| \geq 10\geq \max\left(4,\frac{16}{3}\right)$. Therefore, by \eqref{eqn:sigma_estimates} and H\"older's inequality we obtain
\begin{equation*}
\int_{-1}^{-\varepsilon} \int_M |\nabla \langle \nabla v^a,\nabla v^b\rangle |^2 d\nu_{(p,0),t} dt \leq C(n,C_I|\varepsilon) \delta^{1/4},
\end{equation*}
which proves the claim.
\end{proof}

Applying the Poincar\'e inequality Theorem \ref{thm:poincare} at each $t\in [-1,-\varepsilon]$ we obtain
\begin{align*}
\int_{-1}^{-\varepsilon}\int_M \left| \langle \nabla v^a, \nabla v^b\rangle-c^{ab}\right|^2 d\nu_{(p,0),t} dt &\leq \int_{-1}^{-\varepsilon} 2|t| \int_M |\nabla \langle \nabla v^a, \nabla v^b\rangle |^2 d\nu_{(p,0),t}dt,\\
&\leq 2 \int_{-1}^{-\varepsilon}\int_M |\nabla \langle \nabla v^a, \nabla v^b\rangle |^2 d\nu_{(p,0),t}dt.
\end{align*}
Hence, by the Claim, if $0<\delta\leq 1/20$ 
\begin{equation}
\int_{-1}^{-\varepsilon}\int_M \left|\langle\nabla v^a,\nabla v^b \rangle -c^{ab}\right|^2 d\nu_{(p,0),t}dt \leq 2C\delta^{1/4}.
\end{equation}

Finally, using \eqref{eqn:ct_1} and the $L^2$-triangle inequality we obtain
\begin{align*}
&\left(\int_{-1}^{-\varepsilon} \int_M \left| \langle\nabla v^a,\nabla v^b \rangle  -\delta^{ab} \right|^2 d\nu_{(p,0),t} dt \right)^{1/2} \leq\\
&\leq \left(\int_{-1}^{-\varepsilon} \int_M \left| \langle\nabla v^a,\nabla v^b \rangle -c^{ab} \right|^2 d\nu_{(p,0),t}dt \right)^{1/2} +\left(\int_{-1}^{-\varepsilon} \int_M \left|c^{ab}-\delta^{ab}\right|^2 d\nu_{(p,0),t}dt \right)^{1/2}\\
&\leq 2^{1/2}C^{1/2} \delta^{1/8}+20\delta<\varepsilon,
\end{align*}
provided $\delta>0$ are chosen small enough.
\end{proof}

\begin{lemma}[Pointwise gradient estimate]\label{lemma:heat_gradient_bound}
Let $(M,g(t))_{t\in [-1,0]}$ be a smooth complete Ricci flow satisfying (RF1), (RF3), (RF4), (RF5), and $p\in M$. Then, for any $\varepsilon>0$  and $0<\delta\leq\delta(n,C_I,H,K|\varepsilon)$, if $v:M\times [-1,0]\rightarrow \mathbb R$ is a solution to the heat equation satisfying 
\begin{align}
\int_{-1}^{-\delta}\int_M \left| |\nabla v|^2 -1\right|^2 d\nu_{(p,0),t} dt&<  \delta,\label{eqn:L2gradone}
\end{align}
Then, there is $\gamma=\gamma(n,H)$ such that $|\nabla v|^2 (x,t)<1+\varepsilon$, for every $t\in [-\gamma,0]$ and $x\in B\left(p,t,\frac{1}{\varepsilon}\right)$.
\end{lemma}
\begin{proof}
Since $v$ is a solution to the heat equation, we have
\begin{equation}\label{eqn:evol_gradv}
\left(\frac{\partial}{\partial t} - \Delta\right) |\nabla v|^2 = - 2|\hess v|^2.
\end{equation}
Now, let $\gamma \in (0,1/4)$ be a constant that will be specified in the course of the proof, and let $t_0\in [-\gamma,0]$, $x_0 \in B\left(p,t_0,\frac{1}{\varepsilon}\right)$ and $u_{(p,0)}=(4\pi|t|)^{-n/2} e^{-f}$ be the backwards conjugate heat kernel starting at $(p,0)$. 

Now, by \eqref{eqn:evol_gradv}, we have that for any $-1\leq \bar t< t_0$
\begin{equation}\label{eqn:nablav_estimate_1}
\begin{aligned}
|\nabla v|^2(x_0,t_0) &\leq \int_M |\nabla v|^2(\cdot, \bar t) d\nu_{(x_0,t_0),\bar t}\leq 1+ \int_M \left| |\nabla v|^2(\cdot,\bar t) -1 \right| d\nu_{(x_0,t_0),\bar t}.
\end{aligned}
\end{equation}
Since $t_0 \geq -\gamma$, if $-1\leq \bar t\leq -2\gamma$, we have $t_0 - \bar t\geq \gamma$ and
$\frac{1}{2}\leq \frac{|t_0 - \bar t|}{|\bar t|}\leq  1$.
Under assumptions (RF3) and (RF4), we can apply Lemma \ref{lemma:compare_kernels} for $\beta=\frac{1}{2}$, $\alpha=\frac{4C_1-C_2}{4C_1}$, and then H\"older's inequality, to obtain
\begin{equation}\label{eqn:gr_change_center}
\begin{aligned}
\int_M \left||\nabla v|^2 -1\right| d\nu_{(x_0,t_0),\bar t} &\leq Ce^{\frac{d_{g(\bar t)} (p,x_0)^2}{C|\bar t|}} \int_M \left||\nabla v|^2 -1 \right| e^{\alpha f} d\nu_{(p,0),\bar t}\\
&\leq C(n,H,K|\varepsilon) \int_M \left||\nabla v|^2 -1 \right| e^{\alpha f} d\nu_{(p,0),\bar t}
\end{aligned}
\end{equation}
since
$
(d_{g(\bar t)} (p,x_0))^2 \leq 2(d_{g(t_0)}(p,x_0))^2 +2K^2\leq 2\varepsilon^{-2} + 2K^2,
$
due to the distance distortion estimate (RF5).

Therefore, combining \eqref{eqn:nablav_estimate_1} with  \eqref{eqn:gr_change_center}, and applying the mean value theorem in $[-4\gamma,-2\gamma]$, we obtain
\begin{equation}\label{eqn:nablav_estimate_2}
|\nabla v|^2(x_0,t_0) \leq 1+C(n,H,K|\varepsilon) \int_{-4\gamma}^{-2\gamma} \int_M \left||\nabla v|^2 -1 \right| e^{\alpha f} d\nu_{(p,0), t} d t.
\end{equation}

Now, in order to estimate the integral on the right-hand side of \eqref{eqn:nablav_estimate_2},  apply H\"older's inequality twice, first with $\frac{1}{s_1}+\frac{1}{s_2}=1$ for $s_2=\alpha^{-1/2}$ and then for $\frac{1}{2}+\frac{1}{2}=1$, to obtain by \eqref{eqn:L2gradone}
\begin{equation}\label{eqn:J}
\begin{aligned}
&\int_{-4\gamma}^{-2\gamma} \int_M \left||\nabla v|^2 -1 \right| e^{\alpha f} d\nu_{(p,0), t}d t \leq \\
&\leq \left(\int_{-4\gamma}^{-2\gamma} \int_M \left||\nabla v|^2 -1 \right| d\nu_{(p,0), t}d t \right)^{1/s_1} \left(\int_{-4\gamma}^{-2\gamma} \int_M \left||\nabla v|^2 -1 \right| e^{\sqrt{\alpha} f} d\nu_{(p,0), t} d t \right)^{1/s_2}\\
&\leq \delta^{1/2s_1} \left(\int_{-4\gamma}^{-2\gamma} \int_M \left||\nabla v|^2 -1 \right| e^{\sqrt{\alpha} f} d\nu_{(p,0), t} d t \right)^{1/s_2} = \delta^{1/2s_1} I^{1/s_2}
\end{aligned}
\end{equation}
where we set $I=\int_{-4\gamma}^{-2\gamma} \int_M \left||\nabla v|^2 -1 \right| e^{\sqrt{\alpha} f} d\nu_{(p,0), t} d t $.

To estimate the integral $I$, we first apply the $L^1$ triangle inequality and H\"older's inequality once more to obtain
\begin{equation}\label{eqn:I}
\begin{aligned}
I&\leq \underbrace{\int_{-4\gamma}^{-2\gamma} \int_M |\nabla v|^2 e^{\sqrt \alpha f} d\nu_{(p,0), t}dt}_{II} +\underbrace{\int_{-4\gamma}^{-2\gamma}\int_M e^{\sqrt\alpha f} d\nu_{(p,0), t} dt.}_{III}
\end{aligned}
\end{equation}
Since $\frac{\partial}{\partial t} |\nabla v| \leq \Delta |\nabla v|$, by Corollary \ref{cor:L2toL2weighted}, we can chooce $\gamma=\gamma(n,H)\in (0,1/4)$ so that
$$II=\int_{-4\gamma}^{-2\gamma} \int_M |\nabla v|^2 e^{\sqrt\alpha f} d\nu_{(p,0), t}dt\leq 4\gamma \int_{-1}^{-1/2}\int_M |\nabla v|^2 d\nu_{(p,0), t} dt.$$
Then, we can estimate, if $\delta>0$ is small enough,
\begin{align*}
\int_{-1}^{-1/2} \int_M |\nabla v|^2 d\nu_{(p,0), t} dt &\leq \frac{1}{2} + \int_{-1}^{-1/2} \int_M \left| |\nabla v|^2 -1 \right| d\nu_{(p,0), t}dt\\
&\leq \frac{1}{2} +\frac{1}{\sqrt 2}\left(  \int_{-1}^{-1/2} \int_M \left| |\nabla v|^2 -1 \right|^2 d\nu_{(p,0), t} dt\right)^{1/2}\\
&\leq \frac{1}{2} +\frac{1}{\sqrt 2} \delta^{1/2} \leq 1
\end{align*}
Therefore, $II \leq 4\gamma$, and by Lemma \ref{lemma:int_ker_bounds} we obtain that $III\leq C(n,C_I,C_1)$, hence $I\leq C(n,C_I,H)$.

It follows, by \eqref{eqn:nablav_estimate_2}, \eqref{eqn:J} and \eqref{eqn:I}, that $|\nabla v|^2(x_0,t_0)<1+\varepsilon$, if $0<\delta\leq \delta(n,C_I,H,K|\varepsilon)$.

\end{proof}

\section{The heat equation on shrinking solitons}\label{sec:heat_gsrs}
\subsection{The heat equation on a shrinking gradient Ricci soliton}
Let $(M,\bar g,\bar f)$ a shrinking gradient Ricci soliton, namely
\begin{equation}
\ric(\bar g)+\hess_{\bar g} \bar f =\frac{\bar g}{2}.
\end{equation}
Let $\varphi_t:M\rightarrow M$, $t<0$, the $1$-parameter family of diffeomorphisms that evolves by
\begin{equation}
\frac{d}{dt} \varphi_t = \frac{1}{|t|} \nabla^{\bar g} \bar f \circ \varphi_t.
\end{equation}
As we saw  in Subsection \ref{subsection:gsrs} that $g(t) = |t| \varphi_t^* \bar g$, $t<0$, is a Ricci flow on $M$.

Let $I\subset \mathbb R$ be an interval and $(v(\cdot,t))_{t\in I}$, $\sup I=0$, be a solution to the heat equation on $(M, g(t))_{t<0}$. 
Define $\bar v(s)=\bar v (\cdot,s)$ by the relation $v(x,t) = |t|^{1/2}\bar v (\varphi_t(x),s)$,
where $s=-\log(-t)$. Then, by direct computation, we obtain

\begin{equation}\label{eqn:drift_heat}
\frac{\partial \bar v}{\partial s} = \left(\Delta_{\bar f} +\frac{1}{2}\right)\bar v,
\end{equation}
where $\Delta_{\bar f} \bar v=\Delta_{\bar g} \bar v - \langle \nabla^{\bar g} \bar v, \nabla^{\bar g} \bar f\rangle$ is the associated drift Laplacian.

\subsubsection{The spectrum of $\Delta_{\bar f} +\frac{1}{2}$}
The following theorem is the combination of Theorems 1 and 2 in \cite{ChengZhou}.

\begin{theorem}\label{thm:spectrum}
Let $(M,\bar g,\bar f)$ be a normalized gradient shrinking Ricci soliton at scale $1$, $\bar\nu$ be the measure with $d\bar\nu=(4\pi)^{-n/2} e^{-\bar f} d\vol_{\bar g}$. Denote by $L^2_{\bar\nu}$, $H^1_{\bar\nu}$ the associated Lebesgue and Sobolev spaces.  Then, $\Delta_{\bar f}+\frac{1}{2}: L^2_{\bar\nu}\rightarrow L^2_{\bar\nu}$ has discrete spectrum 
$\frac{1}{2}=\lambda_0 >  0\geq \lambda_1\geq \lambda_2\geq \cdots$. Moreover, 
\begin{enumerate}
\item $L^2_{\bar\mu}$ has a countable orthonormal basis $\{\phi_i\}_{i=0}^\infty \subset L^2_{\bar\nu}$ that consists of eigenfunctions of $\Delta_{\bar f}+\frac{1}{2}$, namely $\left(\Delta_{\bar f}+\frac{1}{2}\right)\phi_i=\lambda_i \phi_i$. 
\item Given any $v\in H^1_{\bar \mu}$, there are unique $b_i=(v,\phi_i)_{L^2_{\bar\mu}}$, such that $v=\sum_{i=0}^{+\infty} b_i \phi_i$, where the convergence of the series is in $H^1_{\bar\mu}$.
\item The eigenspace corresponding to the eigenvalue $\frac{1}{2}$ consists of the constant functions. 
\item The eigenvalue $0$ appears with multiplicity $k\geq 1$, namely $\lambda_1=\cdots=\lambda_k=0$, if and only if $1\leq k \leq n$ and $(M,\bar g)$ is isometric to a product $(M'\times \mathbb R^k, g'\oplus g_{\mathbb R^k})$. In this case, up to isometry, we can express
$$\bar f(x,a)=\frac{|a|^2}{4}+f'(x),$$
where $f'(x)=\bar f(x,0)$ for any $x\in M'$ and $(M',g',f')$ is a normalized gradient shrinking Ricci soliton at scale $1$ and $\lambda_1(\Delta_{f'}+\frac{1}{2})<0$. Moreover, we can choose $\phi_i=x^i$, $i=1,\ldots,k$, where $x^i$ are the linear functions on $M'\times\mathbb R^k$ induced by the coordinate functions of $\mathbb R^k$.
\item If $\phi$ is an eigenfunction of $\Delta_{\bar f}+\frac{1}{2}$ with eigenvalue $\lambda\leq 0$ then
$$\int_M |\hess \phi|^2 d\bar\nu \leq  |\lambda| \int_M |\nabla \phi|^2 d\bar \nu<+\infty.$$
\end{enumerate}
\end{theorem}
\begin{proof}
The theorem is a combination of Theorems 1 and 2 in \cite{ChengZhou}. The only assertion that remains to prove is that if $v\in H^1_{\bar \nu}$ then the series $\sum_{i=0}^\infty b_i \phi_i$ converges in $H^1_{\bar \nu}$.

To see this, let $v_k = \sum_{i=0}^k b_i \phi_i$ and compute
\begin{align*}
0\leq\int_M |\nabla v - \nabla v_k|^2 d\bar \nu &= \int_M |\nabla v|^2 d\bar \nu + \int_M |\nabla v_k|^2 d\bar \nu -2\int_M \langle\nabla v_k,\nabla v\rangle d\bar \nu\\
&=\int_M |\nabla v|^2 d\bar \nu  +  \sum_{i=0}^{k} b_i^2 \lambda_i.
\end{align*}
Since $v\in H^1_{\bar \nu}$, it follows that
\begin{equation}\label{eqn:bi2li}
\sum_{i=0}^k b_i^2 |\lambda_i| \leq \int_M |\nabla v|^2 d\bar \nu<+\infty.
\end{equation}

For $k>l$ we obtain, integrating by parts,
\begin{equation*}
\int_M |\nabla v_k -\nabla v_l|^2 d\bar\nu = \sum_{i=l}^k b_i^2 |\lambda_i|,\end{equation*}
therefore, by \eqref{eqn:bi2li} the sequence $\nabla v_k$ is Cauchy in $L^2_{\bar \nu}$.

On the other hand, for every compactly supported smooth vector field $X$ on $M$, integrating by parts we obtain
\begin{align*}
\int_M \langle\nabla v_k, X\rangle d\bar \nu = -\int_M v_k \dive_f X d\bar \nu \longrightarrow -\int_M v \dive_f X d\bar \nu = \int_M \langle\nabla v,X\rangle d\bar \nu,
\end{align*}
where $\dive_f X:= \dive X - \langle \nabla f, X\rangle$. It follows that $\nabla v_k$ converges to $\nabla v$ in $L^2_{\bar \nu}$. 
\end{proof}

\begin{lemma}\label{lemma:proj_evol}
Let $(M,\bar g,\bar f)$ be a normalized gradient shrinking Ricci soliton at scale $1$, $\bar\nu$ be the measure induced by $d\bar\nu=(4\pi)^{-n/2} e^{-\bar f}$. Let $\bar v\in C^\infty(M\times I)$ be a smooth solution to \eqref{eqn:drift_heat}, such that $\max_{s\in [s_1,s_2]} ||\bar v(s)||_{H^1_{\bar \nu}}\leq C(s_1,s_2)$ for every $[s_1,s_2]\subset I$. Then,  if $\phi$ is an $L^2_{\bar\nu}$ unit-normalized eigenfunction of $\Delta_{\bar f}+\frac{1}{2}$, with associated eigenvalue $\lambda$, the function
 $I\ni s\mapsto (\bar v(s),\phi)_{L^2_\nu}$  is absolutely continuous and
$$\frac{d}{ds} (\bar v(s),\phi)_{L^2_{\bar\nu}} = \lambda (\bar v(s),\phi)_{L^2_{\bar\nu}}.$$
Therefore, in the $H^1_{\bar\nu}$ sense, for any $s,s_0\in I$,
$$\bar v(s) = \sum_{i=0}^{+\infty} (v(s_0),\phi_i)_{L^2_{\bar\nu}} e^{\lambda_i(s-s_0)} \phi_i.$$
\end{lemma}
\begin{proof}
Let $\eta_j$ be a sequence of compactly supported smooth functions which converge in $H^1_{\bar\nu}$ to $\phi$. Then
\begin{equation*}
\frac{d}{ds} \int_M \bar v(s) \eta_j d\bar\nu= \int_M \left(\Delta_{\bar f} +\frac{1}{2}\right) \bar v(s) \eta_j d\bar\nu=-\int_M \left(\langle \nabla \bar v(s),\nabla \eta_j\rangle -\frac{1}{2}\bar v(s) \eta_j \right)d\bar\nu.
\end{equation*}
Therefore,  for any $s_1<s_2$ in $I$
\begin{equation}
\int_M \bar v(s_2) \eta_j d\bar\nu = \int_M \bar v(s_1) \eta_j d\bar\nu -\int_{s_1}^{s_2} \int_M \left(\langle \nabla \bar v(s),\nabla \eta_j\rangle -\frac{1}{2}\bar v(s) \eta_j \right)d\bar\nu.
\end{equation}
Taking the limit in $j$, this gives
\begin{equation}
(\bar v(s_2),\phi)_{L^2_{\bar \nu}} = (\bar v(s_1),\phi)_{L^2_{\bar \nu}}-\int_{s_1}^{s_2} \int_M \left(\langle \nabla \bar v(s),\nabla \phi\rangle -\frac{1}{2}\bar v(s) \phi \right)d\bar\nu ds.
\end{equation}
It follows that $(\bar v(s),\phi)_{L^2_{\bar\nu}}$ is absolutely continuous and for almost every $s\in I$
\begin{equation}
\frac{d}{ds}(\bar v(s),\phi)_{L^2_{\bar\nu}}=\int_M \bar v(s) \left(\Delta_{\bar f}  +\frac{1}{2}\right)\phi d\bar\nu=\lambda (\bar v(s),\phi)_{L^2_{\bar\nu}}.
\end{equation}

\end{proof}

\subsection{Spectral gap and consequences}
In this section, let $(M,g(t))_{t\in (-\infty,0)}$ be a smooth complete Ricci flow that satisfies (RF1), and suppose that it is induced by a normalized gradient shrinking Ricci soliton $(M,\bar g,\bar f)$ at scale $1$ that satisfies  $-\Lambda \leq \bar\mu(\bar g)\leq 0$. Let $\nu_f\in\mathcal S$ be the conjugate heat flow induced by $\bar f$, namely $f(\cdot,t)=\bar f_t$, as in Subsection \ref{subsection:gsrs}, and suppose that $\nu_f$ satisfies (CHF1) with respect to $p\in M$.  

\begin{lemma}\label{lemma:gap}
Suppose that there is $q\in B(p,-1,R)$ such that $(M,g(t),q)_{t\in (-\infty,0)}$ is $k$-selfsimilar but not $(k+1,\eta)$-selfsimilar for some $1\leq k \leq n$ and $\eta>0$. Then, if $\lambda_i$, $i=0,1,\ldots$, are the eigenvalues of $\Delta_{\bar f}+\frac{1}{2}$, there is $\delta=\delta(n,C_I, \Lambda,H,R,\eta)>0$ such that $\lambda_i< -\delta$, for any $i\geq k+1$.
\end{lemma}
\begin{proof}
Suppose that there is a sequence $(M_j,g_j(t),p_j)_{t\in (-\infty,0)}$ of smooth complete Ricci flows satisfying (RF1) such that, for each $j$
\begin{itemize}
\item[A.] $(M_j,g_j(t))_{t\in (-\infty,0)}$ is induced by a normalized gradient shrinking Ricci soliton $(M,\bar g_j, \bar f_j)$ at scale $1$ with $-\Lambda \leq \mu(\bar g_j) \leq 0$.
\item[B.] there is a $q_j \in B(p_j, -1,R)$ such that $(M_j,g_j(t),q_j)_{t\in (-\infty,0)}$ is $k$-selfsimilar with spine $\mathcal S_j$ but not $(k+1,\eta)$-selfsimilar.
\item[C.] there is a sequence of conjugate heat flows $\nu_{f_j}\in\mathcal S_j$ such that each $\nu_{f_j}$ satisfies (CHF1) with respect to $p_j$. Moreover, $\bar g_j=g_j(-1)$, $\bar f_j(\cdot)=f_j(\cdot,-1)$.
\item[D.] there  is an eigenvalue $\lambda_j<0$ of $\Delta_{\bar f_j}+\frac{1}{2}$, such that $\lambda_j\rightarrow 0$. Let $d\bar \nu_j = (4\pi)^{-n/2} e^{-\bar f_j} d\vol_{\bar g_j}$ and suppose that $\phi_j\in L^2_{\bar \nu_j}$ is an eigenfunction for $\lambda_j$. 
\end{itemize}

By (RF1), Assumption A and Proposition \ref{prop:compactness_rf}, we can assume that $(M_j,g_j(t),p_j)_{t\in (-\infty,0)}$ converges to a smooth complete pointed Ricci flow $(M_\infty,g_\infty(t),p_\infty)_{t\in (-\infty,0)}$.
Moreover, by Assumption B and Lemma \ref{lemma:k_delta_convergence}, we may assume that $q_j\rightarrow q_\infty \in M_\infty$ so that $(M_\infty,g_\infty(t),q_\infty)_{t\in (-\infty,0)}$ is $k$-selfsimilar with spine $\mathcal S_\infty$, but not $k+1$-selfsimilar. Therefore $(M_\infty,g_\infty(t))_{t\in (-\infty,0)}$ splits exactly $k$-Euclidean factors.

Now, by Assumption C, we can further assume that the conjugate heat flows $\nu_{f_j}$ converge to some conjugate heat flow $\nu_{f_\infty}$ on $(M_\infty,g_\infty(t))_{t\in (-\infty,0)}$. Since $\nu_{f_j}\in\mathcal S_j$ we also get that $\nu_{f_\infty}\in\mathcal S_\infty$. Denote by $(M_\infty,\bar g_\infty,\bar f_\infty)$ the associated gradient shrinking Ricci soliton at scale $1$, namely $\bar g_\infty=g_\infty(-1)$, $\bar f_\infty=f_\infty(\cdot,-1)$ and set $d\bar \nu_\infty=(4\pi)^{-n/2} e^{-\bar f_\infty}$. 

Since each $(M_j,g_j(t),q_j)_{t\in (-\infty,0)}$ is $k$-selfsimilar, but not $(k+1,\eta)$-selfsimilar, $(M_j,\bar g_j,\bar f_j)$ splits exactly $k$ Euclidean factors. Therefore, for each $j$, we know by Theorem \ref{thm:spectrum} that the kernel of $\Delta_{\bar f_j}+\frac{1}{2}$ is spanned by $k$ eigenfunctions $v_j^a:M_j \rightarrow \mathbb R$, $a=1,\ldots,k$ which satisfy $\hess_{\bar g_j} v_j^a=0$ and can be normalized so that for any $a,b=1,\ldots,k$
\begin{equation}
\frac{1}{2}\int_{M_j} v_j^a v_j^b d\bar \nu_j=-\int_{M_j} v_j^a \Delta_{\bar f_j} v_j^b d\bar \nu_j=\int_{M_j} \langle\nabla v_j^a,\nabla v_j^b\rangle d\bar \nu_j = \delta^{ab}.
\end{equation}
In particular, since the vector fields $\nabla v_j^a$ are parallel, it follows that $\langle\nabla v_j^a,\nabla v_j^b\rangle = \delta^{ab}$
on $M_j$.

Moreover, since $\hess_{\bar g_j} v_j^a=0$, we can assume, passing to a further subsequence, that $v_j^a$ smoothly converge to $v_\infty^a$, such that  $\hess_{\bar g_\infty} v_\infty^a=0$ and $\langle\nabla v_\infty^a,\nabla v_\infty^b\rangle = \delta^{ab}$.

By Theorem \ref{thm:spectrum}, the eigenfunctions $\phi_j$ satisfy
\begin{align}
\int_{M_j} \langle \nabla v_j^a,\nabla \phi_j\rangle d\bar \nu_j&=0,\\
\int_{M_j} |\hess_{\bar g_j} \phi_j|^2 d\bar \nu_j &\leq |\lambda_j| \int_M |\nabla \phi_j|^2 d\bar \nu_j. \label{eqn:hess_phi_j}
\end{align}
and $|\nabla \phi_j|\in L^2_{\bar \nu_j}$.

We also compute
$$\nabla_i \sqrt{|\nabla \phi_j|^2+\varepsilon} = \frac{\nabla_i |\nabla \phi_j|^2}{2\sqrt{|\nabla \phi_j|^2+\varepsilon}}=\frac{\nabla_m \phi_j \nabla_i\nabla_m \phi_j}{\sqrt{|\nabla \phi_j|^2+\varepsilon}}$$
thus 
$$ \left| \nabla \sqrt{|\nabla \phi_j |^2+\varepsilon}\right|^2 \leq \frac{|\nabla \phi_j|^2}{|\nabla \phi_j|^2 +\varepsilon} |\hess_{\bar g_j} \phi_j|^2\leq |\hess_{\bar g_j} \phi_j|^2$$
Since, by Theorem \ref{thm:spectrum}, the non-trivial eigenvalues of $\Delta_{\bar f_j}$ are less than or equal to $-\frac{1}{2}$, it follows that
\begin{align*}
&\int_{M_j} \left( \sqrt{|\nabla \phi_j|^2 +\varepsilon} - \int_{M_j} \sqrt{|\nabla \phi_j|^2 +\varepsilon} d\bar \nu_j \right)^2 d\bar \nu_j \leq\\
&\leq 2\int_{M_j} \left|\nabla \sqrt{|\nabla \phi_j|^2+\varepsilon} \right|^2 d\bar \nu_j \leq 2\int_{M_j} |\hess_{\bar g_j} \phi_j|^2 d\bar \nu_j\\
&\leq 2|\lambda_j| \int_{M_j} |\nabla \phi_j|^2 d\bar \nu_j,
\end{align*}
where we used \eqref{eqn:hess_phi_j} in the last inequality.

Sending $\varepsilon\rightarrow0$ we thus obtain
\begin{equation}
\int_{M_j} \left(|\nabla \phi_j| - \int_{M_j} |\nabla \phi_j| d\bar \mu_j \right)^2 d\bar \nu_j \leq 2|\lambda_j| \int_{M_j} |\nabla \phi_j|^2 d\bar \nu_j.
\end{equation}
Normalizing $\phi_j$ so that $\int_{M_j} |\nabla \phi_j| d\bar \mu_j=1$ this becomes
\begin{equation}\label{eqn:nabla_phi_1}
\int_{M_j} \left(|\nabla \phi_j| - 1 \right)^2 d\bar \nu_j \leq 2|\lambda_j| \int_{M_j} |\nabla \phi_j|^2 d\bar \nu_j.
\end{equation}
By the $L^2$ triangle inequality we obtain
\begin{align*}
\left(\int_{M_j} |\nabla \phi_j |^2 d\bar \nu_j \right)^{1/2}&\leq \left( \int_{M_j}(|\nabla \phi_j| -1)^2 d\bar \nu_j \right)^{1/2}+1\\
&\leq \sqrt 2 |\lambda_j|^{1/2} \left(\int_{M_j} |\nabla \phi_j|^2 d\bar \nu_j\right)^{1/2} +1.
\end{align*}
Since $\lambda_j\rightarrow 0$, this gives
$$ \int_{M_j} |\nabla \phi_j|^2 d\bar \nu_j\leq C,$$
for large $j$. 

Applying this to  \eqref{eqn:hess_phi_j} and \eqref{eqn:nabla_phi_1} we obtain that
\begin{equation}\label{eqn:hessian_lambda}
\int_{M_j} (|\nabla \phi_j| - 1)^2 + |\hess_{\bar g_j} \phi_j|^2 d\bar \nu_j \leq C|\lambda_j|,
\end{equation}

Since $|\langle\nabla v_j^a,\nabla \phi_j\rangle| \leq |\nabla v_j^a | |\nabla \phi_j|$ and $|\nabla v_j^a|=1$ on $M_j$, we know that $\langle\nabla v_j^a,\nabla \phi_j\rangle\in L^2_{\bar \nu_j}$ because  $|\nabla \phi_j|\in L^2_{\bar \nu_j}$. Moreover, since $\hess_{\bar g_j} v_j^a = 0$,
$$|\nabla_i (\langle\nabla v_j^a,\nabla \phi_j\rangle)|=| \nabla_k v_j^a \nabla_i\nabla_k \phi_j |\leq |\nabla v_j^a| |\hess_{\bar g_j} \phi_j|=|\hess_{\bar g_j} \phi_j|,$$
so $|\nabla_i (\langle\nabla v_j^a,\nabla \phi_j\rangle)|\in L^2_{\bar \nu_j}$, by \eqref{eqn:hess_phi_j} and \eqref{eqn:hessian_lambda}.

Therefore, $\langle\nabla v_j^a,\nabla \phi_j\rangle\in H^1_{\bar \nu}$ and since the non-trivial eigenvalues of $\Delta_{\bar f_j}$ are $\leq -\frac{1}{2}$, by Theorem \ref{thm:spectrum}, and $\int_{M_j} \langle \nabla v_j^a,\nabla \phi_j \rangle d\bar \nu_j=0$ we obtain
\begin{equation}\label{eqn:ortho_lambda}
\int_{M_j} (\langle\nabla v_j^a,\nabla \phi_j\rangle)^2 d\bar \nu_j \leq 2\int_{M_j} |\nabla (\langle\nabla v_j^a,\nabla \phi_j\rangle)|^2 d\bar \nu_j\leq 2\int_{M_j} |\hess_{\bar g_j} \phi_j|^2 d\bar \nu_j \leq C|\lambda_j|\rightarrow 0.
\end{equation}

By elliptic regularity, passing to a further subsequence, we can assume that $\phi_j$ converge smoothly to a limit $\phi_\infty$ on $M_\infty$, which satisfies
\begin{align*}
|\hess_{\bar g_\infty} \phi_\infty|+|\langle \nabla v_\infty^a,\nabla \phi_\infty\rangle | + | |\nabla \phi_\infty|-1|=0 \quad \textrm{on $M_\infty$},
\end{align*}
due to \eqref{eqn:hessian_lambda} and \eqref{eqn:ortho_lambda}.

We conclude that $\nabla v_\infty^1,\ldots,\nabla v_\infty^k,\nabla \phi_\infty$ are $k+1$ parallel linearly independent vector fields on $(M_\infty, \bar g_\infty)$, thus it splits $k+1$ Euclidean factors which is a contradiction.
 \end{proof}

\subsubsection{Almost linear growth}

\begin{lemma}\label{lemma:growth_linear}
Suppose that there is $q\in B(p,-1,R)$ such that $(M,g(t),q)_{t\in (-\infty,0)}$ is $k$-selfsimilar but not $(k+1,\eta)$-selfsimilar, and that $v:M\times(-\infty,0)\rightarrow \mathbb R$ is a solution to the heat equation with $\max_{t\in [t_1,t_2]} ||v(\cdot,t)||_{H^1_{\nu_{f,t}}}\leq C(t_1,t_2)$ for every $t_1<t_2<0$.

There is $\varepsilon_0=\varepsilon_0(n,C_I,\Lambda,H,R,\eta)>0$ such that if for some $0<\varepsilon\leq \varepsilon_0$ \begin{equation}\label{eqn:integral_v_growth}
\int_M v^2 d\nu_{f,t} \leq C \left(|t|^{1+\varepsilon}+1\right),
\end{equation} 
 for every $t<0$, then $v$ is an affine function of the coordinates $x^1,\ldots,x^k$ induced by the splitting $M=M'\times \mathbb R^k$, and
 $$\int_M v d\nu_{f,t} = 0$$
 for every $t<0$.
\end{lemma}
\begin{proof}
Recall that $(M,g(t))_{t\in(-\infty,0)}$ is the Ricci flow induced by the normalized gradient shrinking Ricci soliton $(M,\bar g,\bar f)$, at scale $1$, with $\bar g=g(-1)$, $\bar f=f(\cdot,-1)$, namely there is a family of diffeomorphisms $\varphi_t:M\rightarrow M$, $\varphi_{-1}=id_M$, $\frac{d}{dt}\varphi_t=\frac{1}{|t|}\nabla^{\bar g} \bar f\circ\varphi_t$ such that $g(t)=|t|\varphi_t^{*} \bar g$ and $f(x,t)=\bar f(\varphi_t(x))$. Define the probability measure $\bar \nu$ on $M$ so that $d\bar \nu=(4\pi)^{-n/2}e^{-\bar f} d\vol_{\bar g}$ and the function $\bar v$ so that $v(x,t)=|t|^{1/2} \bar v(\varphi_t(x),s)$, $s=-\log |t|$. Clearly, for every $s_1<s_2$, $\max_{s\in [s_1,s_2]} ||\bar v(\cdot,s)||_{H^1_{\bar \nu}}\leq C(s_1,s_2)$.

Therefore,  for $t\leq -1$ and $s=-\log|t|\leq 0$, \eqref{eqn:integral_v_growth} becomes
\begin{equation*}
\int_M (\bar v(s))^2 d\bar \nu=|t|^{-1}\int_M v^2 d\nu_{f,t} \leq C|t|^{\varepsilon} = Ce^{\varepsilon |s|}.
\end{equation*}

Let's suppose that for some $i\geq k+1$, the corresponding eigenfunction $\phi_i$ of $\Delta_{\bar f}+\frac{1}{2}$ satisfies $(\bar v(0),\phi_i)_{L^2_\mu}\not=0$. 

Then, for any $s\leq 0$, by Lemma \ref{lemma:proj_evol} and Lemma \ref{lemma:gap}, we obtain
\begin{align*}
\int_M (\bar v(s))^2 d\bar \nu\geq  (\bar v(s), \phi_i)_{L^2_{\bar\nu}}^2 = e^{2\lambda_i s}(\bar v(0), \phi_i)_{L^2_{\bar\nu}} ^2 \geq e^{\delta |s|}(\bar v(0), \phi_i)_{L^2_{\bar\nu}} ^2
\end{align*}

It follows that for every $s\leq 0$
\begin{equation}
e^{\delta|s|}(\bar v(\cdot,0), \phi_i)_{L^2_{\bar\mu}} ^2 \leq C( e^{\varepsilon|s|}+1),
\end{equation}
which fails for large $|s|$, if $\varepsilon<\delta$, unless $(\bar v(\cdot,0), \phi_i)_{L^2_{\bar\mu}} ^2=0$. Therefore 
$$\bar v(s)= \sum_{i=1}^k (\bar v(0), x_i)_{L^2_{\bar\nu}} x^i+ e^{\frac{|s|}{2}}(\bar v(0),\phi_0)_{L^2_{\bar\nu}}.$$
Similarly, we conclude that $(\bar v(\cdot,0),\phi_0)_{L^2_{\bar\mu}} = 0$, and the result follows.
\end{proof}

\subsubsection{Hessian decay.}

\begin{lemma}\label{lemma:hessian_grad}
Let $(M,\bar g,\bar f)$ be a normalized gradient shrinking Ricci soliton at scale $1$, with bounded curvature, $\bar \nu$ be the probability measure induced by $d\bar \nu=(4\pi)^{-n/2} e^{-\bar f}d\vol_{\bar g}$. Let $\bar v\in C^\infty(M\times I)$ be a smooth solution to \eqref{eqn:drift_heat}, such that $\sup_{s\in [a,b]} ||\bar v(\cdot,s)||_{H^1_{\bar \nu}}\leq C(a,b)$, for any $[a,b]\subset I$.
 
Then for every $s_1<s_2$ in $I$
\begin{equation*}
2\int_{s_1}^{s_2} \int_M |\hess_{\bar g} \bar v|^2 d\bar \nu ds =  \int_M |\nabla \bar v(s_1)|^2 d\bar\nu - \int_M |\nabla \bar v(s_2)|^2 d\bar\nu<+\infty.
\end{equation*}
\end{lemma}
\begin{proof}
A direct computation using the gradient shrinking Ricci soliton equation and the Bochner formula gives
\begin{equation}
\left(\frac{\partial}{\partial s}-\Delta_{\bar f}\right)|\nabla \bar v|^2 = - 2|\hess_{\bar g} \bar v|^2.
\end{equation}
Multiplying by a smooth function $\eta\geq 0$ with compact support and integrating we obtain
\begin{align*}
\frac{d}{ds} \int_M |\nabla \bar v|^2 \eta d\bar\nu &=\int_M (\Delta_{\bar f} |\nabla \bar v|^2 - 2|\hess_{\bar g} \bar v|^2 )\eta d\bar \nu\\
&=\int_M |\nabla \bar v|^2 \Delta_{\bar f} \eta d\bar\nu - 2\int_M |\hess_{\bar g} \bar v|^2 \eta d\bar\nu.
\end{align*}
It follows that for any $s_1<s_2$ in $I$,
\begin{equation}\label{eqn:integral_evol_grad}
\begin{aligned}
&\int_M |\nabla \bar v(s_2)|^2 \eta d\bar \nu - \int_M |\nabla \bar v(s_1)|^2 \eta d\bar\nu =\\
&= \int_{s_1}^{s_2} \int_M |\nabla \bar v|^2 \Delta_{\bar f} \eta d\bar\mu  - 2\int_{s_1}^{s_2}   \int_M |\hess_{\bar g} \bar v|^2 \eta d\bar\nu ds.
\end{aligned}
\end{equation}

Now, in order to construct a cutoff function $\eta$ with nice bounds, set $\tilde f=\bar f+c$ so that 
\begin{align*}
\Delta \tilde f +R -\frac{n}{2}&=0,\\
|\nabla \tilde f|^2+R-\tilde f&=0,
\end{align*}
as in Proposition \ref{prop:soliton_identities}. Moreover,  since $R\geq 0$ on a shrinking soliton, it follows that $\tilde f\geq0$.

Define $\mathbf r=2\sqrt{\tilde f}$. Then, when $\mathbf r>0$, it is smooth and since $R\geq 0$
\begin{align*}
\nabla_i \mathbf r&=\frac{2}{\mathbf r}\nabla_i \tilde f,\\
|\nabla \mathbf r|^2&= \frac{4}{\mathbf r^2}|\nabla \tilde f|^2 = \frac{4}{\mathbf r^2}\left( -R+\tilde f \right)=1-\frac{4R}{\mathbf r^2}\leq 1,\\
\Delta \mathbf r&= \frac{2}{\mathbf r} \Delta \tilde f-\frac{2}{\mathbf r^2}\nabla_i \mathbf r\nabla \tilde f_i \\
&=2\mathbf r^{-1}\left( -R+\frac{n}{2}\right) - \mathbf r^{-1}|\nabla \mathbf r|^2.
\end{align*}

Now, let $\psi:\mathbb R\rightarrow \mathbb R$ smooth,  such that $0\leq \psi\leq 1$, $\psi=1$ in $(-\infty,1]$, $\psi=0$ in $[2,+\infty)$, such that $|\psi'|+|\psi''|\leq C$.

Define for any $\rho\geq 1$, $\eta_\rho(x)=\psi(\mathbf r(x)/\rho)$. Then, the support of $\nabla \eta_\rho$ is in the set $\{\rho\leq \mathbf r\leq 2\rho\}$ and 
\begin{align*}
\nabla_i \eta_\rho&=\frac{\psi'(\mathbf r/\rho)}{\rho} \nabla_i \mathbf r,\\
\Delta \eta_\rho&=\frac{\psi'(\mathbf r/\rho)}{\rho} \Delta \mathbf r + \frac{\psi''(\mathbf r /\rho)}{\rho^2} |\nabla \mathbf r|^2.
\end{align*}
Therefore,
\begin{align*}
|\nabla \eta_\rho|^2&=\frac{(\psi'(\mathbf r/\rho))^2}{\rho^2}  |\nabla \mathbf r|^2 \leq \frac{C}{\rho^2}\\
|\langle\nabla \eta_\rho,\nabla f\rangle|&=\frac{|\psi'(\mathbf r/\rho)|}{2\rho} \mathbf r |\nabla \mathbf r|^2\leq C\\
|\Delta \eta_\rho|&=\left| \frac{\psi'(\mathbf r/\rho)}{\rho} \Delta \mathbf r + \frac{\psi''(\mathbf r/\rho)}{\rho^2} |\nabla \mathbf r|^2\right|\leq \frac{C}{\rho^2}\\
|\Delta_{\bar f} \eta_\rho| &\leq  C(n).
\end{align*}

Thus, since $\sup_{s\in [s_1,s_2]} ||\bar v(\cdot,s)||_{H^1_{\bar \mu}}\leq C(s_1,s_2)$, we obtain that
\begin{align*}
\lim_{\rho\rightarrow+\infty}  \int_M |\nabla \bar v(s_i)|^2 \eta_\rho d\bar\nu &=  \int_M |\nabla \bar v(s_i)|^2 d\bar\nu \\
\lim_{\rho\rightarrow+\infty} 2\int_{s_1}^{s_2}  \int_M |\hess_{\bar g} \bar v|^2 \eta_\rho d\bar\nu& = 2\int_{s_1}^{s_2}  \int_M |\hess_{\bar g} \bar v|^2  d\bar\nu ds\\
\lim_{\rho\rightarrow+\infty} \left|\int_{s_1}^{s_2} \int_M |\nabla \bar v|^2 \Delta_{\bar f} \eta_\rho d\bar\nu ds \right| &=\lim_{\rho\rightarrow+\infty} \left|\int_{s_1}^{s_2} \int_{\{\rho\leq \mathbf r\leq 2\rho\}} |\nabla \bar v|^2 \Delta_{\bar f} \eta_\rho d\bar\nu ds \right|\\
&\leq C\lim_{\rho\rightarrow +\infty} \int_{s_1}^{s_2}\int_{\{\rho\leq\mathbf  r\leq 2\rho\}}|\nabla \bar v|^2  d\bar\nu ds=0.
\end{align*}

Thus, by \eqref{eqn:integral_evol_grad} we conclude that 
$$\int_{s_1}^{s_2}\int_M |\hess_{\bar g} \bar v(s)|^2 d\bar \nu ds<+\infty$$ and
\begin{equation*}
\int_M |\nabla \bar v(s_2)|^2  d\bar\nu - \int_M |\nabla \bar v(s_1)|^2 d\bar\nu =  - 2\int_{s_1}^{s_2}  \int_M |\hess_{\bar g} \bar v|^2  d\bar\nu ds.
\end{equation*}

\end{proof}

\begin{lemma}\label{lemma:hessian_decay_solitons}
Suppose that there is $q\in B(p,-1,R)$ such that $(M,g(t),q)_{t\in (-\infty,0)}$  is $k$-selfsimilar but not $(k+1,\eta)$-selfsimilar, for some $1\leq k \leq n$ and $\eta>0$.  There is $\beta=\beta(n,C_I,\Lambda,H,R,\eta)>0$ such that if $v\in C^\infty(M\times I)$, $I\subset(-\infty,0)$ is a  solution to the heat equation with $\sup_{t\in [a,b]} ||v(\cdot,t)||_{H^1_{\nu_{f,t}}}\leq C(a,b)$, for any $[a,b]\subset I$, then for every $t_1,t_2\in I$, $t_1<t_2$
\begin{equation}
\int_{t_2}^{t_2/4} \int_M |\hess v|^2 d\nu_{f,t} d t \leq \left(\frac{|t_2|}{|t_1|}\right)^{\beta} \int_{t_1}^{t_1/4} \int_M |\hess v|^2 d\nu_{f,t} d t.
\end{equation}
\end{lemma}
\begin{proof}
Recall that $\nu_f\in\mathcal S$ satisfies (CHF1) with respect to $p$. Let $(M,\bar g, \bar f)$ be the associated normalized gradient shrinking Ricci soliton at scale $1$, with $\bar g=g(-1)$, $\bar f=f(\cdot,-1)$, and $g(t)=|t|\varphi_t^* \bar g$ and $f(x,t)=\bar f(\varphi_t(x))$, where $\varphi_t$ are the diffeomorhpisms given by Proposition \ref{prop:backwards_forwards}. Set $d\bar\nu = (4\pi)^{-n/2} e^{-\bar f} d\vol_{\bar g}$. 

Consider the associated solution $\bar v(x,s)=e^{\frac{|s|}{2}}v(\varphi^{-1}_{-e^{-s}}(x),-e^{-s})$ of \eqref{eqn:drift_heat}, so that for any $t,t_i \in I$ and $s=-\log |t|$, $s_i=-\log|t_i|$
\begin{align*}
\int_M |\nabla v|_{g(t)}^2 d\nu_{f,t}&= \int_M |\nabla \bar v (s)|_{\bar g}^2 d\bar \nu,\\
\int_{t_i}^{t_i/4} \int_M |\hess v|^2 d\nu_{f,t} d t&= \int_{s_i}^{s_i+\log 4} \int_M |\hess_{\bar g} \bar v(s)|^2 d\bar \nu ds.
\end{align*}
By Lemma \ref{lemma:hessian_grad}, it follows that
\begin{equation*}
2\int_{t_i}^{t_i/4} \int_M |\hess v|^2 d\nu_{f,t} dt =  \int_M |\nabla \bar v(s_i)|^2 d\bar \nu -  \int_M |\nabla \bar v(s_i +\log 4)|^2 d\bar\nu.
\end{equation*}

Since $\bar v(s)=\sum_{i=0}^{+\infty} (\bar v(s_0),\phi_i)_{L^2_{\bar\nu}} e^{\lambda_i(s-s_0)} \phi_i$ and the convergence is in $H^1_{\bar \nu}$, we have that
$$\nabla \bar v(s)=\sum_{i=0}^{+\infty} (\bar v(s_0),\phi_i)_{L^2_\mu} e^{\lambda_i(s-s_0)} \nabla \phi_i,$$
in $L^2_{\bar \nu}$, hence
\begin{align*}
\int_M |\nabla \bar v(s)|^2 d\bar\nu -\int_M \sum_{i=0}^k (\bar v(s),\phi_i)_{L^2_\nu} ^2 |\nabla \phi_i|^2 d\bar\nu&=\sum_{i=k+1}^{+\infty} (\bar v(s),\phi_i)_{L^2_{\bar\nu}}^2  \int_M |\nabla \phi_i|^2 d\bar\nu\\
&=\sum_{i=k+1}^{+\infty} (\bar v(s+\log 4),\phi_i)_{L^2_{\bar\nu}}^2 e^{-2\lambda_i \log 4}\int_M |\nabla \phi_i|^2 d\bar\nu.
\end{align*}
Similarly,
\begin{align*}
\int_M |\nabla \bar v(s+\log 4)|^2 d\bar\nu -\int_M \sum_{i=0}^k (\bar v(s+\log 4),\phi_i)_{L^2_{\bar\nu}} ^2 |\nabla \phi_i|^2 d\bar\nu= \sum_{i=k+1}^{+\infty}(\bar v(s+\log 4),\phi_i)_{L^2_{\bar\nu}}^2 \int_M |\nabla \phi_i|^2 d\bar\nu.
\end{align*}
Noting that
$$\int_M \sum_{i=0}^k (\bar v(s+\log 4),\phi_i)_{L^2_{\bar\nu}} ^2 |\nabla \phi_i|^2 d\bar\nu = \int_M \sum_{i=0}^k (\bar v(s),\phi_i)_{L^2_{\bar\nu}} ^2 |\nabla \phi_i|^2 d\bar\nu,$$
since $\lambda_1=\cdots=\lambda_k=0$, we obtain that for $s_1<s_2$
\begin{align*}
&\int_M |\nabla \bar v(s_2)|^2  d\bar\nu-  \int_M |\nabla \bar v (s_2+\log 4)|^2d\bar\nu=\\
&=\sum_{i=k+1}^{+\infty} (\bar v(s_2+\log 4),\phi_i)_{L^2_{\bar\nu}}^2 (e^{-2\lambda_i\log 4}-1 ) \int_M |\nabla \phi_i|^2 d\bar\nu\\
&=\sum_{i=k+1}^{+\infty} (\bar v(s_1+\log 4),\phi_i)_{L^2_{\bar\nu}}^2 e^{2\lambda_i (s_2-s_1)} (e^{-2\lambda_i \log 4}-1 ) \int_M |\nabla \phi_i|^2 d\bar\nu\\
&\leq  e^{-2\delta(s_2-s_1)}\sum_{i=k+1}^{+\infty} (\bar v(s_1+\log 4),\phi_i)_{L^2_{\bar\nu}}^2 (e^{-2\lambda_i \log 4} - 1 ) \int_M |\nabla \phi_i|^2 d\bar\nu\\
&= e^{-2\delta(s_2-s_1)}\left( \int_M |\nabla \bar v|^2 (\cdot,s_1) d\bar\nu-  \int_M |\nabla \bar v|^2(\cdot,s_1+\log 4) d\bar\nu\right).
\end{align*}

Hence
\begin{equation*}
\begin{aligned}
&\int_{t_2}^{t_2/4}\int_M |\hess_{g(t)} v|^2 d\nu_{f,t} dt =\int_{s_2}^{s_2+\log4}\int_M |\hess_{\bar g} \bar v|^2 d\bar \nu ds \\
&\leq e^{-2\delta(s_2-s_1)}\int_{s_1}^{s_1+\log 4}\int_M |\hess_{\bar g} \bar v|^2 d\bar\nu ds=\left(\frac{|t_2|}{|t_1|}\right)^{2\delta}\int_{t_1}^{t_1/4}\int_M |\hess_{g(t)} v|^2 d\nu_{f,t} dt. 
\end{aligned}
\end{equation*}

\end{proof}

\section{Splitting maps}\label{sec:splitting_maps}
In this section we explore some basic properties of $(k,\delta)$-splitting maps that we will use frequently in the remaining of the paper. 

\begin{definition}[$(k,\delta)$-splitting map]
\label{def:splitting_map}
Let $(M,g(t),p)_{t\in [-r^2,0]}$ be a smooth complete pointed Ricci flow and let $1\leq k\leq n$. Then $v=(v^1,\ldots,v^k): M\times (-r^2,0)\rightarrow \mathbb R^k$ is a $(k,\delta)$-splitting map around $p$ at scale $r$ if
\begin{enumerate}
\item Each $v^a$ solves the heat equation $\frac{\partial v^a}{\partial t} = \Delta_{g(t)} v^a$.
\item For every $a=1,\ldots,k$
\begin{equation}\label{eqn:def_splitting_map_hessian}
\int_{-r^2}^{-\delta r^2}\int_{M} | \hess_{g(t)} v^a|^2 d\nu_{(p,0),t} dt \leq\delta
\end{equation}
\item For any $a,b=1,\ldots,k$
\begin{equation}\label{eqn:def_splitting_map_gradient}
\int_{-r^2}^{-\delta r^2}\int_{M} \left| \langle\nabla v^a,\nabla v^b\rangle -\delta^{ab} \right|^2  d\nu_{(p,0),t} dt\leq\delta r^2.
\end{equation}
\end{enumerate}
Similarly, if $(g(t))$ is only smooth for $t<0$ and $\nu$ is a conjugate heat flow on $(M,g(t))_{t\in[-r^2,0)}$ then we say that a solution to the heat equation $v:M\times (-r^2,0)\rightarrow \mathbb R^k$ is a $(k,\delta)$-splitting map with respect to $\nu$ at scale $r$, if \eqref{eqn:def_splitting_map_hessian} and \eqref{eqn:def_splitting_map_gradient} hold with respect to  $\nu$.
\end{definition}
If $v=(v^1,\ldots,v^k)$ is a $(k,\delta)$-splitting map we will often use the notation
\begin{align*}
||\nabla v|| &=\max_a |\nabla v^a|,\\
||\hess v|| &=\max_a |\hess v^a|.
\end{align*}

\begin{lemma}[Change of center and scale of splitting maps]\label{lemma:near_splitting}
 Let $(M,g(t))_{t\in [-1,0]}$ be a smooth compact  Ricci flow and let $v:M\times (-1,0)\rightarrow \mathbb R^k$ be a $(k,\delta)$-splitting map around $p$ at scale $1$, and fix $r_0 \in (0,1]$ and $\varepsilon>0$.
\begin{enumerate}
\item If $0<\delta\leq\delta(r_0|\varepsilon)$ then $v$ is also an $(k,\varepsilon)$-splitting map around $p$ at scale $r$ for every $r\in [r_0,1]$.
\item Fix $s>0$. If $(M,g(t))_{t\in [-1,0]}$ satisfies (RF1), (RF3-4) then there is $\gamma=\gamma(n,H)\in (0,1]$ such that if $0<\delta\leq \delta(n,C_I,H|s,\varepsilon)$ then $v$ is a $(k,\varepsilon)$-splitting map around any $q\in B(p,-1,s)$ at scale $\gamma$.
\end{enumerate}
\end{lemma}
\begin{proof}
Assertion 1 easily follows from the estimate
\begin{align*}
&\int_{-r^2}^{-\varepsilon r^2} \int_M \left(r^{-2}|\langle\nabla v^a,\nabla v^b\rangle -\delta^{ab} |^2 +|\hess v^a|^2 \right) d\nu_{(p,0),t} dt\\
&\leq \int_{-1}^{-\delta} \int_M \left(r^{-2}|\langle\nabla v^a,\nabla v^b\rangle -\delta^{ab} |^2 +|\hess v^a|^2 \right) d\nu_{(p,0),t} dt\\
&<r_0^{-2}\delta +\delta<\varepsilon,
\end{align*}
if we choose $\delta>0$ small enough.

Proving Assertion 2 requires us to convert integral estimates in $M\times [-1,-\delta]$ with respect to the conjugate heat kernel measure $d\nu_{(p,0)}=u_{(p,0)} d\vol_{g(t)}$, $u_{(p,0)}=(4\pi|t|)^{-n/2} e^{-f}$, to estimates in $M\times [-\gamma^2,-\varepsilon\gamma^2]$ with respect to the conjugate heat kernel measure $d\nu_{(q,0)}=u_{(q,0)}d\vol_{g(t)}$, for any $q\in B(p,-1,s)$, and an appropriate choice of $\gamma\in (0,1)$. This is achieved by applying Lemma \ref{lemma:compare_kernels} and the concentration estimates of Corollary \ref{cor:L2toL2weighted}.

Namely, choosing $\alpha=\frac{2C_1-C_2}{2C_1}$ and applying Lemma \ref{lemma:compare_kernels} for $\beta=1$, under assumptions (RF3) and (RF4), we obtain that for any $t\in [-r^2,0)$
\begin{equation}\label{eqn:comparison_1}
e^{\alpha f(\cdot,t)} u_{(p,0)}(\cdot,t) \geq C(n,H) e^{-\frac{(d_{g(t)}(p,q))^2}{C_1 |t|}}u_{(q,0)}(\cdot,t).
\end{equation}
Since, by (RF1), $|\riem|(x,t)\leq C_I \varepsilon^{-1} \gamma^{-2}$ for $(x,t)\in M\times [-1,-\varepsilon \gamma^2]$, we have the standard distance distortion estimate
$$d_{g(t)}(p,q) \leq C(n,C_I,\varepsilon \gamma^2) d_{g(-1)}(p,q)\leq C(n,C_I,\varepsilon \gamma^2) s$$
for every $t\in [-1,-\varepsilon \gamma^2]$. Thus, for every $t\in [-1,-\varepsilon \gamma^2]$, inequality \eqref{eqn:comparison_1} becomes
\begin{equation}\label{eqn:comparison_2}
\begin{aligned}
e^{\alpha f(\cdot,t)} u_{(p,0)}(\cdot,t) &\geq C(n,H) e^{-\frac{(d_{g(t)}(p,q))^2}{C_1 |t|}} u_{(q,0)}(\cdot,t) \\
&\geq C(n,H)  e^{-C(n,C_I,C_1,\gamma|\varepsilon)s^2} u_{(q,0)}(\cdot,t).
\end{aligned}
\end{equation}
Therefore,
\begin{equation}\label{eqn:change_centers}
\begin{aligned}
&\int_{-\gamma^2}^{-\varepsilon\gamma^2} \int_M \left(\gamma^{-2} \left| \langle\nabla v^a,\nabla v^b\rangle -\delta^{ab} \right|^2+|\hess v^a|^2 \right) d\nu_{(q,0),t} dt \leq \\
&\leq C(n,H) e^{C(n,C_I,C_1,\gamma|\varepsilon)s^2}\int_{-\gamma^2}^{-\varepsilon\gamma^2} \int_M \left(\gamma^{-2} \left| \langle\nabla v^a,\nabla v^b\rangle -\delta^{ab} \right|^2+|\hess v^a|^2 \right) e^{\alpha f} d\nu_{(p,0),t} dt.
\end{aligned}
\end{equation}
Now, by H\"older's inequality for $\frac{1}{s_1}+\frac{1}{s_2}=1$, $s_2= \alpha^{-1/2}$, the $L^2$ triangle inequality, and assuming that $0<\delta<\varepsilon \gamma^2$, we can estimate
\begin{equation*}
\begin{aligned}
&\int_{-\gamma^2}^{-\varepsilon \gamma^2} \int_M \left| \langle \nabla v^a,\nabla v^b\rangle - \delta^{ab} \right|^2 e^{\alpha f} d\nu_{(p,0),t}dt \leq\\
&\leq \left( \int_{-\gamma^2}^{-\varepsilon \gamma^2} \int_M \left| \langle \nabla v^a,\nabla v^b\rangle - \delta^{ab} \right|^2 d\nu_{(p,0),t}dt\right)^{1/s_1} \\
&\cdot \left(\int_{-\gamma^2}^{-\varepsilon \gamma^2} \int_M \left| \langle \nabla v^a,\nabla v^b\rangle - \delta^{ab} \right|^2 e^{\sqrt\alpha f} d\nu_{(p,0),t} dt\right)^{1/s_2}\\
&\leq \delta^{1/s_1}\left[ \left( \int_{-\gamma^2}^{-\varepsilon\gamma^2} \int_M \langle \nabla v^a,\nabla v^b\rangle^2 e^{\sqrt\alpha f} d\nu_{(p,0),t}dt \right)^{1/2} + \left( \int_{-\gamma^2}^{-\varepsilon\gamma^2} \int_M  e^{\sqrt\alpha f} d\nu_{(p,0),t}dt \right)^{1/2}\right]^{2/s_2}\\
&\leq \delta^{1/s_1}\left[ \left(\int_{-\gamma^2}^{-\varepsilon \gamma^2} \int_M ||\nabla v||^4 e^{\sqrt\alpha f} d\nu_{(p,0),t} dt\right)^{1/2}+ \left( \int_{-\gamma^2}^{-\varepsilon\gamma^2} \int_M  e^{\sqrt\alpha f} d\nu_{(p,0),t} dt \right)^{1/2}\right]^{2/s_2}.
\end{aligned}
\end{equation*}
Hence, \eqref{eqn:change_centers} becomes
\begin{equation}\label{eqn:estimates_wrt_eaf}
\begin{aligned}
\int_{-\gamma^2}^{-\varepsilon\gamma^2} \int_M \left(\gamma^{-2} \left| \langle\nabla v^a,\nabla v^b\rangle -\delta^{ab} \right|^2+|\hess v^a|^2 \right) &d\nu_{(q,0),t)} dt \leq \\
\leq C(n,C_I,H,\gamma,s|\varepsilon) &\left(\gamma^{-2}\delta^{1/s_1}( I^{1/2} +II^{1/2})^{2/s_2} + III \right),
\end{aligned}
\end{equation}
where
\begin{align*}
I&=\int_{-\gamma^2}^{-\varepsilon \gamma^2} \int_M ||\nabla v||^4 e^{\sqrt\alpha f} d\nu_{(p,0),t} dt,\\
II&=\int_{-\gamma^2}^{-\varepsilon\gamma^2} \int_M  e^{\sqrt\alpha f} d\nu_{(p,0),t} dt, \\
III&=\int_{-\gamma^2}^{-\varepsilon \gamma^2} \int_M |\hess v^a|^2 e^{\alpha f} d\nu_{(p,0),t} dt.
\end{align*}
From Lemma \ref{lemma:int_ker_bounds} and assumptions (RF1), (RF3), we obtain that $II\leq C(n,C_I,C_1)$. In order to estimate the quantities $I$ and $III$, recall from Corollary \ref{cor:subsolutions} that for any $a=1,\ldots,k$
\begin{align*}
\left(\frac{\partial}{\partial t}-\Delta\right)|\nabla v^a| \leq 0,\\
\left(\frac{\partial}{\partial t}-\Delta\right)\left(|t|^{E/2} | \hess v^a|\right) \leq 0,
\end{align*}
where $E=E(n,C_I)$.

It follows from Corollary \ref{cor:L2toL2weighted} that there is  $\gamma=\gamma(n,H)$ such that for every $t_1\in[-1,-1/2]$ and $t\in [-\gamma^2,-\varepsilon \gamma^2]$
\begin{equation}\label{eqn:timewise_concentration}
\begin{aligned}
\int_M |\nabla v^a|^4(\cdot,t) e^{\sqrt\alpha f(\cdot,t)} d\nu_{(p,0),t} &\leq  C(n,C_I,C_1) \left( \int_M |\nabla v^a|^2(\cdot,t_1) d\nu_{(p,0),t_1}\right)^2 \\
\int_M |t|^E |\hess v^a|^2(\cdot,t) e^{\alpha f(\cdot,t)} d\nu_{(p,0),t}&\leq C(n,C_I, C_1)\int_M |t_1|^E |\hess v^a|^2(\cdot,t_1) d\nu_{(p,0),t_1}.
\end{aligned}
\end{equation}
Now, applying the mean value theorem to estimate the right-hand side of \eqref{eqn:timewise_concentration} and then integrating the left-hand side in the interval $[-\gamma^2,-\varepsilon\gamma^2]$, we obtain
\begin{align*}
I&=\int_{-\gamma^2}^{-\varepsilon\gamma^2}\int_M |\nabla v^a|^4(\cdot,t) e^{\sqrt\alpha f(\cdot,t)} d\nu_{(p,0),t} dt\\
&\leq C(n,C_I)\left(\int_{-1}^{-1/2}  \int_M |\nabla v^a|^2(\cdot,t_1) d\nu_{(p,0),t_1}dt_1\right)^2 \leq C(n,C_I)\\
III&=\int_{-\gamma^2}^{-\varepsilon\gamma^2}\int_M |t|^E |\hess v^a|^2(\cdot,t) e^{\alpha f(\cdot,t)} d\nu_{(p,0),t}dt\\
&\leq C(n,C_I,C_1)\int_{-1}^{-1/2}\int_M |t_1|^E |\hess v^a|^2(\cdot,t_1) d\nu_{(p,0),t_1}dt_1\leq C(n,C_I,C_1)\delta,
\end{align*}
since $v$ is a $(k,\delta)$-splitting map around $p$. Therefore, if $0<\delta\leq \delta(n,C_I,H,s|\varepsilon)$, it follows from \eqref{eqn:estimates_wrt_eaf} that $v$ is a $(k,\varepsilon)$-splitting map around $q$.

\end{proof}

\begin{lemma}[Gradient estimate of a splitting map]\label{lemma:gradient_one}
Fix $\varepsilon>0$ and let $(M,g(t))_{t\in[-1,0]}$ be a smooth complete Ricci flow satisfying (RF1), (RF3-5), and $p\in M$. There is $\gamma=\gamma(n,H)$ such that if  $0<\delta\leq\delta(n,C_I|\varepsilon)$ and $v: M \times [-r^2,0]\rightarrow \mathbb R^k$ a $(k,\delta)$-splitting map around $p$ at scale $r$, then for any $t\in [-\gamma r^2,0]$ and $x\in B\left(p,t,r\right)$
$$|\nabla v^a|(x,t)\leq 1+\varepsilon.$$
\end{lemma}
\begin{proof}
Apply Lemma \ref{lemma:heat_gradient_bound} to each component of $v$.
\end{proof}

\begin{lemma}[Normalizing a splitting map]\label{lemma:normalize_splitting}
Let $(M,g(t),p)_{t\in [-1,0]}$ be a smooth complete Ricci flow and $v:M \times [-1,0]\rightarrow \mathbb R$ a $(k,\delta)$-splitting map around $p$ at scale $1$. Then for every  $r\in [2\sqrt \delta, 1]$ there is a unique lower triangular $k\times k$ matrix $T_r$ such that  $v_r=T_r v$, with $v_r^a=(T_r)^a_m v^m$, which satisfies for every $a,b=1,\ldots,k$
\begin{equation}\label{eqn:Tr_on}
\frac{4}{3r^{2}}\int_{-r^2}^{-r^2/4}\int_M \langle \nabla v_r^a,\nabla v_r^b \rangle  d\nu_{(p,0),t} dt =  \delta^{ab}.
\end{equation}
Then $||T_r-I_k||\leq C(n) r^{-2}\sqrt\delta$, where $||\cdot||$ denotes the maximum norm of a $k\times k$ matrix.
\end{lemma}
\begin{proof}
Since $v$ is a $(k,\delta)$-splitting map around $p$ at scale $1$, by H\"older's inequality we obtain, for every $r\in[2\sqrt\delta,1]$,
\begin{align*}
\frac{4}{3r^2}\int_{-r^2}^{-\frac{r^2}{4}} \int_M& \left| \langle \nabla v^l,\nabla v^m \rangle -\delta^{lm} \right|  d\nu_{(p,0),t} dt\leq \frac{4}{3r^2} \int_{-1}^{-\delta} \int_M \left| \langle \nabla v^l,\nabla v^m \rangle -\delta^{lm} \right| d\nu_{(p,0),t} dt\\
&\leq \frac{4}{3r^2}\left( \int_{-1}^{-\delta} \int_M \left| \langle \nabla v^l,\nabla v^m \rangle -\delta^{lm} \right|^2 d\nu_{(p,0),t} dt \right)^{1/2} \leq\frac{4}{3r^2} \delta^{1/2}.
\end{align*}

Therefore, there is lower triangular $k\times k$ matrix $\tilde T$ such that $||\tilde T-I_k||\leq C(n) r^{-2} \sqrt{\delta}$ and $\tilde T v$ satisfies
\begin{equation}\label{eqn:tilde_T_on}
\frac{4}{3r^2}\int_{-r^2}^{-\frac{r^2}{4}} \int_M  \langle \nabla (\tilde T v)^a,\nabla (\tilde T v)^b\rangle d\nu_{(p,0),t} dt =\delta^{ab}.
\end{equation}

To prove the uniqueness assertion, consider the bilinear form 
$$B(f,g)=\frac{4}{3r^2}\int_{-r^2}^{-\frac{r^2}{4}} \int_M \langle \nabla f,\nabla g \rangle d\nu_{(p,0),t} dt.$$
and denote also by $B$ the matrix  $B_{ab} = B(v^a,v^b)$. Equations \eqref{eqn:Tr_on} and \eqref{eqn:tilde_T_on} then imply that
$$T_r^*B T_r = I_k = \tilde T^* B \tilde T,$$
so
$$B=(T_r^{-1})^*T_r^{-1} =  (\tilde T^{-1})^* \tilde T^{-1}.$$
By the uniqueness of the Cholesky decomposition, it follows that $T_r =\tilde T$. Therefore,
$$||T_r - I_k||\leq C(n) r^{-2} \sqrt \delta.$$
\end{proof}

\begin{proposition}[Compactness of splitting maps]\label{prop:s_map_compactness}
Let $(M_j,g_j(t),p_j)_{t\in (-1,0]}$ be a pointed sequence of smooth compact Ricci flows satisfying (RF1-3), a sequence $\varepsilon_j$ such that $\liminf_j \varepsilon_j = \varepsilon \in [0,1]$ and a sequence $v_j:M_j\times [-1,0]\rightarrow \mathbb R^k$ of $(k,\varepsilon_j)$-splitting maps around $p_j$ at scale $1$, normalized so that
\begin{equation}\label{eqn:compactness_average_assumption}
\int_{-1}^0\int_{M_j} v_j^a d\nu_{(p_j,0),t}dt = 0.
\end{equation}
Then, there is a pointed smooth complete  Ricci flow $(M_\infty,g_\infty(t),p_\infty)_{t\in (-1,0)}$ satisfying (RF1), a conjugate heat flow $\nu_{\infty,t}$, satisfying (CHF1) with respect to $p_\infty$, and a $(k,\varepsilon)$-splitting map $v_\infty:M_\infty\times (-1,0)\rightarrow \mathbb R^k$ with respect to $\nu_\infty$ at scale $1$, which, after passing to a subsequence, are the limits of $(M_j,g_j(t),p_j)_{t\in (-1,0]}$, $\nu_{(p_j,0)}$ and $v_j$ respectively.

Moreover, $v_\infty$ satisfies
\begin{equation}\label{eqn:compactness_average_conclusion}
 \int_{M_\infty} v_\infty^a d\nu_{\infty,t} dt=0,
\end{equation}
for every $t\in (-1,0)$ and
\begin{equation}\label{eqn:compactness_B_convergence}
\lim_{j\rightarrow+\infty}\int_{-1}^{-1/4}\int_{M_j} \langle\nabla v_j^a,\nabla v_j^b\rangle d\nu_{(p_j,0),t} dt = \int_{-1}^{-1/4} \int_{M_\infty}\langle\nabla v_\infty^a,\nabla v_\infty^b\rangle d\nu_{\infty,t}dt
\end{equation}
\end{proposition}
\begin{proof}
By Proposition \ref{prop:compactness_rf}, after passing to a subsequence, we can assume that $(M_j,g_j(t),p_j)_{t\in (-1,0)}$ converges to a limit $(M_\infty,g_\infty(t),p_\infty)_{t\in (-1,0)}$, and that $\nu_{(p_j,0)}$ converges to a conjugate heat flow $\nu_{\infty}$. To obtain the convergence of $v_j$ we need to establish the necessary a priori estimates.

Applying H\"older's inequality, for large $j$ we obtain
\begin{align*}
&\int_{-1}^{-\varepsilon_j} \int_{M_j} | \nabla v_j^a|^2 d\nu_{(p_j,0),t}dt \leq 1+\int_{-1}^{-\varepsilon_j} \int_{M_j} \left| | \nabla v_j^a|^2 - 1\right| d\nu_{(p_j,0),t} dt \\
&\leq 1+\left(\int_{-1}^{-\varepsilon_j} \int_{M_j} \left| | \nabla v_j^a|^2 - 1 \right|^2 d\nu_{(p_j,0),t} dt\right)^{1/2}\leq 1+\varepsilon_j^{1/2}\leq 3.
\end{align*}
Therefore, since $t\mapsto \int_{M_j} v_j^a(\cdot,t) d\nu_{(p_j,0),t}$ is constant, by \eqref{eqn:compactness_average_assumption} the Hein--Naber Poincare inequality Theorem \ref{thm:poincare} implies that
\begin{equation}\label{eqn:splitting_map_L2_bound}
\begin{aligned}
\int_{-1}^{-\varepsilon_j}\int_{M_j} (v_j^a)^2 d\nu_{(p_j,0),t} dt  &\leq \int_{-1}^{-\varepsilon_j} 2|t| \int_{M_j} |\nabla v_j^a|^2 d\nu_{(p_j,0),t} dt \\
&\leq 2\int_{-1}^{-\varepsilon_j} \int_{M_j} |\nabla v_j^a|^2 d\nu_{(p_j,0),t} dt \leq 6.
\end{aligned}
\end{equation}
Moreover, since $\frac{\partial}{\partial t} (v_j^a)^2 \leq \Delta (v_j^a)^2$, the function $t\mapsto \int_{M_j}(v_j^a)^2 d\nu_{(p_j,0),t}$ is non-increasing, hence \eqref{eqn:splitting_map_L2_bound} implies
\begin{equation}\label{eqn:splitting_map_L2_bound_2}
\int_{-1}^{0}\int_{M_j} (v_j^a)^2 d\nu_{(p_j,0),t} dt \leq 10,
\end{equation}
for large $j$.

By the lower bound (RF4) on the conjugate heat kernels $u_{(p_j,0)}$, \eqref{eqn:splitting_map_L2_bound_2} we obtain for every $t\in (-1,0)$ and $r>0$
\begin{equation*}
\begin{aligned}
 \int_{B(p_j,t,r)} (v_j^a)^2 d\vol_{g_j(t)} &\leq C_2^{-1} (4\pi | t|)^{n/2} \int_{B(p_j,t,r)} e^{\frac{d_{g_j( t)}(p_j,\cdot)^2}{C_2 |t|}}(v_j^a)^2 d\nu_{(p_j,0), t} \\
 &\leq C_2^{-1} (4\pi |t|)^{n/2} e^{\frac{r^2}{C_2|t|}}  \int_{B(p_j,t,r)} (v_j^a)^2 d\nu_{(p_j,0),t}.
  \end{aligned}
\end{equation*}
Now for any $(x,\bar t)\in M_j\times (-1,0)$ and $\rho>0$ such that $\bar t-\rho^2 \in (0,1)$ we know by the triangle inequality and the distance distortion bound (RF5) that
$$B(x,\bar t, \rho) \subset B(p_j,\bar t, d_{g_j(\bar t)}(p_j,x) + \rho)\subset B(p_j,t, d_{g_j(\bar t)}(p_j,x) + \rho +K),$$
for any $t\in [\bar t-\rho^2,\bar t]$. 

Therefore,
\begin{equation*}
\int_{\bar t-\rho^2}^{\bar t} \int_{B(x,\bar t,\rho)} (v_j^a)^2 d\vol_{g_j(t)} dt \leq C e^{\frac{d_{g_j(\bar t)}(p_j,x)^2}{C_2|\bar t|}}\int_{-1}^0 \int_{M_j}(v_j^a)^2 d\nu_{(p_j,0),t} dt \leq Ce^{\frac{d_{g_j(\bar t)}(p_j,x)^2}{C_2|\bar t|}},
\end{equation*}
where $C$ only depends on $n$, the constant $C_2$ from assumption (RF4) bound $C_2$, and the distance distortion constant $K$ from assumption (RF5).

Since $(v_j^a)^2$ satisfies $\partial_t (v_j^a)^2\leq \Delta (v_j^a)^2$, and assumption (RF1-2) are valid, we can apply the parabolic mean value inequality (Theorem 25.2 in \cite{RF_Part3}) in each domain $B(x,\bar t, \rho)\times [\bar t-\rho^2,\bar t]$  for which the estimate above holds, to obtain that
$$(v_j^a)^2(x,\bar t) \leq C(n,C_I, \Lambda,C_2, K| \rho)e^{\frac{d_{g_j(\bar t)}(p_j,x)^2}{C_2|\bar t|}}.$$
This bound, by parabolic regularity, suffices to bound derivatives of any order, uniform in compact regions of $M_\infty\times (-1,0)$, after pulling back each $v_j$ to $M_\infty$, via the diffeomorphisms associated to the convergence of $(M_j,g_j(t),p_j)$ to $(M_\infty,g_\infty(t),p_\infty)$. Therefore, we establish that a subsequence of $v_j$ converges to a smooth limit $v_\infty:M_\infty\times (-1,0)\rightarrow \mathbb R^k$, namely $v_j$ and its derivatives converge uniformly in compact subsets of  $M_\infty\times (-1,0)$.

Now, since the functions $|\hess v_j^a|^2$ and $\left|\langle\nabla v_j^a,\nabla v_j^b\rangle - \delta^{ab}\right|^2$ are non-negative, by Fatou's Lemma it follows that
$$\int_{-1}^{-\varepsilon}\int_{M_\infty} \left(|\hess v_\infty^a|^2 + \left|\langle\nabla v_\infty^a,\nabla v_\infty^b\rangle - \delta^{ab}\right|^2\right) u_\infty d\vol_{g_\infty(t)} dt \leq \varepsilon,$$
hence $v_\infty$ is a $(k,\varepsilon)$-splitting map with respect to $\nu_\infty$ at scale $1$.

To obtain \eqref{eqn:compactness_average_conclusion}, observe that \eqref{eqn:splitting_map_L2_bound_2} also implies such that for all $j$ and $t'\in [-1,0)$
\begin{equation}\label{eqn:compactness_L2_bound_timewise}
\int_{M_j} (v_j^a)^2d\nu_{(p_j,0),t'}\leq  (t'+1)^{-1}\int_{-1}^{t'} \int_{M_j} (v_j^a)^2 d\nu_{(p_j,0),t}dt \leq 6(t'+1)^{-1},
\end{equation}
since $t\mapsto \int_{M_j} (v_j^a)^2(\cdot,t) d\nu_{(p_j,0),t}$ is non-increasing. 

Therefore, by they Cauchy--Schwarz inequality and Lemma \ref{lemma:int_ker_bounds}, using (RF1) and (RF3), we know that for any $r\geq r(t,\delta)$,
\begin{align*}
 &\int_{M_j\setminus B(p_j,t,r\sqrt{|t|})} |v_j^a| d\nu_{(p_j,0),t}\leq\\
&\leq \left(\int_{M_j\setminus B(p_j,t,r\sqrt{|t|})} (v_j^a)^2 d\nu_{(p_j,0),t}\right)^{1/2} \left(\int_{M_j\setminus B(p_j,t,r\sqrt{|t|})}  d\nu_{(p_j,0),t}\right)^{1/2}\\
&\leq C(t+1)^{-1/2}\left(\int_{M_j\setminus B(p_j,t,r\sqrt{|t|})}  d\nu_{(p_j,0),t} \right)^{1/2}\leq \delta.
\end{align*}

Therefore, by Fatou's Lemma,
\begin{align*}
\left|\int_{M_\infty\setminus B(p_\infty,t,r\sqrt{|t|})} v_\infty^a d\nu_{\infty,t}\right| &\leq \int_{M_\infty\setminus B(p_\infty,-t,r\sqrt{|t|})} |v_\infty^a| d\mu_{\infty,t}\\
&\leq \liminf_j \int_{M_j\setminus B(p_j,t,r\sqrt{|t|})} |v_j^a| d\nu_{(p_j,0),t} \leq \delta.
\end{align*}
On the other hand, by \eqref{eqn:compactness_average_assumption}
\begin{align*}
&\left|\int_{B(p_j,t,r\sqrt{|t|})} v_j^a d\nu_{(p_j,0),t}\right| \leq \int_{M_j\setminus B(p_j,t,r\sqrt{|t|})} |v_j^a| d\nu_{(p_j,0),t}<\delta,
\end{align*}
therefore, for large $j$,
\begin{equation*}
\left| \int_{B(p_\infty,t,r\sqrt{|t|})} v_\infty^a d\nu_{\infty,t}\right|<2\delta.
\end{equation*}
We conclude that for every $\delta>0$ and $t\in (-1,0)$
$$\left|\int_{M_\infty} v_\infty^a d\nu_{\infty,t} \right| \leq 3\delta,$$
which suffices to prove \eqref{eqn:compactness_average_conclusion}.

The proof of \eqref{eqn:compactness_B_convergence} is similar, after noting that  
$$|\langle \nabla v_j^a,\nabla v_j^b\rangle | \leq |\langle \nabla v_j^a,\nabla v_j^b\rangle - \delta^{ab}| + \delta^{ab}\leq |\langle \nabla v_j^a,\nabla v_j^b\rangle - \delta^{ab}| + 1,$$
which gives 
$$|\langle \nabla v_j^a,\nabla v_j^b\rangle |^2 \leq \left( |\langle \nabla v_j^a,\nabla v_j^b \rangle - \delta^{ab}| + 1 \right)^2\leq 2|\langle \nabla v_j^a,\nabla v_j^b \rangle - \delta^{ab}| ^2 +2.$$
This allows us to obtain the uniform bound
\begin{align*}
&\int_{-1}^{-1/4}\int_{M_j} |\langle\nabla v_j^a,\nabla v_j^b\rangle|^2 d\nu_{(p_j,0),t}dt\leq\\
&\leq  2 \int_{-1}^{-1/4}\int_{M_j} (|\langle \nabla v_j^a,\nabla v_j^b \rangle - \delta^{ab}| ^2 +1)d\nu_{(p_j,0),t} dt \leq C.
\end{align*}
Now, the same line of reasoning used to prove  \eqref{eqn:compactness_average_conclusion} via \eqref{eqn:compactness_L2_bound_timewise} proves \eqref{eqn:compactness_B_convergence}.

\end{proof}

Combining Lemma \ref{lemma:k_delta_convergence} with Proposition \ref{prop:s_map_compactness} and an easy contradiction-compactness argument we obtain the following corollary.

\begin{corollary}
Let $\varepsilon>0$ and $(M,g(t),p)_{t\in (-2\delta^{-1},0]}$ be a smooth compact Ricci flow satisfying (RF1-3) and let $v:M\times[-1,0] \rightarrow \mathbb R^k$ be a $(k,\delta)$-splitting map around $p$ at scale $1$. If $0<\delta\leq \delta(n,C_I, \Lambda,H|\varepsilon)$ then the following holds. If $(M,g(t),p)_{t\in (-2\delta^{-1},0]}$ is $(0,\delta)$-selfsimilar at scale $1$ then it is $(k,\varepsilon)$-selfsimilar at scale $1$.
\end{corollary}

In the next section we will see that the converse is also true. Under some a-priori assumptions, if a pointed Ricci flow is $(k,\delta^2)$-selfsimilar, with $\delta>0$ small enough, then there exists a $(k,\varepsilon)$-splitting map, which is in some sense, optimal.

\section{Construction of splitting maps}\label{sec:construction}

In this section we show that, under certain assumptions such as a Ricci flow $(M,g(t),p)$ being $(k,\delta)$-selfsimilar, $(k,\varepsilon)$-splitting maps exist. Moreover, the splitting maps we will construct are optimal in the sense that their hessian is controlled by the geometry of the flow, in terms of the pointed $\mathcal W$ entropy. A fact that will be crucial later on, is that this control is linear.

\subsection{Regularization of the potential of a conjugate heat kernel}
We first need the following lemma, which constructs nice regularizations of conjugate heat kernel potentials.
\begin{lemma}\label{lemma:regularized_soliton_function}
Fix $0<\varepsilon<1$, $T\geq 10$, $R<+\infty$, let $(M,g(t),p)_{t\in [-2\delta^{-1},0]}$ be a smooth compact Ricci flow satisfying (RF1-5), and denote by $d\nu_{(p,0),t} = (4\pi |t|)^{-n/2}e^{-f(\cdot,t)} d\vol_{g(t)}$ the conjugate heat kernel measure starting at $(p,0)$.

Let  $0<\delta\leq\delta(n,C_I,\Lambda,H,K|\varepsilon,T,R)$
and suppose that either $(M,g(t),p)_{t\in [-2\delta^{-1},0]}$ is $(0,\delta)$-selfsimilar at scale $1$ or 
\begin{equation*}
\mathcal W_p(\delta) - \mathcal W_p(\delta^{-1})<\delta.
\end{equation*}  
 
 Then, there is a solution $w: M \times [-T,0]\rightarrow \mathbb R$ to 
\begin{equation*}
\frac{\partial w}{\partial t}  =\Delta w - \frac{n}{2}
\end{equation*}
and a constant $E=E(n,C_I)$ such that \begin{enumerate}
\item for any $(x,t)\in M\times [-T,0]$ and any $\bar t\in [-T,0]$
\begin{equation}\label{eqn:regularized_soliton_bound}
|w|^2(x,t) +|\nabla w|^2(x,t)+ |t|^E|\hess w|^2(x,t)\leq C(n,C_I,\Lambda,C_2,K|T)\left((d_{g(\bar t)}(p,x))^4+1\right) 
\end{equation} 
\item
for any $t\in [-T,0)$ it satisfies
\begin{equation}\label{eqn:regularized_soliton_quantity}
\int_M \left| |t| \ric + \hess w -\frac{g}{2}\right|^2 d\nu_{(p,0),t} \leq C(n,C_I)\left(\frac{T}{|t|}\right)^E\left( \mathcal W_p(1) - \mathcal W_p(2T) \right).
\end{equation}
\end{enumerate}
and it satisfies in $B(p,-1,\varepsilon^{-1})\times [-T,-\varepsilon]$
\begin{equation}\label{eqn:w_close_f}
\left|w-|t|(f-\mathcal W_p(1)) \right| < \varepsilon. 
\end{equation}
Moreover, for any $\alpha\in (0,1)$, if $T= T(\alpha)$ we have the  concentration estimate
\begin{equation}\label{eqn:reg_sol_concentration_1}
|t|^E\int_M \left| |t| \ric + \hess w -\frac{g}{2}\right|^2 e^{\alpha f} d\nu_{(p,0),t} \leq C(n,C_I,C_1|\alpha)\left( \mathcal W_p(1) - \mathcal W_p(2 T) \right),
\end{equation}
for any $t\in [-10,0]$.
\end{lemma}

\begin{proof}
Let $f$ be such that $d\nu_{(p,0),t}=(4\pi|t|)^{-n/2} e^{-f(\cdot,t)} d\vol_{g(t)}$.  By the monotonicity formula \eqref{eqn:pointed_monotonicity} we know that
\begin{equation}\label{eqn:f_sol}
\int_{-2T}^{-T} \int_M |t| \left| \ric + \hess f -\frac{g}{2|t|} \right|^2 d\nu_{(p,0),t} dt \leq \mathcal W_p(1)- \mathcal W_p (2 T)=:\eta.
\end{equation}

By the mean value theorem, there is $\bar T\in [\frac{3}{2}T,2T]$ such that
\begin{equation}\label{eqn:soliton_quantity_small}
\frac{T}{2\bar T} \int_M \left| \bar T \ric +\hess (\bar T f(\cdot,-\bar T)) -\frac{g}{2}\right|^2 d\nu_{(p,0),-\bar T}  \leq \eta.
\end{equation}

Let $w:M\times [-\bar T,0]\rightarrow \mathbb R$ be the solution to 
\begin{equation*}
\frac{\partial w}{\partial t} = \Delta w -\frac{n}{2}
\end{equation*}
with the initial condition $w(\cdot,-\bar T)=\bar T \left(f(\cdot,-\bar T) -\mathcal W_p(1)\right)$.

\noindent\textit{1. Pointwise bounds:} We first prove that $w$ satisfies the estimate \eqref{eqn:regularized_soliton_bound}. By (RF2) we obtain that
$$-\Lambda\leq \mu(g(-1),1)\leq \mathcal W_p(1)\leq 0.$$
Combining this with the Gaussian upper bounds on conjugate heat kernels (RF4) and (RF1), we can apply Corollary \ref{cor:local_heat_estimates}, to obtain that there is a constant $E=E(n,C_I)$ such that $w$ satisfies
\begin{align*}
|w(x,t)|^2+&(t+\bar T)|\nabla w|^2(x,t)+(t+\bar T)^2|\bar T|^{-E} |t|^E |\hess w|^2\leq C(n,C_I,C_2,\Lambda|T)( (d_{g(-\bar T)}(p,x))^4 +1),
\end{align*}
for every $t\in [-\bar T,0]$ and $x\in M$. 

Then, for every $\bar t\in [-T,0]$, the distance distortion estimate (RF5) implies
$$d_{g(-\bar T)}(p,x)\leq d_{g(\bar t)}(p,x)+2K T,$$ and gives \eqref{eqn:regularized_soliton_bound}, since $2T\geq\bar T\geq \frac{3}{2} T > T$.

\noindent \textit{2. $L^2$-bounds for $t\in [-T,0]$:}
By the standard evolution equation of Ricci curvature under Ricci flow  
$$\left(\frac{\partial}{\partial t} -\Delta_L \right)\ric =0$$ 
and \eqref{eqn:evol_hess}, it follows that $\left(\frac{\partial}{\partial t} -\Delta_L\right)\left(|t|\ric +\hess w -\frac{g}{2}\right)=0$. Thus,  by (RF1) and Lemma \ref{lemma:Lich_type_I} we obtain that there is $E=E(n,C_I)$ such that $S= |t| \ric +\hess w -\frac{g}{2}$ satisfies  
\begin{equation}\label{eqn:w_soliton_quantity_evol}
\left(\frac{\partial}{\partial t} -\Delta\right)(|t|^E |S|^2) \leq 0,
\end{equation}
which implies that for any $t\in [-\bar T, 0]$
\begin{equation}\label{eqn:evol_int_E}
\frac{d}{dt} \int_M |t|^E |S|^2 d\nu_{(p,0),t} \leq 0.
\end{equation}

Since, by \eqref{eqn:soliton_quantity_small}, we know that 
\begin{equation}\label{eqn:estimate_at_barT}
\int_M |\bar T|^E |S|^2d\nu_{(p,0),-\bar T} \leq \frac{2 \eta \bar T}{T} |\bar T|^E \leq C(n,C_I) T^E \eta,
\end{equation}
integrating \eqref{eqn:evol_int_E} we obtain that for any $t\in[-\bar T,0]$
\begin{equation}\label{eqn:v_soliton_q}
|t|^E\int_M \left| |t|\ric +\hess w -\frac{g}{2}\right|^2 d\nu_{(p,0),t}   \leq C(n,C_I) T^E  (\mathcal W_p(1)- \mathcal W_p (2T)),
\end{equation}
which proves \eqref{eqn:regularized_soliton_quantity}.

\noindent \textit{3. Concentration $L^2$-bounds for $t\in [-10,0)$:}
By (RF1) and Corollary \ref{cor:subsolutions} we also know that 
$$\left(\frac{\partial}{\partial t} -\Delta\right)( |t|^{E/2} |S|)\leq 0.$$ Thus, by (RF1), (RF3), Corollary \ref{cor:L2toL2weighted} and \eqref{eqn:w_soliton_quantity_evol}, we know that for every $\alpha\in(0,1)$ there is $T=T(\alpha)<+\infty$ such that for every $t\in [-10,0]$
\begin{equation}
\begin{aligned}
\int_M |t|^E |S|^2 e^{\alpha f} d\nu_{(p,0),t} &\leq C(n,C_I,C_1 |\alpha) \int_M T^E |S|^2  d\nu_{(p,0),-T}, \\
&\leq C(n,C_I,C_1 |\alpha) \int_M |\bar T|^E |S|^2 d\nu_{(p,0),-\bar T}.
\end{aligned}
\end{equation}
Therefore, by \eqref{eqn:estimate_at_barT},  we obtain for any $t\in [-10,0]$
\begin{equation}
\begin{aligned}
&|t|^E\int_M \left||t|\ric +\hess w-\frac{g}{2} \right|^2 e^{\alpha f} d\nu_{(p,0),t}\leq\\
&\leq C(n,C_I,C_1 |\alpha) |\bar T|^E \int_M \left|\bar T\ric +\hess w-\frac{g}{2} \right|^2 d\nu_{(p,0),-\bar T}\\
&\leq C(n,C_I,C_1 |\alpha) (\mathcal W_p (1) -\mathcal W_p(2 T)),
\end{aligned}
\end{equation}
which proves \eqref{eqn:reg_sol_concentration_1}.

\noindent \textit{4. Proximity to the potential of the conjugate heat kernel:} Suppose that there is a sequence $\delta_j\rightarrow 0$ and a pointed sequence smooth compact Ricci flows $(M_j,g_j(t),p_j)_{t\in (-2 \delta_j^{-1},0]}$ satisfying (RF1-5), which are either $(0,\delta_j)$-selfsimilar around $p_j$ at scale $1$ or \begin{equation*}
\mathcal W_p(\delta_j) - \mathcal W_p(\delta_j^{-1})<\delta_j,
\end{equation*}
and let $f_j$ be such that $d\nu_{(p_j,0),t} = (4\pi |t|)^{-n/2} e^{-f_j(\cdot,t)} d\vol_{g_j(t)}$.

By Proposition \ref{prop:compactness_rf}, Lemma \ref{lemma:k_delta_convergence}  and Proposition \ref{prop:W_drop_small}, after passing to a subsequence, we can assume that  \linebreak $(M_j,g_j(t),p_j)_{t\in (-2 \delta_j^{-1},0)}$ converges smoothly to a pointed Ricci flow $(M_\infty,g_\infty(t),p_\infty)_{t\in (-\infty,0)}$, which is induced by a gradient shrinking Ricci soliton with spine $\mathcal S$. Moreover, the conjugate heat flows $\nu_{(p_j,0)}$ converge to a conjugate heat flow $\nu_{f_\infty}\in\mathcal S$ on $(M_\infty,g_\infty(t))_{t\in (-\infty,0)}$, $f_\infty$ satisfies (CHF1) and (CHF2) with respect to $p_\infty$, by (RF3-4), and
\begin{equation}\label{eqn:entropies_converge}
\mathcal W_{p_j}(1)  \rightarrow \bar\mu(g_\infty(-1)).
\end{equation}

 Since $\nu_{f_\infty}\in\mathcal S$, $(M_\infty, g_\infty(t), f_\infty(\cdot,t))$ is a normalized gradient shrinking Ricci soliton at scale $|t|$, for every $t<0$. Therefore,  Proposition \ref{prop:backwards_forwards} implies that $\tilde w_\infty(t)= |t| (f_\infty(t) - \bar\mu(g_\infty(-1))$
satisfies 
\begin{equation}\label{eqn:evolution_tilde_w}
\frac{\partial \tilde w_\infty}{\partial t} = \Delta \tilde w_\infty -\frac{n}{2}
\end{equation}
in $M_\infty\times (-\infty,0)$.

Consider also $\bar T_j\in[\frac{3}{2} T,2T]$, and the functions $w_j: M_j\times [-\bar T_j,0)\rightarrow \mathbb R$ such that
\begin{align*}
\frac{\partial w_j}{\partial t} &= \Delta w_j -\frac{n}{2},\\
w_j(\cdot,-\bar T_j) &= \bar T_j (f_j(\cdot,-\bar T_j) -\mathcal W_{p_j}(1)),
\end{align*}
constructed above. Passing to a further subsequence, we can assume that $\bar T_i\rightarrow \bar T_\infty \in [\frac{3}{2} T,2 T]$.

Parabolic regularity, the locally uniform bounds in $M_j\times [\bar T_j,0]$ from \eqref{eqn:regularized_soliton_bound}, the smooth convergence $f_j(\cdot,-\bar T_j)\rightarrow f_\infty(\cdot,-\bar T_\infty)$ and \eqref{eqn:entropies_converge}, imply that $w_j$ smoothly locally converges to a function $w_\infty:M_\infty\times [-\bar T_\infty,0)\rightarrow \mathbb R$ that satisfies
\begin{equation}\label{eqn:vinfty_evolution}
\begin{aligned}
\frac{\partial w_\infty}{\partial t} &= \Delta w_\infty - \frac{n}{2},\\
w_\infty(\cdot,-\bar T_\infty) &= \bar T_\infty (f_\infty(\cdot,-\bar T_\infty) -\bar\mu(g_\infty(-1)) = \tilde w_\infty(\cdot,-\bar T_\infty).
\end{aligned}
\end{equation}

Since $w_\infty$ and $\tilde w_\infty$ both grow at most quadratically, by \eqref{eqn:regularized_soliton_bound} and (CHF2) respectively, and satisfy the same initial condition at $t=-\bar T_\infty$, we conclude that $w_\infty=\tilde w_\infty$ in $M_\infty\times [-\bar T_\infty,0)$, from which \eqref{eqn:w_close_f} follows.

\end{proof}

\subsection{Sharp splitting maps}

\begin{definition}\label{def:independence}
Let $(X,d,p)$ be a complete pointed metric space, $\mu>0$ and  $R<+\infty$. We will call a set of points $\{x_i\}\subset B(p,R)$, $i=0,\ldots,k$,  
\begin{enumerate}
\item $(k,\mu)$-independent in $B(p,R)$ if there is no $0\leq l<k$ and $x_i'\in \mathbb R^l$, $i=0,\ldots,k$ such that 
$$\left| d(x_i,x_j) - |x'_i-x'_j| \right|<\mu R.$$
\item $(k,\mu,D)$-independent in $B(p,R)$ if there is no $0\leq l <k$, a metric space $(K,d_K)$ with $\diam_{d_K}(K)\leq D$, and points $x_i'\in  K\times \mathbb R^l$, $i=0,\ldots,k$ such that for any $i,j=0,\ldots,k$ 
$$|d(x_i,x_j)-d_{K\times \mathbb R^l}(x_i',x_j')|<D+\mu R,$$
where $K\times \mathbb R^l$ is endowed with the product metric, namely
$$(K\times \mathbb R^l, d_{K\times \mathbb R^l}((q_1,a_1),(q_2,a_2))=\left((d_K(q_1,q_2) )^2+ |a_2-a_1|^2\right)^{1/2},$$
\end{enumerate}
\end{definition}

\begin{remark}\label{rmk:ind}
With a simple contradiction argument we can show that if $X\subset B(0,R)\subset\mathbb R^n$ does not contain any $(k,\delta)$-independent subsets then there is a $k-1$ plane $L$ in $\mathbb R^n$ such that $X\subset B_{\varepsilon R}(L)$, provided that $0<\delta\leq \delta(n|\varepsilon)$. For this, by rescaling we may assume that $R=1$ and then observe that if there is no $L$ so that $X\subset B_{\varepsilon}(L)$, then there is $\{x_i\}_{i=0}^k\subset X$ such that $x_i\not\in B_{\varepsilon } (<x_0,\ldots,x_{i-1}>)$ for every $i=1,\ldots, k$, where $<x_0,\ldots,x_m>$ denotes the $m$ plane in $\mathbb R^n$ defined by $x_0,\ldots,x_m\in\mathbb R^n$. These points will be $(k,\delta)$-independent, if $\delta>0$ is small enough, since otherwise we could take $\delta\rightarrow 0$ to construct an isometric embedding of $\{x_i\}_{i=0}^k$ into $\mathbb R^{k-1}$, which is impossible.
\end{remark}

\begin{lemma}\label{lemma:independences}
Let $(K,d_K)$ be a metric space with $\diam_{d_K}(K)\leq D$, and a collection of $k+1$ points, $x_i=(q_i,a_i)\in B((q,0),R)\subset K\times \mathbb R^k$, $i=0,\ldots,k$, where $K\times \mathbb R^k$ is endowed with the product metric $d_{K\times \mathbb R^k}$, as in Definition \ref{def:independence}. Then 
\begin{enumerate}
\item If  $\{x_i\}$ is $(k,\mu,D)$-independent  then $\{a_i\}\subset B(0,R)$ is $(k,\mu)$-independent.
\item If $\{a_i\}\subset B(0,R)$ is $(k,3\mu)$-independent and $R\geq \frac{D}{\mu}$, then $\{x_i\}$ is $(k,\mu,D)$-independent.
\end{enumerate}
\end{lemma}

\begin{proof}[Proof of Assertion 1]
Suppose that $\{x_i=(q_i,a_i)\}$, as in Assertion 1, is $(k,\mu,D)$-independent but $\{a_i\}\subset B(0,R)$ is not $(k,\mu)$-independent. Namely, there is $0\leq l<k$ and $b_i\in \mathbb R^l$, $i=0,\ldots,k$ such that for every $i,j=0,\ldots,k$
$$\left| |a_i-a_j| - |b_i - b_j|\right|<\mu R.$$
Define $y_i\in K\times \mathbb R^l$ as $y_i =(q_i,b_i)$. Then
$$d_{K\times \mathbb R^k}(x_i,x_j) - d_{K\times \mathbb R^l}(y_i,y_j) = \sqrt{(d_K(q_i,q_j))^2+|a_i-a_j|^2} - \sqrt{(d_K(q_i,q_j))^2 +|b_i-b_j|^2}.$$
To simplify notation, we will use $d$ to denote both $d_{K\times \mathbb R^k}$ and $d_{K\times \mathbb R^l}$.

Therefore, by the assumption $d_K(q_i,q_j)\leq D$, 
$$|a_i-a_j| - \sqrt{D^2+|b_i-b_j|^2}\leq d(x_i,x_j) - d(y_i,y_j) \leq  \sqrt{D^2 +|a_i-a_j|^2 } -|b_i-b_j|,$$
and since $\sqrt{D^2 +t^2}\leq D+t$ for every $t\geq 0$, we obtain that
$$|a_i-a_j| -|b_i-b_j|-D\leq d(x_i,x_j) - d(y_i,y_j) \leq D+|a_i-a_j| - |b_i - b_j|,$$
thus 
$$|d(x_i,x_j) - d(y_i,y_j)| < D+\mu R.$$
It follows that $\{x_i\}$ is not $(k,\mu,D)$-independent, which is a contradiction.
\end{proof}
\begin{proof}[Proof of Assertion 2]
Suppose that $R\geq \frac{D}{\mu}$ and that $\{a_i\}\subset B(0,R)\subset \mathbb R^k$ is $(k,3\mu)$-independent, but $\{x_i\}$ is not $(k,\mu,D)$-independent. Then, there is a compact metric space $(\tilde K,d_{\tilde K})$ with $\diam_{d_{\tilde K}}(\tilde K)\leq D$, $0\leq l <k$ and $y_i=(\tilde q_i,b_i)\in  \tilde K\times \mathbb R^l$ such that
$$\left| d(x_i,x_j) - d(y_i,y_j)\right| < D+\mu R.$$
Using again $d$ to denote both $d_{K\times \mathbb R^k}$ and $d_{K\times \mathbb R^l}$, we have
\begin{align*}
d(x_i,x_j) &= \sqrt{d_K(q_i,q_j)^2 +|a_i - a_j|^2},\\
d(y_i,y_j) &= \sqrt{d_{\tilde K}(\tilde q_i,\tilde q_j)^2 +|b_i - b_j|^2},
\end{align*}
thus, by $\diam_{d_K}(K)\leq D$ and $\diam_{d_{\tilde K}} (\tilde K)\leq D$,
\begin{align*}
d(x_i,x_j) - d(y_i,y_j)&\geq |a_i - a_j| - |b_i-b_j| -D,\\
d(y_i,y_j) - d(x_i,x_j) &\geq |b_i - b_j| -|a_i-a_j| - D. 
\end{align*}
It follows that
\begin{align*}
\left| |a_i-a_j| - |b_i - b_j| \right| &\leq \left| d(x_i,x_j) - d(y_i,y_j) + D \right|\\
&\leq \left| d(x_i,x_j) - d(y_i,y_j)\right| + D\\
&\leq 2D+\mu R\\
&\leq 3\mu R,
\end{align*}
since $D\leq \mu R$. Therefore, $\{a_i\}$ is not $(k,3\mu)$-independent, since $b_i\in \mathbb R^l$ with $0\leq l <k$, which is a contradiction.
\end{proof}

\begin{proposition}\label{prop:sharp_splitting}
Fix $\varepsilon>0$ and $\mu>0$. Let $(M,g(t),p)_{t\in [-2 \delta^{-1},0]}$ be a smooth compact  Ricci flow satisfying (RF1-5). Then there is $T=T(n,H)\geq 10$, $\bar\delta=\bar \delta(n,C_I,\Lambda,H,K|R,\varepsilon)>0$, $C=C(n,C_I,H|R,\mu,\varepsilon)$ and $D'=D'(n,H)<+\infty$ with the following significance.

Suppose that $R\geq \frac{D'}{\mu}$ and $\{x_i\}_{i=0}^k\subset B(p,-1,R)$ is a $(k,\mu,D')$-independent subset at $t=-1$, such that 
for each $i=0,\ldots,k$, either $(M,g(t),x_i)_{t\in [-2\delta^{-1},0]}$ is $(0,\delta)$-selfsimilar, or $\mathcal W_{x_i}(\delta)-\mathcal W_{x_i}(\delta^{-1})<\delta$. Let
 \begin{equation}\label{eqn:E_small}
\mathcal E:=
\sum_{i=0}^k \mathcal W_{x_i} (1) - \mathcal W_{x_i}(2 T)<\delta'
\end{equation}

If $0<\max(\delta,\delta')\leq \bar\delta$ then there is a $(k,\varepsilon)$-splitting map $v=(v^1,\ldots,v^k): M\times [-1,0]\rightarrow \mathbb R^k$ at scale $1$ around $p$ such that 
\begin{equation}
\int_M |\hess  v^a |^2 d\nu_{(p,0),t}  \leq C \mathcal E,  \textrm{ for } t\in [-1,-\varepsilon].\label{eqn:optimal_splitting_hessian}
\end{equation}
for any $a=1,\ldots,k$ and $t\in [-1,-\varepsilon]$.

\end{proposition}

\begin{proof}
Let $D'<+\infty$ be a constant that will be specified during the proof to depend only on $n$ and $H$, let $R\geq \frac{D'}{\mu}$, and suppose that $\{x_i\}_{i=0}^k\subset B(p,-1,R)$ is a $(k,\mu,D')$-independent subset in $(M,g(-1))$, satisfying the assumptions of the proposition.

Apply Lemma \ref{lemma:regularized_soliton_function}, for any $\delta>0$ small enough, to each of the points $x_i$, $i=0,\ldots, k$, setting $\alpha=\alpha(n,C_1,C_2)=\frac{2C_1-C_2}{2C_1}$. Then, there is $T=T(\alpha)=T(n,C_1,C_2)\geq 10$ and $w^i:M\times [-T,0]\longrightarrow \mathbb R$, $i=0,1,\ldots, k$, such that for all $t\in [-10,-\varepsilon]$
\begin{equation}\label{eqn:w_soliton_E}
|t|^E\int_M \left| |t|\ric + \hess w^i -\frac{g}{2} \right|^2 e^{\alpha f_i} d\nu_{(x_i,0),t}\leq C(n,C_I,H)  \mathcal E,
\end{equation}
where $d\nu_{(x_i,0),t}=(4\pi |t|)^{-n/2} e^{-f_i(\cdot,t)} d\vol_{g(t)}$.

Choosing $\beta=1$ in Lemma  \ref{lemma:compare_kernels}, we see that the conjugate heat kernels $u_{(x_i,0)}=(4\pi|t|)^{-n/2} e^{-f_i}$ and $u_{(p,0)}$, satisfy 

\begin{equation}\label{eqn:change_basepoint_splitting}
\begin{aligned}
u_{(p,0)}(\cdot,t) &\leq C(C_1,C_2) e^{\frac{(d_{g(t)}(x_i,p))^2}{C_1 |t|}}  e^{\alpha f_i(\cdot,t)} u_{(x_i,0)}(\cdot,t)\\
&\leq C(n,C_1,C_2) e^{C(n,C_I|\varepsilon) \frac{(d_{g(-1)}(x_i,p))^2}{C_1 |t|}}  e^{\alpha f_i(\cdot,t)} u_{(x_i,0)}(\cdot,t)\\
&\leq C(n,C_1,C_2) e^{C(n,C_I,C_1|\varepsilon) R^2} e^{\alpha f_i(\cdot,t)} u_{(x_i,0)}(\cdot,t)\\
&\leq C(n,C_I,C_1,C_2|R,\varepsilon) e^{\alpha f_i(\cdot,t)} u_{(x_i,0)}(\cdot,t).
\end{aligned}
\end{equation}
for every $t\in [-T,-\varepsilon]$. Here we also used the distance distortion estimate
$$d_{g(-1)}(x_i,p)\geq C(n,C_I|\varepsilon) d_{g(t)}(x_i,p)$$
for all $t\in [-T,-\varepsilon]$, due to (RF1).

For each $a=1,\ldots,k$, define $ L^a= w^a-w^0-w^a(x_0,-1) +w^0(x_a,-1)$, which we readily see that is a solution to the heat equation. Denote $L:M \rightarrow \mathbb R^k$,  $L=(L^1,\ldots,L^k)$.

\noindent\textit{1. Hessian estimate:}
For every $t\in[-T,-\varepsilon]$ and $a=1,\ldots,k$, we have
\begin{align*}
&\int_M |\hess L^a|^2 d\nu_{(p,0),t} \leq\\
&\leq \int_M\left| |t| \ric + \hess w^a-\frac{g}{2} -|t| \ric - \hess w^0 +\frac{g}{2}\right|^2 d\nu_{(p,0),t} \\
&\leq \int_M 2\left| |t| \ric + \hess w^a-\frac{g}{2} \right|^2 d\nu_{(p,0),t}+  \int_M 2\left| |t| \ric + \hess w^0-\frac{g}{2} \right|^2 d\nu_{(p,0),t}.
\end{align*}

Hence, by \eqref{eqn:w_soliton_E} and \eqref{eqn:change_basepoint_splitting}, for every  $t\in [-10,-\varepsilon]$
\begin{equation}\label{eqn:weighted_hessian_estimate}
\begin{aligned}
&\int_M |\hess L^a|^2 d\nu_{(p,0),t} \leq \\
&\leq C(n,C_I,C_1,C_2|R,\varepsilon) \left(2\int_M \left| |t| \ric + \hess w^0-\frac{g}{2} \right|^2 e^{\alpha f_0} d\nu_{(x_0,0),t} \right.\\
&\left.+  2\int_M \left| |t| \ric + \hess w^a-\frac{g}{2} \right|^2 e^{\alpha f_a} d\nu_{(x_a,0),t} \right)\\
&\leq C(n,C_I,C_1,C_2|R,\varepsilon)  \mathcal E.
\end{aligned}
\end{equation}

\noindent \textit{2. Pointwise estimates:} By the estimates \eqref{eqn:regularized_soliton_bound} on each $w^i$, the triangle inequality and the distance distortion estimate of assumption (RF5), we obtain that for every $(x,t)\in M\times [-T,0]$, and every $\bar t\in [-1,0]$
\begin{equation}\label{eqn:L_grad_est_global}
\begin{aligned}
|L^a|^2(x,t) + &|\nabla L^a|^2 (x,t) +|t|^E|\hess L^a|^2(x,t)\\
&\leq C(n,C_I,\Lambda, C_1, C_2, K|R)\left( (d_{g(\bar t)} (p,x))^4+1 \right),
\end{aligned}
\end{equation}
for some constant $E=E(n,C_I)<+\infty$. 

\noindent \textit{3. Normalization:}
We will show that given $\varepsilon>0$, if $0<\delta\leq \delta(n,C_I,\Lambda,C_1,C_2,K|R,\varepsilon)$ there exists a lower triangular $k\times k$ matrix $A$ is such that $ v= AL$ satisfies
\begin{equation*}
\frac{4}{3}\int_{-1}^{-1/4}\int_M \langle \nabla v^a,\nabla v^b\rangle d\nu_{(p,0),t}  dt =\delta^{ab}
\end{equation*}
and $||A||\leq C(n,\mu)$. Therefore, for any $a=1,\ldots,k$ and $t\in [-10,-\varepsilon]$
 \begin{equation}\label{eqn:hessian_tildeL}
 \int_M |\hess v^a |^2 d\nu_{(p,0),t} \leq C(n,C_I,C_1,C_2|R,\mu,\varepsilon) \mathcal E.
 \end{equation}

To prove this, suppose there is an $\varepsilon>0$, a sequence $\delta_j\rightarrow 0$, a sequence $(M_j,g_j(t),p_j)_{t\in [-2\delta_j^{-1},0]}$ of pointed smooth complete Ricci flows satisfying (RF1-5), and points $x_{i,j}\in M_j$, $i=0,1,\ldots,k$, such that the following hold.
\begin{itemize}
\item [A.] For every $j$, $\{x_{i,j}\}_{i=0}^k\subset B(p_j,-1,R)$ is $(k,\mu,D')$-independent at $t=-1$.
\item[B.] For each $i=0,\ldots, k$, either $(M_j,g_j(t),x_{i,j})_{t\in [-2\delta_j^{-1},0]}$ is $(0,\delta_j)$-selfsimilar, or $$\mathcal W_{x_{i,j}}(\delta_j)-\mathcal W_{x_{i,j}}(\delta_j^{-1})<\delta_j.$$
\item[C.] For each $i=0,1,\ldots,k$, there are functions $w^i_j:M_j\times [-T,0]\rightarrow \mathbb R$ such that
\begin{equation*}
\frac{\partial w^i_j}{\partial t}=\Delta w^i_j-\frac{n}{2},
\end{equation*}
satisfying the estimates of Lemma \ref{lemma:regularized_soliton_function} with $\varepsilon_j=1/j$ with respect to the points $x_{i,j}$ respectively, in particular estimate \eqref{eqn:w_close_f}.
\item[D.] There is $L_j:M_j\times [-T,0]\rightarrow \mathbb R^k$ defined as $L_j=(L_j^1,\ldots,L_j^k)$, where  for  $a=1,\ldots,k$, $L^a_j:M_j\times [-T,0]\rightarrow \mathbb R$ are the solutions to the heat equation  defined by 
$$L^a_j=w^a_{j}-w^0_j- w^a_j(x_{0,j},-1) +w^0_j(x_{0,j},-1).$$
\item[E.] Either there is no sequence of lower triangular $k\times k$ matrices $A_j$  such that $v_j=A_j L_j$ satisfy for $a,b=1,\ldots,k$
$$ \frac{4}{3} \int_{-1}^{-1/4}\int_{M_j} \langle \nabla v_j^a,\nabla v_j^b\rangle d\nu_{(p_j,0),t}  dt = \delta^{ab}$$
 or, if it exists, then $||A_j||\rightarrow +\infty$.\end{itemize}

By (RF1-3) and Proposition \ref{prop:compactness_rf}, we can assume that the sequence $(M_j,g_j(t),p_j)_{t\in [-2\delta_j^{-1},0]}$ converges, smoothly and locally in compact sets, to a Ricci flow $(M_\infty,g_\infty(t),p_\infty)_{t\in (-\infty,0)}$, $x_{i,j}\rightarrow x_{i,\infty}\in M_\infty \cap B(p_\infty,-1,2R)$, and that the conjugate heat flows $\nu_{(x_{i,j},0)}$ and $\nu_{(p_j,0)}$ converge to conjugate heat flows $\nu_{i,\infty}$ (with $d\nu_{i,\infty,t}=(4\pi |t|)^{-n/2} e^{-f_{i,\infty}}d\vol_{g_\infty(t)}$) and $\nu_\infty$,  which satisfy (CHF1) and (CHF2) with respect to $x_{i,\infty}$ and $p_\infty$ respectively, due to (RF3) and (RF4). 

By Assumption B, Lemma \ref{lemma:k_delta_convergence}, and Proposition \ref{prop:W_drop_small} it follows that  the Ricci flow $(M_\infty,g_\infty(t))_{t\in (-\infty,0)}$ is induced by a gradient shrinking Ricci soliton with spine $\mathcal S$, and $\nu_{i,\infty}\in\mathcal S$ for every $i=0,\ldots,k$. Hence $(M_\infty,g_\infty(t),f_{i,\infty}(\cdot,t))$ is a normalized gradient shrinking Ricci soliton at scale $|t|$, for every $t\in (-\infty,0)$. Moreover, $d_{g_\infty(-1)}(x_{i,\infty}, \mathcal S_{\textrm{point}}) \leq D=D(n,H)$, for every $i=0,\ldots,k$.


Therefore, by Proposition \ref{prop:spine_structure}, there is $0\leq l \leq n$ such that $M_\infty=M'_\infty\times \mathbb R^l$, $g_\infty(t)=g'_\infty(t)\oplus \mathbb R^l$ and $\mathcal S_{\textrm{point}}=\mathcal K\times \mathbb R^l$ for some $\mathcal K\subset M_\infty'$ with $\diam_{g'_\infty(t)}(\mathcal K)\leq A(n)\sqrt{|t|}$. 



Since $\{x_{i,j}\}_{i=0}^k$ are $(k,\mu,D')$-independent at $t=-1$, for every $j$,  their limits $\{x_{i,\infty}\}_{i=0}^k$ are also $(k,\mu, D')$-independent at $t=-1$ and contained in $B_{g'_\infty(-1)}(\mathcal K, D) \times \mathbb R^l$, with $$\diam_{g'_\infty(-1)}(B_{g'_\infty(-1)}(\mathcal K, D)) \leq 
D+A.$$
Therefore, setting $D'=D'(n,H)=D(n,H)+A(n)$, it follows that $l\geq k$, since $l<k$ contradicts the definition of $(k,\mu, D')$-independence.

Now, since each $(M_\infty,g_\infty(t),f_{i,\infty}(\cdot,t))$ is a normalized gradient Ricci soliton for every $t<0$,  and each $\nu_{i,\infty}\in \mathcal S$ satisfies (CHF2) with respect to $x_{i,\infty}=(q_i,z_i)\in M'_\infty \times \mathbb R^l$, Proposition \ref{prop:spine_structure} asserts that there is $f'\in C^\infty(M'\times (-\infty,0))$ such that for every $t\in (-\infty,0)$ and $(q,x)\in M'_\infty\times\mathbb R^l$
\begin{equation}\label{eqn:fiinfty_form}
f_{i,\infty}((q,z),t)=\frac{|z-z_i|^2}{4|t|} + f'_\infty(q,t),
\end{equation}
and $(M'_\infty,g'_\infty(t),f'_\infty(\cdot,t))$ is a normalized gradient Ricci soliton that does not split any Euclidean factors, and is independent of $i=0,\ldots,k$.  

Now, by Assumption C, after passing to a further subsequence we can assume that the functions $w^i_j\in C^\infty(M_j\times [-T,0))$ converge to functions $w^i_\infty\in C^\infty(M_\infty\times [-T,0))$, due to parabolic regularity and estimate \eqref{eqn:regularized_soliton_bound}. Moreover, 
$$w^i_\infty = |t|(f_{a,\infty} - \mu(g_\infty(-1)) ),$$
since $\mathcal W_{x_{i,j}}(1)\rightarrow \bar \mu(g_\infty(-1))$, by Proposition \ref{prop:entropy_convergence}.

Therefore, for each $a=1,\ldots,k$, $L_j^a$ smoothly converges to a smooth function $L_\infty^a$ such that
$$L^a_{\infty}=w_\infty^a-w_\infty^0-w^a_\infty(x_{0,\infty},-1) +w^0_\infty(x_{0,\infty},-1).$$
Thus, by \eqref{eqn:fiinfty_form} we obtain that $L^a_\infty((q,z),t)=-\frac{1}{2}\langle z_a-z_0,z-z_0\rangle$ and
\begin{equation}
\nabla L^a_\infty = -\frac{1}{2} (z_a^m-z_0^m)\frac{\partial}{\partial x^m}\label{eqn:nablaLiinfty}
\end{equation}
 where $x^m$ denote the coordinates on $M_\infty$  which correspond to the Euclidean factor $\mathbb R^l$, and $z_a=(z_a^1,\ldots,z_a^l)\in\mathbb R^l$.

Now, the $(k,\mu,D')$-independence of $\{x_{i,\infty}\}_{i=0}^k$ implies that $\{z_i\}_{i=0}^k\subset B(0,R)\subset \mathbb R^l$ is $(k,\mu)$-independent, by Lemma \ref{lemma:independences}, since $R\geq \frac{D'}{\mu}$. It follows that the linear functions $\frac{1}{2} \langle z_a-z_0,x\rangle$, $a=1,\ldots,k$ are linearly independent  linear functionals in $\mathbb R^l$.

Therefore, the exists a lower triangular $k\times k$ matrix $A_\infty$, with $||A_\infty||\leq C(n,\mu)$ such that if $L_\infty=(L^1_\infty,\ldots,L^k_\infty)$ then $\tilde L_\infty=(\tilde L^1_\infty,\ldots,\tilde L^k_\infty)$ with $\tilde L_\infty=A_\infty L_\infty$ satisfy
\begin{equation}
\frac{4}{3}\int_{-1}^{-1/4} \int_{M_\infty} \langle \nabla \tilde L^a_\infty,\nabla \tilde L^b_\infty\rangle  d\nu_{\infty,t} dt =\delta^{ab}
\end{equation}
 
Now, by \eqref{eqn:L_grad_est_global} and Lemma \ref{lemma:int_ker_bounds}
\begin{equation*}
\int_{-1}^{-1/4} \int_{M_j} \langle \nabla L^a_j,\nabla L^b_j \rangle d\nu_{(p_j,0),t}  dt \rightarrow \int_{-1}^{-1/4}\int_{M_\infty}\langle \nabla L^a_\infty,\nabla L^b_\infty\rangle d\nu_{\infty,t} dt,
\end{equation*}
hence there are lower triangular $k\times k$ matrices $A_j\rightarrow A_\infty$ so that 
$v_j=(v_j^1,\ldots,v_j^k)$, defined as $v_j=A_j L_j$, satisfy
\begin{equation}
\frac{4}{3}\int_{-1}^{-1/4} \int_{M_j} \langle \nabla  v^a_j,\nabla v^b_j\rangle  d\nu_{(p_j,0),t} dt =\delta^{ab},
\end{equation}
which is a contradiction.
 
 \noindent \textit{4. $L^2$-gradient estimates:} If $0<\delta\leq \delta'(n,C_I,C_1,C_2|R,\mu,\varepsilon')$ then by \eqref{eqn:E_small} and \eqref{eqn:hessian_tildeL}, for any $a=1,\ldots,k$ 
 \begin{equation}\label{eqn:hess_v_small}
 \int_{-10}^{-\varepsilon'}\int_M |\hess v^a|^2 d\nu_{(p,0),t}dt <\varepsilon',
 \end{equation}
hence, applying Lemma \ref{lemma:space_time_Poincare} to $v$ and choosing $0<\varepsilon'=\varepsilon'(n,C_I|\varepsilon)<\varepsilon$, we  also obtain
 \begin{equation*}
 \int_{-1}^{-\varepsilon} \int_M \left| \langle \nabla v^a,\nabla v^b \rangle -\delta^{ab}\right|^2 d\nu_{(p,0),t} dt <\varepsilon,
 \end{equation*}  
for any $a,b=1,\ldots,k$. Thus, $v$ is a $(k,\varepsilon)$-splitting map satisfying  all the estimates of the lemma.
 \end{proof}
 \begin{definition}
Let $T=T(n,H)$ and $D'=D'(n,H)<+\infty$ be the constants provided by Proposition \ref{prop:sharp_splitting}, $r>0$, and let $(M^n,g(t))_{t\in I}$, $ [-2Tr^2,0]\subset I$, be a smooth compact Ricci flow satisfying (RF1-5).
 
 Given $\{x_i\}_{i=0}^k \subset M$, we define
\begin{equation*}
\mathcal E^k_r(\{x_i\}_{i=0}^k)=\sum_{i=0}^k \mathcal W_{x_i}(r^2)-\mathcal W_{x_i}(2Tr^2).
\end{equation*}
Moreover, if $p\in M$, $\frac{D'}{\mu}\leq R<+\infty$  and $[-2Tr^2,0] 
\subset (-2\delta^{-1}r^2,0)\subset I$, we define
\begin{equation*}
\begin{aligned}
\mathcal E^{(k,\mu,\delta,R)}_r(p)= \inf&\left\{ \mathcal E^k_r(\{x_i\}_{i=0}^k),\{x_i\}_{i=0}^k \subset B(p,-r^2,Rr)\textrm{ is $(k,\mu, D'r)$-independent at $t=-r^2$}, \right.\\
&\left. \textrm{and for each $i=0,\ldots,k$, either } (M,g(t),x_i)_{t\in (-2\delta^{-2}r^2,0)} \textrm{is $(0,\delta)$-selfsimilar, or} \right.\\
&\left.\mathcal W_{x_i}(\delta r^2)-\mathcal W_{x_i}(\delta^{-1} r^2)<\delta.\right\}. 
\end{aligned}
\end{equation*}
We call $\mathcal E^{(k,\mu,\delta,R)}_r(p)$ the $(k,\mu,\delta,R)$-entropy pinching around $p$ at scale $r$.
\end{definition}

\begin{theorem}\label{thm:sharp_splitting}
Fix $\varepsilon>0$ and $0<\mu\leq \frac{1}{6}$. Let $(M,g(t),p)_{t\in(-2\delta^{-2},0]}$ be a smooth compact Ricci flow satisfying (RF1-5). Suppose that there is $q\in B(p,-1,R)$ such that $(M,g(t),q)_{t\in(-2\delta^{-2},0)}$ is $(k,\delta^2)$-selfsimilar around $q$ at scale $1$ with respect to $\mathcal L_{q,1}$, and if $q\not = p$ suppose that
$$\mathcal W_p(\delta)-\mathcal W_p(\delta^{-1}) <\delta.$$

There is $D'=D'(n,H)<+\infty$ such that if $R\geq \frac{D'}{\mu}$ and $0<\delta\leq \delta(n,C_I,\Lambda,H,K|R,\mu,\varepsilon)$, then there is a $(k,\varepsilon)$-splitting map $v:M\times [-1,0]\rightarrow \mathbb R^k$, $v=(v^1,\ldots,v^k)$, at scale $1$ around $p$ such that
\begin{equation*}
\int_M |\hess  v^a|^2 d\nu_{(p,0),t} \leq C(n,C_I,\Lambda,H|R,\mu,\varepsilon) \mathcal E^{(k,\mu,\delta,R)}_1(p),
\end{equation*}
for every $a=1,\ldots,k$ and $t\in [-1,-\varepsilon]$.
\end{theorem}
\begin{proof}
Apply Proposition \ref{prop:sharp_splitting} to obtain $\bar \delta=\bar\delta(n,C_I,\Lambda,H,K|R,\mu,\varepsilon)>0$ and $D'=D'(n,H)<+\infty$ such that the conclusions of Proposition \ref{prop:sharp_splitting} hold.

Moreover, by Corollary \ref{cor:distance_2D}, if $0<\delta\leq \delta(n,C_I,\Lambda,H|\varepsilon)$ then
\begin{equation}\label{eqn:small_distance}
d_{g(-1)}(p,\mathcal L_{q,1})\leq 2D.
\end{equation}
Let $q'\in \mathcal L_{q,1}$ be such that $d_{g(-1)}(p,q')\leq 2D$.

By Proposition \ref{prop:L}, if  $0<\delta\leq \delta(n,C_I,\Lambda,C_1|R,\frac{\bar\delta}{k},\delta')\leq \bar \delta$,  then $\mathcal L_{q,1}\cap B(q,-1,5R)$ at time $t=-1$ is $\delta'$-GH close to a ball of radius $5R$ in a  product metric space $(\mathcal K\times \mathbb R^l,d_{\mathcal K\times \mathbb R^l})$, $l\geq k$, with $\diam(\mathcal K)\leq A(n)$. Moreover, for each $x\in \mathcal L_{q,1}\cap B(q,-1,5R)$
\begin{equation}\label{eqn:d1_entropy_drop}
\mathcal W_x(1)-\mathcal W_x(2T)<\frac{\bar\delta}{2k}
\end{equation}
and $(M,g(t),x)_{t\in (-2\delta^{-1},0)}$ is $(0,\delta)$-selfsimilar.

Now, using Remark \ref{rmk:ind}, take any $(k,6\mu)$-independent subset of $B(0,R/2)\subset \mathbb R^k\subset\mathbb R^l$. This will also be $(k,3\mu)$-independent in $B(0,R)$. Using the $\delta'$-GH approximation, Lemma \ref{lemma:independences} and \eqref{eqn:small_distance}, we can then construct a subset $\{x_i\}_{i=0}^k\subset \mathcal L_{q,1} \cap B(p,-1,R)$ which is $(k,\mu,D')$-independent at $t=-1$, if $\delta'$ is small enough depending on $n$, $\mu$ and $R$.

Hence, by \eqref{eqn:d1_entropy_drop},
$$\mathcal E^{(k,\mu,\delta,R)}_{1}(p)\leq \sum_{i=0}^k\mathcal W_{x_i}(1)-\mathcal  W_{x_i}(2 T) <\bar\delta/2.$$
Thus, we can apply Proposition \ref{prop:sharp_splitting} for any $(k,\mu,D')$-independent at $t=-1$ subset $\{y_i\}_{i=0}^k\subset B(p,-1,R)$ such that either $(M,g(t),y_i)_{t\in (-2\delta^{-2},0)}$ is $(0,\delta)$-selfsimilar, or $\mathcal W_{y_i}(\delta)-\mathcal W_{y_i}(\delta^{-1})<\delta$, for each $i$, and take the infimum. Namely, take any such $\{y_i\}_{i=0}^k$ satisfying
$$\sum_{i=0}^k\mathcal W_{x_i}(1)-\mathcal  W_{x_i}(2 T) \leq 2  \mathcal E^{(k,\mu,\delta,R)}_{1}(p)<\bar\delta,$$ 
and let $v$ be a $(k,\varepsilon)$-splitting map constructed from $\{y_i\}_{i=0}^k$ using  Proposition \ref{prop:sharp_splitting}.
\end{proof}

\begin{proof}[Proof of Theorem \ref{intro_thm:sharp_splitting_maps}]
By Proposition \ref{prop:RF35}, $(M,g(t),p)_{t\in (-2\delta^{-2},0]}$ satisfies (RF1-5) for some constants $C_I,\Lambda,C_1,C_2,K$ depending only on $n,C_I,\Lambda$. Thus Theorem \ref{thm:sharp_splitting} proves the result.
\end{proof}

\section{A Geometric Transformation Theorem}\label{sec:GTT}

\begin{theorem}\label{thm:transformations}
Let $\varepsilon>0$, $\eta>0$, $0<\mu\leq \frac{1}{6}$, $1\leq k\leq n$, $R\geq \frac{D'}{\mu}$, where $D'=D'(n,H)<+\infty$, and $(M,g(t),p)_{t\in (-2\delta^{-2},0]}$ be a smooth compact Ricci flow satisfying (RF1-5). Let $r\in (0,1)$ and suppose that for every $s\in[r,1]$ there is a $q_s\in B(p,-s^2,Rs)$ such that $(M,g(t),q_s)_{t\in (-2\delta^{-2},0)}$ is $(k,\delta^2)$-selfsimilar but not $(k+1,\eta)$-selfsimilar at scale $s$, and if $q_s\not = p$ then
\begin{equation*}
\mathcal W_p (\delta s^2)-\mathcal W_p(\delta^{-1} s^2) <\delta
\end{equation*}
holds.

Let  $v$ be a $(k,\delta)$-splitting map $v$ around $p$ at scale $1$.

If $0<\delta\leq\delta(n,C_I,\Lambda,C_1,C_2,K,\eta|R,\mu,\varepsilon)$  then for every $r\leq s\leq 1$ there is a lower triangular $k\times k$ matrix  $T_s$ such that
\begin{enumerate}
\item $v_s:=T_s v=(v_s^1,\ldots,v_s^k)$ is a $(k,\varepsilon)$-splitting map around $p$ at scale $s$, normalized so that for any $a,b=1,\ldots,k$
\begin{equation*}
\frac{4}{3s^2} \int_{-s^2}^{-s^2/4} \int_M \langle \nabla v_s^a,\nabla v_s^b\rangle d\nu_{(p,0)} dt = \delta^{ab}.
\end{equation*}
\item Let $s_j= \frac{1}{2^j}$. Then there $\theta=\theta(n,C_I, \Lambda,C_1)$ such that for every $r\leq s\leq 1$ and $a=1,\ldots,k$
\begin{equation}
\begin{aligned}
 \int_{-s^2}^{-\frac{s^2}{4}} \int_M &|\hess v_s^a|^2 d\nu_{(p,0),t} dt \leq \\
 &\leq C(n,C_I,\Lambda,C_1,C_2,\eta|R,\mu)  \left(\delta s^\theta+\sum_{s\leq s_j\leq 1} \left( \frac{s}{s_j} \right)^\theta \mathcal E^{(k,\mu,\delta,R)}_{s_j} (p)\right).
\end{aligned}
\end{equation}
\end{enumerate}
\end{theorem}

\subsection{Adjusting a splitting map}

\begin{proposition}[Transformations]\label{prop:transformations}
Fix $\eta>0$, $\varepsilon>0$, $1\leq k\leq n$, $r\in (0,1)$ and 
let $(M,g(t),p)_{t\in (-2\delta^{-2},0]}$  be pointed smooth complete Ricci flow satisfying (RF1-3). Suppose that for every $s\in [r,1]$ there is $q_s\in B(p,-s^2,Rs)$ such that $(M,g(t),q_s)_{t\in (-2\delta^{-2},0)}$ is $(k,\delta^2)$-selfsimilar  but not  $(k+1,\eta)$-selfsimilar at scale $s$ around $q_s$. Moreover, if $q_s\not = p$ we further assume that
\begin{equation}
\mathcal W_p (\delta s^2 ) - \mathcal W_p ( \delta^{-1} s^2) <\delta.
\end{equation}

Let $v$ be a $(k,\delta)$-splitting map around $p$ at scale $1$. If  $0<\delta \leq \delta(n,C_I,\Lambda,C_1,R,\eta|\varepsilon)$ then for each $s\in [r,1]$ there is a lower triangular $k\times k$ matrix $T_s$ such that
\begin{enumerate}
\item $v_s:=T_s v$ is a $(k,\varepsilon)$-splitting map at scale $s$.
\item For any $a,b=1,\ldots,k$,
 \begin{equation}\label{eqn:vr_normalization}
\frac{4}{3s^2}\int_{-s^2}^{-s^2/4}\int_M \langle \nabla v_s^a,\nabla v_s^b\rangle  d\nu_{(p,0),t} dt = \delta^{ab}.
\end{equation}
\item $|| T_s\circ T_{2s}^{-1} - I_k||<C(n)\sqrt\varepsilon$ and whenever $r\leq s_1\leq s_2 \leq \frac{1}{2}$, we have 
\begin{equation}
||T_{s_1}\circ T_{s_2}^{-1}|| \leq \left(\frac{s_2}{s_1}\right)^{C(n)\sqrt{\varepsilon}} \quad \textrm{and}\quad ||T_{s_1}||\leq (1+\varepsilon) s_1^{-C(n)\sqrt\varepsilon},
\end{equation}
where $||\cdot||$ denotes the maximum norm of a $k\times k$ matrix.
\end{enumerate}
\end{proposition}
\begin{proof}
We will prove Assertions (1)-(3) arguing by contradiction. Suppose there is an $\eta>0$ and $\varepsilon>0$, which we may assume to be small enough, a sequence $\delta_j\rightarrow 0$, a sequence smooth compact Ricci flows $(M_j, g_j(t),p_j)_{t\in (-2\delta_j^{-2},0]}$ satisfying (RF1-3) and a sequence $v_j:M_j\times [-1,0]\rightarrow \mathbb R^k$ of $(k,\delta_j)$-splitting maps around $p_j$ at scale $1$, with the following properties:
\begin{itemize}
\item[A.] There is a sequence $r_j>0$ such that for every $j$ and  $s\in [r_j,1]$, there is $q_{j,s}\in B(p_j,-s^2,R s)$ such that  $(M_j, g_j(t),q_{j,s})_{t\in (-2\delta_j^{-2},0]}$ is $(k,\delta_j^2)$-selfsimilar but not $(k+1,\eta)$-selfsimilar  around $q_{j,s}$, at scale $s\in [r_j,1]$, and if $q_{j,s}\not = p_j$
\begin{equation}
\mathcal W_{p_j}( \delta_j s^2) - \mathcal W_{p_j} (\delta_j^{-1} s^2) <\delta_j.
\end{equation}
\item[B.]  There is a sequence $\rho_j>r_j$  such that for each $s\in [\rho_j,1]$ there is a lower triangular $k\times k$ matrix $T_{j,s}$ such that $v_{j,s}=T_{j,s} v_j$ are $(k,\varepsilon)$-splitting maps at scale $s$ satisfying for any $a,b=1,\ldots,k$
 \begin{equation}\label{eqn:adjusting_normalization}
\frac{4}{3s^2}\int_{-s^2}^{-s^2/4}\int_{M_j} \langle \nabla v_{j,s}^a,\nabla v_{j,s}^b\rangle d\nu_{(p_j,0),t} dt  =\delta^{ab}.
\end{equation}
\item[C.]
For each $j$, there is no lower triangular $k\times k$ matrix $T_{j,\rho_j/2}$ such that $v_{j,\rho_j/2}=T_{j,\rho_j/2}v_j$ is a $(k,\varepsilon)$-splitting map around $p_j$ at scale $\rho_j/2$ with 
 \begin{equation}\label{eqn:contr_norm}
\frac{4}{3\rho_j^2}\int_{-\rho_j^2/4}^{-\rho_j^2/16}\int_{M_j} \langle \nabla v_{j,\rho_j/2}^a,\nabla v_{j,\rho_j/2}^b\rangle  d\nu_{(p_j,0),t} dt =\delta^{ab},
\end{equation}
for any $a,b=1,\ldots,k$.
\end{itemize}

First note that, by Assumption B, applying  Lemma \ref{lemma:normalize_splitting} to $v_{j,2s}$, for any $s\in[\rho_j,1/2]$, we have that 
$$|| T_{j,s} T_{j,2s}^{-1} - I_k || <C(n) \sqrt\varepsilon,$$
since $v_{j,s}=T_{j,s} v_j = T_{j,s} T_{j,2s}^{-1} v_{j,2s}$. As in \cite{CJN}, it follows that for all $\rho_j\leq s_1\leq s_2 \leq 1$
\begin{equation}\label{eqn:dilation_ratio}
||T_{j,s_1} T_{j,s_2}^{-1} || \leq \left(\frac{s_2}{s_1}\right)^{C(n)\sqrt\varepsilon}.
\end{equation}
In particular, applying Lemma  \ref{lemma:normalize_splitting} to the $(k,\delta)$-splitting map $v$ we obtain $||T_{j,1} - I || \leq C(n)\sqrt\delta$,  thus
\begin{equation}
\begin{aligned}
||T_{j,s}||&= ||T_{j,s} T_{j,1}^{-1} T_{j,1}||\\
&=|| T_{j,s} T_{j,1}^{-1}(( T_{j,1}-I) +I ) ||\\
&\leq k|| T_{j,s} T_{j,1}^{-1}|| \cdot || T_{j,1}-I +I || \\
&\leq k|| T_{j,s} T_{j,1}^{-1}| \left( || T_{j,1}-I || + 1 \right)\\
&\leq k s^{-C(n) \sqrt\varepsilon} ( C(n) \sqrt \delta+1)\\
&\leq (1+\varepsilon) s^{-C(n) \sqrt\varepsilon},
\end{aligned}
\end{equation}
if $0<\delta\leq \delta(n,\varepsilon)$.

Moreover, again by Lemma \ref{lemma:normalize_splitting} and Assumptions B and C, it follows that $\rho_j\rightarrow 0$ and $r_j\rightarrow 0$.

Now, for each $a=1,\ldots,k$, let
$$c_j = \int_{M_j} v^a_{j,\rho_j} d\nu_{(p_j,0),-\rho_j^2}$$
and consider the rescaled sequences 
\begin{align*}
 \hat g_j(t)&=\rho_j^{-2} g_j( \rho_j^2 t), \\
\hat v_j^a(x,t) &= \rho_j^{-1} (v^a_{j,\rho_j}(x,\rho_j^2 t)-c_j).
\end{align*}
We will denote by $\hat \nu_{(p_j,0)}$ the backwards conjugate heat kernel measures based at $(p_j,0)$ in \linebreak $(M_j,\hat g_j(t))_{t\in[-\frac{1}{\rho_j^2},0]}$, namely 
$\hat \nu_{(p_j,0),t} =  \nu_{(p_j,0),\rho_j^2 t}$. Moreover, let $q_j = q_{j,\rho_j}\in B_{\hat g_j}(p_j,-1, R)$.

\noindent\textbf{Claim 1.} ($L^2$ estimates). For every $\tau\in [1,\rho_j^{-2}]$ and $a=1,\ldots,k$
\begin{align}
\frac{4}{3\tau}\int_{-\tau}^{-\tau/4} \int_{M_j} |\nabla \hat v_j^a|^2 d\hat\nu_{(p_j,0),t} dt& \leq \tau^{C(n)\sqrt\varepsilon}, \label{eqn:claim_nabla_tilde_vj}\\
\int_{-\tau}^{-\tau/4} \int_{M_j} |\hess \hat v_j^a|^2 d\hat\nu_{(p_j,0),t} dt &<\varepsilon \tau^{C(n)\sqrt\varepsilon},\label{eqn:claim_hessian_tilde_vj}\\
\frac{4}{3\tau}\int_{-\tau}^{-\tau/4}\int_{M_j} |\hat v_j^a|^2 d\hat\nu_{(p_j,0),t}  dt  &\leq 2\tau^{1+C(n)\sqrt\varepsilon},\label{eqn:claim_hatvj_spacetime}\\
\int_{M_j} |\hat v_j^a|^2  d\hat\nu_{(p_j,0),-\tau/4}  &\leq \frac{1}{2}\tau^{1+C(n)\sqrt\varepsilon}.\label{eqn:claim_hatvj_space}
\end{align}

\begin{proof}[Proof of Claim 1]
By Assumption B, for every $s\in [\rho_j,1]$, $v_{j,s}$ is a $(k,\varepsilon)$-splitting map at scale $s$ around $p$ and satisfies for any $a,b=1,\ldots,k$
\begin{align*}
\frac{4}{3s^2}\int_{-s^2}^{-s^2/4} \int_{M_j} \langle \nabla v_{j,s}^a,\nabla v_{j,s}^b\rangle d\nu_{(p_j,0),t} dt &= \delta^{ab},\\
\int_{-s^2}^{-s^2/4} \int_{M_j} |\hess v_{j,s}^a|^2 (\cdot,t) d\nu_{(p_j,0),t} dt &<\varepsilon,
\end{align*}
and  $v_{j,\rho_j} = T_{j,\rho_j} v_j=T_{j,\rho_j} T_{j,s}^{-1} v_{j,s}$.

Therefore,   using \eqref{eqn:dilation_ratio}, we can estimate for any $a=1,\ldots,k$
\begin{align*}
\int_{-(s/\rho_j)^2}^{-(s/2\rho_j)^2} \int_{M_j} |\nabla \hat v_j^a|^2 d\hat \nu_{(p_j,0),t}  dt&=\frac{1}{\rho_j^2}\int_{-s^2}^{-s^2/4} \int_{M_j} |\nabla v_{j,\rho_j}^a|^2 d\nu_{(p_j,0),t}  dt\\
&\leq\frac{1}{\rho_j^2} || T_{j,\rho_j} T_{j,s}^{-1}||^2 \max_a \left(\int_{-s^2}^{-s^2/4} \int_{M_j} |\nabla v_{j,s}^a|^2 d\nu_{(p_j,0),t}  dt \right)\\
&\leq \left( \frac{s}{\rho_j}\right)^{C(n)\sqrt\varepsilon} \frac{3s^2}{4\rho_j^2}.
\end{align*}

Similarly,
\begin{align*}
\int_{-(s/\rho_j)^2}^{-(s/2\rho_j)^2} \int_{M_j} |\hess \hat v_j^a|^2 d\hat\nu_{(p_j,0),t}  dt &=\int_{-s^2}^{-s^2/4} \int_{M_j} |\hess  v_{j,\rho_j}^a|^2 d\nu_{(p_j,0),t}  dt \\
&\leq || T_{j,\rho_j} T_{j,s}^{-1} ||^2  \max_a\left( \int_{-s^2}^{-s^2/4} \int_{M_j} |\hess  v_{j,s}^a|^2 d\nu_{(p_j,0),t}dt\right)\\
&\leq \varepsilon\left( \frac{s}{\rho_j}\right)^{C(n)\sqrt\varepsilon}. 
\end{align*}

It follows that for every $\tau\in [1,\frac{1}{\rho_j^2}]$
\begin{align}
\frac{4}{3\tau}\int_{-\tau}^{-\tau/4} \int_{M_j} |\nabla \hat v_j^a|^2 d\hat\nu_{(p_j,0),t}  dt \leq \tau^{C(n)\sqrt\varepsilon}\\
\int_{-\tau}^{-\tau/4} \int_{M_j} |\hess \hat v_j^a|^2 d\hat\nu_{(p_j,0),t}  dt<\varepsilon \tau^{C(n)\sqrt\varepsilon},\label{eqn:hessian_tilde_vj}
\end{align}
which prove \eqref{eqn:claim_nabla_tilde_vj} and \eqref{eqn:claim_hessian_tilde_vj}.

Now, since $\int_{M_j}\hat v_j (\cdot,t)d\hat\nu_{(p_j,0),t}  dt=0$, by the definition of $\hat v_j$, we can apply the Poincar\'e inequality, Theorem \ref{thm:poincare}, and \eqref{eqn:claim_nabla_tilde_vj}, to obtain
\begin{equation}
\begin{aligned}
\int_{-\tau}^{-\tau/4}\int_{M_j} |\hat v_j^a|^2 d\hat\nu_{(p_j,0),t}  dt &\leq \int_{-\tau}^{-\tau/4} 2|t| \int_{M_j} |\nabla \hat v_j^a|^2 d\hat\nu_{(p_j,0),t}  dt\\
&\leq  \frac{3}{2} \tau^{2+C(n)\sqrt\varepsilon}=\frac{3\tau}{4}  2\tau^{1+C(n)\sqrt\varepsilon} ,
\end{aligned}
\end{equation}
which proves \eqref{eqn:claim_hatvj_spacetime}.

Finally, observe that since $\frac{\partial}{\partial t} (\hat  v_j^a)^2=\Delta (\hat v_j^a)^2 - 2|\nabla \hat v_j^a|^2$, each function $$\tau\mapsto \int_{M_j} |\hat v_j^a|^2 d\hat \nu_{(p_j,0),-\tau}$$
is non-decreasing. Therefore, by \eqref{eqn:claim_hatvj_spacetime},
\begin{equation*}
\int_{M_j} |\hat v_j^a|^2d\hat \nu_{(p_j,0),-\tau/4} \leq \frac{4}{3\tau}\int_{-\tau}^{-\tau/4}\int_{M_j} |\hat v_j^a|^2 d\hat \nu_{(p_j,0),t} dt\leq 2\tau^{1+C(n)\sqrt\varepsilon},
\end{equation*}
which proves \eqref{eqn:claim_hatvj_space}.
\end{proof}

\noindent\textbf{Claim 2.}(Hessian concentration) If $\varepsilon<1$, then for every $t\in [-10,0)$ and $p_0>2$ there is $E=E(n,C_I)<+\infty$ and $C(n,C_I,p_0)<+\infty$ and $p(t)\geq p_0$ such that, if $j$ is large enough,
\begin{equation}\label{eqn:claim_hessian_Lp}
\left(\int_{M_j} ||\hess \hat v_j||^{p(t)} d\hat \nu_{(p_j,0),t}  \right)^{1/p(t)}\leq C(n,C_I,p_0) \varepsilon^{1/2} |t|^{-E/2}
\end{equation}
Thus, for every $\delta'>0$ there is $\bar R<+\infty$ large so that for all $t\in [-10,0)$ 
\begin{equation}\label{claim:concentration}
\int_{M_j\setminus B(p_j,t,\bar R \sqrt{|t|})} |t|^E||\hess \hat v_j||^2 d\hat \nu_{(p_j,0),t}  <\delta'.
\end{equation}
\begin{proof}[Proof of Claim 2]
By Claim 1, for every $\tau\in [1,\rho_j^{-2}]$
\begin{equation*}
\begin{aligned}
 &\int_{-\tau}^{-\tau/4}\int_{M_j} ||\hess \hat v_j||^2  d\hat \nu_{(p_j,0),t}   dt \leq \varepsilon \tau^{C(n)\sqrt\varepsilon} 
\end{aligned}
\end{equation*}
hence by the mean value theorem, there is $\bar t\in [-\tau,-\tau/4]$ so that
\begin{equation}
\int_{M_j} ||\hess \hat v_j||^2  d\hat \nu_{(p_j,0),\bar t}  \leq \frac{4\varepsilon}{3} \tau^{1+C(n)\sqrt\varepsilon}
\end{equation}
Since $\left(\frac{\partial}{\partial t}-\Delta\right)(|t|^{E/2}|\hess\hat v^a_j|)\leq 0$, by the hypercontractivity Theorem \ref{thm:hypercontractivity}, we obtain that for every $t\in [\bar t,0]$, 
\begin{equation}\label{eqn:hessian_Lp_tau}
\begin{aligned}
&\left(\int_{M_j} |t|^{Ep(t)/2} ||\hess \hat v_j||^{p(t)} d\hat \nu_{(p_j,0),t}  \right)^{1/p(t)} \leq \\
&\leq \left( \int_{M_j} |\bar t|^{E} ||\hess \hat v_j||^2 d\hat \nu_{(p_j,0),\bar t}  \right)^{1/2} \leq 2 \varepsilon^{1/2} \tau^{E/2} \tau^{\frac{1+C(n)\sqrt\varepsilon}{2}},
\end{aligned}
\end{equation}
where $p(t)=1+\frac{|\bar t|}{|t|}\geq 2$.

To obtain stimate \eqref{eqn:claim_hessian_Lp} for every $t\in [-10,0)$ it suffices to apply \eqref{eqn:hessian_Lp_tau} for $\tau_{p_0}=40p_0$.

Finally, \eqref{claim:concentration} follows from H\"older's inequality and for $\frac{2}{p(t)}+\frac{1}{q(t)}=1$ and Lemma \ref{lemma:int_ker_bounds}, since for any measurable function $v$
\begin{align*}
&\int_{M_j\setminus B(p_j,t,\bar R \sqrt{|t|})} v^2 d\hat \nu_{(p_j,0),t}  \leq\\
&\leq\left( \int_{M_j\setminus B(p_j,t,\bar R \sqrt{|t|})} v^{p(t)} d\hat \nu_{(p_j,0),t}   \right)^{1/p(t)} \left(  \int_{M_j\setminus B(p_j,t,\bar R \sqrt{|t|})} d\hat \nu_{(p_j,0),t}  \right)^{1/q(t)}.
\end{align*}
\end{proof}

Now, by Proposition \ref{prop:compactness_rf}, Assumptions (RF1-3), Assumption A and Lemma \ref{lemma:k_delta_convergence}, after passing to a subsequence, we can assume that $(M_j,\hat g_j(t),p_j)$ converges to a pointed Ricci flow $$(M_\infty=M_\infty'\times \mathbb R^k,\hat g_\infty(t)=g'_\infty(t)\oplus g_{\mathbb R^k},p_\infty)_{t\in(-\infty,0)}$$ induced by a gradient shrinking Ricci soliton satisfying (RF1), that $q_j \rightarrow q_\infty \in\mathcal S_{\textrm{point}}\cap B(p_\infty,-1,2R)$, that the conjugate heat flows $\hat \nu_{(p_j,0)}$ converge to a conjugate heat flow $\hat \nu_\infty \in\mathcal S$ that satisfies (CHF1) with respect to $p_\infty$, and that $-\Lambda \leq\mathcal W_{\hat g_j,p_j}(1) \rightarrow \bar\mu(\hat g_\infty(-1))$. In particular $-\Lambda \leq \bar\mu(\hat g_\infty(-1))\leq 0$. Moreover, $(M_\infty, \hat g_\infty(t),q_\infty)_{t\in (-\infty,0)}$ is $k$-selfsimilar but not $(k+1,\eta)$-selfsimilar at scale $1$.

By the estimates of Claim 1, arguing as in the proof of Proposition \ref{prop:s_map_compactness}, passing to a further subsequence, we can assume that  for each $a=1,\ldots,k$, $\hat v_j^a$ converges to an ancient solution $\hat v_\infty^a$ of the heat equation on $(M_\infty,\hat g_\infty(t))_{t\in (-\infty,0)}$. 

  Moreover, by \eqref{eqn:claim_hatvj_space}, \eqref{eqn:claim_nabla_tilde_vj} and Fatou's Lemma, we obtain that for every $a=1,\ldots,k$ and $t\leq -1$,
\begin{align}\label{eqn:v_infty_growth}
\int_{M_\infty}(\hat v^a_\infty)^2 d\hat \nu_{\infty,t} &\leq C |t|^{1+C(n)\sqrt\varepsilon},\\
\int_{M_\infty}|\nabla \hat v^a_\infty|^2 d\hat \nu_{\infty,t} &\leq C |t|^{C(n)\sqrt\varepsilon}.
\end{align}

Hence each $\hat v^a_\infty(\cdot,t) \in H^1_{\hat\nu_{\infty,t}}$, for every $t<0$. By Lemma \ref{lemma:growth_linear} there is $\varepsilon_0=\varepsilon_0(n,C_I,\Lambda,H,2R,\eta)>0$ such that if $C(n)\sqrt\varepsilon<\varepsilon_0$ then \eqref{eqn:v_infty_growth} implies that each $\hat v^a_\infty$ is an affine function of the $\mathbb R^k$-coordinates of $M_\infty=M_\infty'\times\mathbb R^k$, for all $t\in(-\infty,0)$.

Now, by \eqref{claim:concentration} of Claim 2, for any $\varepsilon''>0$,
\begin{align*}
\lim_{j\rightarrow+\infty} \int_{-10}^{-\varepsilon''} \int_{M_j} ||\hess \hat v_j||^2 d\hat \nu_{(p_j,0),t}  dt &=\int_{-10}^{-\varepsilon''} \int_{M_\infty} ||\hess \hat v_\infty||^2 d\hat \nu_{\infty,t} dt =0,
\end{align*}
thus for $j$ large
\begin{equation}\label{eqn:hessian_much_smaller}
\int_{-10}^{-\varepsilon''} \int_{M_j} ||\hess \hat v_j||^2 d \hat \nu_{(p_j,0),t}  dt  < \varepsilon''
\end{equation}
Therefore, by \eqref{eqn:hessian_much_smaller} and Assumption B, we can employ Lemma \ref{lemma:space_time_Poincare} to obtain for any $\varepsilon'>0$ and $0<\varepsilon''\leq \varepsilon''(n,C_I|\varepsilon')$
\begin{equation}
\int_{-1}^{-\varepsilon'} \int_M \left| \langle \nabla \hat v_j^a,\nabla \hat v_j^b\rangle -\delta^{ab}\right|^2 d\hat \nu_{(p_j,0),t}  dt <\varepsilon'.
\end{equation}

Therefore, for any $\varepsilon'>0$ and large $j$, $\hat v_j$ are $(k,\varepsilon')$-splitting maps around $p_j$ at scale $1$, in \linebreak $(M_j,\hat g_j(t))_{t\in (-\frac{1}{\rho_j^2},0)}$.

Hence, by Assertion 1 of Lemma \ref{lemma:near_splitting}, if $0<\varepsilon'\leq\varepsilon'(\varepsilon)$, $\hat v_j$ are $(k,\varepsilon/2)$-splitting maps at scale $1/2$ around $p_j$ in $(M_j,\hat g_j(t))_{t\in (-\frac{1}{\rho_j^2},0)}$. Moreover, by Lemma \ref{lemma:normalize_splitting}, there is a lower triangular $k\times k$ matrix $A_j$ such that 
\begin{equation} 
\frac{4}{3(\frac{1}{2})^{2}}\int_{-1/4}^{-1/16} \int_{M_j} \langle\nabla (A_j\hat v_j)^a,\nabla (A_j\hat v_j)^b\rangle d\hat \nu_{(p_j,0),t} dt=\delta^{ab},
\end{equation}
and
\begin{equation}\label{eqn:aj}
||A_j-I_k||<C(n)\sqrt{\varepsilon'}.
\end{equation} 
It follows that if $0<\varepsilon'\leq \varepsilon'(n|\varepsilon)$, then $A_j\hat  v_j$ are $(k,\varepsilon)$-splitting maps around $p_j$ at scale $1/2$ in $(M_j,\hat g_j(t))_{t\in (-\frac{1}{\rho_j^2},0)}$.

We conclude that, if $j$ is large enough, $v_{j,\rho_j/2}:=A_j T_{j,\rho_j} v_j$ are $(k,\varepsilon)$-splitting maps at scale $\rho_j/2$ around $p_j$ in $(M_j,g_j(t))$ satisfying \eqref{eqn:contr_norm}, which contradicts Assumption C. 
\end{proof}

\subsection{Hessian decay}

Let $(M,g(t))_{t\in [-1,0]}$ be a smooth compact Ricci flow, $v$ a $(k,\delta)$-splitting map around $p$ at scale $1$ and $h:M\times [-1,0]\rightarrow \mathbb R$ a solution of the heat equation.

Define the non-linear part $\tilde h$ of $h$ by $\tilde h= h - \xi_i v^i$, where $\xi=(\xi_1,\ldots,\xi_k)\in\mathbb R^k$ are such that
\begin{equation}\label{eqn:definition_tilde_h}
\int_{-1}^{-1/4}\int_{M}|\nabla \tilde h|^2 d\nu_{(p,0),t}  dt =\min_{\zeta=(\zeta_1,\ldots,\zeta_k) \in \mathbb R^k}\int_{-1}^{-1/4}\int_{M} |\nabla h - \zeta_i \nabla v^i|^2  d\nu_{(p,0),t}  dt,
\end{equation}
which in particular implies that
\begin{equation}\label{eqn:orthogonal_functions}
\int_{-1}^{-1/4}\int_M \langle \nabla v^a,\nabla \tilde h \rangle  d\nu_{(p,0),t}  dt= 0,
\end{equation}
for any $a=1,\ldots,k$.

Since, for any $a,b=1,\ldots,k$
\begin{align*}
\left| \int_{-1}^{-1/4} \int_M ( \langle \nabla v^a,\nabla v^b\rangle - \delta^{ab})d\nu_{(p,0),t}  dt\right| &\leq \int_{-1}^{-1/4} \int_M | \langle \nabla v^a,\nabla v^b\rangle - \delta^{ab}| d\nu_{(p,0),t}  dt\\
&\leq  \left(\int_{-1}^{-\delta}\int_M | \langle \nabla v^a,\nabla v^b\rangle - \delta^{ab}| ^2 d\nu_{(p,0),t}  dt\right)^{1/2}\\
&\leq \sqrt\delta,
\end{align*}
we know that
\begin{equation}\label{eqn:close_to_I}
\left| \frac{4}{3} \int_{-1}^{-1/4}\int_M  \langle \nabla v^a,\nabla v^b\rangle d\nu_{(p,0),t}  dt - \delta^{ab}\right| \leq \sqrt\delta.
\end{equation}
Therefore, if
\begin{equation}\label{hessian_dec_normalization}
\frac{4}{3}\int_{-1}^{-1/4} \int_{M} |\nabla h|^2 d\nu_{(p,0),t}  dt =1.
\end{equation}
then, by \eqref{eqn:orthogonal_functions}  and \eqref{eqn:close_to_I}  we obtain that 
\begin{equation}\label{eqn:xi_bound}
|\xi| \leq C(n).
\end{equation}

Thus, in the presence of $(k,\delta)$-splitting map, a solution to the heat equation decomposes as  $h=\tilde h+h_{\textrm{linear}}$, where $\tilde h$ is defined as above and $h_{\textrm{linear}}=\xi_i v^i$. The following lemma establishes the hessian decay of the non-linear part of $h$, if the Ricci flow is $(k,\delta)$-selfsimilar but not $(k+1,\eta)$-selfsimilar.

\begin{lemma}\label{lemma:hessian_decay}
Let $(M,g(t),p)_{t \in (-2\delta^{-1},0]}$ be a smooth compact Ricci flow satisfying (RF1-3) and $h: M\times [-1,0]\rightarrow \mathbb R$ a solution to the heat equation. Suppose that there is a $q\in B(p,-1,R)$ such that $(M,g(t),q)_{t \in (-2\delta^{-1},0]}$ is $(k,\delta)$-selfsimilar but not $(k+1,\eta)$-selfsimilar, and if $q\not = p$
 \begin{equation}
\mathcal W_p(\delta) -\mathcal W_p(\delta^{-1}) <\delta.
 \end{equation}
 Let $v:M\times [-1,0]\rightarrow \mathbb R^k$ be a $(k,\delta)$-splitting map around $p$ at scale $1$.

There is $\tau=\tau(n,C_I,\Lambda,C_1,\eta)=4^{-m}\in(0,1/4)$ such that
if $0<\delta\leq \delta(n,C_I,\Lambda, C_1,\eta)$ then $\tilde h$ satisfies 
\begin{align*}
\int_{-\tau}^{-\frac{\tau}{4}} \int_{M} |\hess \tilde h|^2 d\nu_{(p,0),t}  dt \leq \frac{1}{4} \int_{-1}^{-\frac{1}{4}} \int_{M} |\hess \tilde h|^2 d\nu_{(p,0),t}  dt.
\end{align*}
\end{lemma}

\begin{proof}
Since $\tilde h$ solves the heat equation, we can normalize $\tilde h$, by adding a constant, so that for every $t\in [-1,0]$
\begin{equation}\label{eqn:tilde_norm_1}
\int_M \tilde h(\cdot,t) d\nu_{(p,0),t} =0.
\end{equation}
Moreover, by scaling $\tilde h$, we can also assume that
\begin{equation}\label{eqn:tilde_norm_2}
\frac{4}{3}\int_{-1}^{-1/4}\int_M |\nabla \tilde h|^2 d\nu_{(p,0),t}  dt=1,
\end{equation}
since if $\nabla \tilde h=0$ on $M\times [-1,-1/4]$ then $\tilde h=0$ and there is nothing to prove.

\noindent\textbf{Claim:} Suppose that for some $\tau\in (0,1/4)$, $\tilde h$ satisfies
\begin{equation}\label{eqn:hessian_decay_not}
\int_{-1}^{-1/4} \int_M |\hess \tilde h|^2 d\nu_{(p,0),t}  dt<4 \int_{- \tau}^{- \tau/4} \int_M |\hess \tilde h|^2 d\nu_{(p,0),t}  dt. 
\end{equation}
then
\begin{align}
\int_{-1}^{-\tau/4} \int_M |\hess \tilde h|^2 d\nu_{(p,0),t}  dt &\leq C(n,C_I,\tau) \int_{-\tau}^{- \tau/4}\int_M |\hess \tilde h|^2 d\nu_{(p,0),t}  dt,\label{eqn:earlier_from_later_hessian}\\
\int_{-1/4}^{-1/16} \int_M |\hess \tilde h|^2 d\nu_{(p,0),t}  dt &\leq L(n,C_I) \int_{-\tau}^{- \tau/4}\int_M |\hess \tilde h|^2 d\nu_{(p,0),t}  dt.\label{eqn:decay_not}\\
1 - C(n,C_I,\tau) \int_{-\tau}^{-\tau/4} \int_M |\hess \tilde h|^2 d\nu_{(p,0),t} dt &\leq \frac{16}{3} \int_{-1/4}^{-1/16} \int_M |\nabla \tilde h|^2 d\nu_{(p,0),t} dt \leq 1.\label{eqn:grad_lower}
\end{align}
Moreover,
\begin{align}
\int_{-1}^{- \tau/4} \int_M \left(|\tilde h|^2 +|\nabla \tilde h|^2  +|\hess \tilde h|^2\right) d\nu_{(p,0),t}  dt &\leq C(n,C_I,\tau),\label{eqn:L2}\\
\int_{-1/4}^{-\tau/4} \int_M  \left( |\nabla \tilde h|^3  +|\hess \tilde h|^3 \right)d\nu_{(p,0),t}  dt &\leq C(n,C_I,\tau),\label{eqn:Lp}
\end{align}

\begin{proof}[Proof of Claim]
We first prove \eqref{eqn:decay_not} and \eqref{eqn:Lp}. By Lemma \ref{lemma:Lich_type_I} and \eqref{eqn:evol_hess} there is $E=E(n,C_I)$ so that
\begin{equation}\label{eqn:hessian_tilde_h_monotonicity}
\frac{d}{dt} \int_M |t|^E |\hess \tilde h|^2 d\nu_{(p,0),t}\leq 0.
\end{equation}
It follows that, for any $s\in [-1/4,-\tau/4]$,
\begin{equation}\label{eqn:earlier_hessian}
\begin{aligned}
|s|^E\int_M  |\hess \tilde h|^2 d\nu_{(p,0),s}&\leq\frac{1}{4^E}\int_M  |\hess \tilde h|^2 d\nu_{(p,0),-1/4} \\
&\leq \frac{4}{3}\int_{-1}^{-1/4} \int_M |t|^E |\hess \tilde h|^2 d\nu_{(p,0),t} dt\\
&\leq \frac{4}{3}\int_{-1}^{-1/4} \int_M |\hess \tilde h|^2 d\nu_{(p,0),t} dt\\
&\leq \frac{16}{3} \int_{-\tau}^{-\tau/4} \int_M |\hess \tilde h|^2 d\nu_{(p,0),t} dt,
\end{aligned}
\end{equation}
where we also used  \eqref{eqn:hessian_decay_not} for the last inequality.

Integrating \eqref{eqn:earlier_hessian} gives
\begin{equation}
\int_{-1/4}^{-1/16} \int_M |\hess \tilde h|^2 d\nu_{(p,0),t} dt \leq  L(n,C_I) \int_{-\tau}^{- \tau/4}\int_M |\hess \tilde h|^2 d\nu_{(p,0),t} dt,
\end{equation}
which proves \eqref{eqn:decay_not}, and
\begin{equation}\label{eqn:full_hessian}
\int_{-1}^{-\tau/4} \int_M |\hess \tilde h|^2 d\nu_{(p,0),t} dt \leq C(n,C_I,\tau)  \int_{-\tau}^{- \tau/4}\int_M |\hess \tilde h|^2 d\nu_{(p,0),t} dt,
\end{equation}
which proves \eqref{eqn:earlier_from_later_hessian}. Note that to arrive to \eqref{eqn:full_hessian} we need to apply \eqref{eqn:hessian_decay_not} once more.

Recall that any solution of the heat equation satisfies 
\begin{equation}\label{eqn:evol_grad}
\left(\frac{\partial}{\partial t} -\Delta \right)|\nabla \tilde h|^2 =-2|\hess \tilde h|^2.
\end{equation}
hence, 
\begin{align}\label{eqn:L2grad_evol}
\frac{d}{dt} \int_M |\nabla\tilde h|^2  d\nu_{(p,0),t}  =-2\int_M |\hess \tilde h|^2 d\nu_{(p,0),t} \leq 0.
\end{align}
Let $\bar t\in [-1,-1/4]$ be such that
$$\int_M |\nabla \tilde h|^2 d\nu_{(p,0),\bar t} = \frac{4}{3} \int_{-1}^{-1/4} \int_M |\nabla \tilde h|^2 d\nu_{(p,0),t} dt =1,$$
by the mean value theorem.

Integrating \eqref{eqn:L2grad_evol} we thus obtain for any $t_0\in [\bar t,-\tau/4]$
\begin{equation}\label{eqn:grad_hess_1}
\begin{aligned}
0\leq\int_M |\nabla \tilde h|^2 d\nu_{(p,0),t_0}&=\int_M |\nabla \tilde h|^2 d\nu_{(p,0),\bar t}-
2\int_{\bar t}^{t_0} \int_M |\hess \tilde h|^2 d\nu_{(p,0),t} dt \\
&=1-2\int_{\bar t}^{t_0} \int_M |\hess \tilde h|^2 d\nu_{(p,0),t} dt\leq 1,
\end{aligned}
\end{equation}
and for every $s_0\in [-1, \bar t]$
\begin{equation}\label{eqn:grad_hess_2}
\begin{aligned}
\int_M |\nabla \tilde h|^2 d\nu_{(p,0),s_0} &= \int_M |\nabla \tilde h|^2 d\nu_{(p,0),\bar t} +2\int_{s_0}^{\bar t} \int_M |\hess \tilde h|^2 d\nu_{(p,0),t} dt \\
&= 1+2\int_{s_0}^{\bar t} \int_M |\hess \tilde h|^2 d\nu_{(p,0),t} dt.
\end{aligned}
\end{equation}

From \eqref{eqn:full_hessian} and \eqref{eqn:grad_hess_1} it follows that
\begin{equation}\label{eqn:full_hessian_bound}
\int_{-1}^{-\tau/4} \int_M |\hess \tilde h|^2 d\nu_{(p,0),t} dt\leq C(n,C_I,\tau) \int_{-\tau}^{-\tau/4} \int_M |\hess \tilde h|^2 d\nu_{(p,0),t} dt\leq C(n,C_I,\tau), 
\end{equation}
hence \eqref{eqn:grad_hess_2} gives
\begin{equation}\label{eqn:t_h_grad_initial}
\int_M |\nabla \tilde h|^2 d\nu_{(p,0),-1} \leq C(n,C_I,\tau).
\end{equation}

Moreover, by \eqref{eqn:tilde_norm_2} and \eqref{eqn:grad_hess_1} we also obtain
\begin{equation}\label{eqn:tilde_h_gradient}
\int_{-1}^{-\tau/4} \int_M |\nabla \tilde h|^2 d\nu_{(p,0),t}dt \leq 2.
\end{equation}
These suffice to prove the  $L^2$ bounds for $\nabla \tilde h$ and $\hess \tilde h$ in \eqref{eqn:L2}. The $L^2$ bound of $\tilde h$ in \eqref{eqn:L2} follows from \eqref{eqn:tilde_norm_1}, \eqref{eqn:tilde_h_gradient} and the Poincar\'e inequality of Theorem \ref{thm:poincare}. Namely,
\begin{equation}
\int_{-1}^{-\tau/4} \int_M |\tilde h|^2 d\nu_{(p,0),t} dt \leq  \int_{-1}^{-\tau/4} 2|t|\int_M |\nabla \tilde h|^2 d\nu_{(p,0),t} dt \leq 4.
\end{equation}

On the other hand, \eqref{eqn:grad_hess_1} and \eqref{eqn:full_hessian} imply that 
$$1\geq \frac{16}{3}\int_{-1/4}^{-1/16}\int_M |\nabla \tilde h|^2 d\nu_{(p,0),t} \geq 1- C(n,C_I,\tau) \int_{-\tau}^{-\tau/4} \int_M |\hess \tilde h|^2 d\nu_{(p,0),t} dt,$$
which proves \eqref{eqn:grad_lower}.

Since $3\leq 1+\frac{1}{|t|}$ if $t\geq-1/4$, Theorem \ref{thm:hypercontractivity} and the fact that $\left(\frac{\partial}{\partial t}-\Delta\right)|\nabla \tilde h|\leq0$, by Corollarly \ref{cor:subsolutions}, imply that
\begin{equation}
\left(\int_M |\nabla\tilde h|^3  d\nu_{(p,0),t}\right)^{1/3}\leq  \left(\int_M |\nabla\tilde h|^2  d\nu_{(p,0),-1}\right)^{1/2} \leq C(n,C_I,\tau),
\end{equation}
for every $t\in [-1/4,-\tau/4]$, by \eqref{eqn:t_h_grad_initial}. Thus, integrating we obtain the gradient estimate in \eqref{eqn:Lp}.

To prove \eqref{eqn:Lp}, first note that by \eqref{eqn:full_hessian_bound} and the mean value theorem there is $\tilde t\in [-1,-1/2]$ such that
\begin{equation}\label{eqn:tilde_t_hessian}
\int_M |\hess \tilde h|^2 d\nu_{(p,0),\tilde t} \leq C(n,C_I,\tau).
\end{equation}
Now, since
\begin{align*}
\left(\frac{\partial}{\partial t}-\Delta\right)\left( |t|^{E/2} |\hess \tilde h| \right)&\leq0,
\end{align*}
and $1+\frac{|\tilde t|}{| t|} \geq 3$, for any $t\in [-1/4,-\tau/4]$, by Theorem \ref{thm:hypercontractivity} and \eqref{eqn:tilde_t_hessian} we obtain that
\begin{equation*}
|t|^{3E/2}\left(\int_M |\hess \tilde h|^3 d\nu_{(p,0),t} \right)^{1/3} \leq  \left(\int_M |\hess \tilde h|^2 d\nu_{(p,0),\bar t} \right)^{1/2}\leq C(n,C_I,\tau),
\end{equation*}
Integrating in the interval $[-1/4,-\tau/4]$ we obtain the hessian estimate in \eqref{eqn:Lp}.
\end{proof}

We will now prove the lemma arguing by contradiction. For this, suppose that there is a sequence $\delta_j\rightarrow 0$ and a pointed sequence  $(M_j,g_j(t),p_j)_{t\in [-2\delta_j^{-2},0]}$ of smooth compact Ricci flows satisfying (RF1-3) and $q_j \in B(p_j,-1,R)$ such that $(M_j,g_j(t),q_j)_{t\in [-2\delta_j^{-2},0]}$ is $(k,\delta_j)$-selfsimilar but not $(k+1,\eta)$-selfsimilar at scale $1$ around $p_j$, and if $q_j\not =p_j$
\begin{equation}
\mathcal W_{p_j}(\delta_j) - \mathcal W_{p_j}(\delta_j^{-1})<\delta_j.
\end{equation}
 Moreover, assume that there is a sequence of $(k,\delta_j)$-splitting maps $v_j:M\times [-1,0]\rightarrow \mathbb R$ around $p_j$ at scale $1$, which we can assume to be normalized so that for any $a=1,\ldots,k$
$$\int_{-1}^0\int_M v_j^a d\nu_{(p_j,0),t}  dt =0.$$

Finally, let $\tau\in (0,1/4)$ be a small constant that will be specified towards the end of the proof, and assume that $h_j:M_j\times [-1,0]\rightarrow \mathbb R$ are solutions to the heat equation, such that the corresponding $\tilde h_j$  is normalized as in \eqref{eqn:tilde_norm_1} and \eqref{eqn:tilde_norm_2} and satisfies
\begin{equation*}
\int_{-1}^{-1/4}\int_{M_j} |\hess \tilde h_j|^2 d\nu_{(p_j,0),t}  dt < 4\int_{-\tau}^{- \tau/4} \int_{M_j} |\hess \tilde h_j|^2 d\nu_{(p_j,0),t}  dt,
\end{equation*}
By the Claim it follows that each $\tilde h_j$ satisfies \eqref{eqn:earlier_from_later_hessian}, \eqref{eqn:decay_not}, \eqref{eqn:L2} and \eqref{eqn:Lp}.

Thus, by (RF1-3), Proposition \ref{prop:compactness_rf}, Lemma \ref{lemma:k_delta_convergence}, Proposition \ref{prop:W_drop_small} and Proposition \ref{prop:s_map_compactness}, passing to a subsequence if necessary, we can assume that $(M_j,g_j(t),p_j)_{t\in [-2\delta_j^{-1},0]}$ converges to a smooth complete pointed Ricci flow $(M_\infty,g_\infty(t),p_\infty)_{t\in (-\infty,0)}$ and that
\begin{enumerate}
\item $q_j\rightarrow q_\infty \in B(p_\infty,-1,2R)$ and $(M_\infty,g_\infty(t),q_\infty)_{t\in (-\infty,0)}$ is $k$-selfsimilar but not $(k+1,\eta)$-selfsimilar at scale $1$.
\item the conjugate heat kernel measures $\nu_{(p_j,0)}$ smoothly converge to  a conjugate heat flow $\nu_\infty \in\mathcal S$, that satisfies (CHF1) with respect to $p_\infty$.
\item the functions $v_j^a$ smoothly converge, uniformly in compact subsets of $M_\infty\times (-1,0)$, to  linear functions $v_\infty^a$, satisfying
\begin{equation*}
\begin{aligned}
 \int_{M_\infty}v_\infty^a d\nu_{\infty,t}  &=0, \quad\textrm{for every $t\in (-1,0)$},\\
 \frac{4}{3}\int_{-1}^{-1/4} \int_{M_\infty} \langle \nabla v_\infty^a,\nabla v_\infty^b \rangle d\nu_{\infty,t}  dt &=\delta^{ab}.
 \end{aligned}
\end{equation*} 
\end{enumerate}
Moreover, using the estimates on $\tilde h_j$ from the Claim and arguing as in the proof of Proposition \ref{prop:s_map_compactness}, we can assume that the functions $\tilde h_j$ smoothly converge, uniformly in compact subsets of $M_\infty\times [-1/4,-\tau/4]$, to a solution $\tilde h_\infty:M_\infty\times [-1/4,-\tau/4] \rightarrow \mathbb R$ of the heat equation that satisfies
\begin{align}
\int_{M_\infty} \tilde h_\infty d\nu_{\infty,t}&=0, \quad \textrm{for every $t\in [-1/4, -\tau/4]$},\\
\int_{-1/4}^{-1/16} \int_{M_\infty} |\nabla \tilde h_\infty|^2 d\nu_{\infty,t} dt &= \lim_{j\rightarrow+\infty} \int_{-1/4}^{-1/16} \int_{M_j} |\nabla \tilde h_j|^2 d\nu_{(p_j,0),t} dt, \label{eqn:limit_non_trivial}\\
\int_{-1/4}^{-1/16} \int_{M_\infty} \langle \nabla v_\infty^a,\nabla \tilde h_\infty\rangle d\nu_{\infty,t} dt  &=\lim_{j\rightarrow+\infty} \int_{-1/4}^{-1/16} \int_{M_j} \langle \nabla v_j^a,\nabla \tilde h_j\rangle d\nu_{(p_j,0),t} dt =0, \label{eqn:tildehinfty_perp}
\end{align}
for any $a=1,\ldots,k$.

In particular, to prove \eqref{eqn:limit_non_trivial} it suffices to use the $L^3$ bound on $\nabla \tilde h_j$ from the Claim, while for \eqref{eqn:tildehinfty_perp} it suffices to use the $L^2$ bound on $\nabla \tilde h_j$ from the Claim, together with the $L^2$ bound on $\nabla v_j^a$ due to $v_j$ being $(k,\delta_j$)-splitting maps.

Also, \eqref{eqn:Lp} of the Claim, H\"older's inequality and Lemma \ref{lemma:int_ker_bounds},  suffice to establish the convergence
\begin{equation}\label{eqn:hessian_convergence}
\lim_{i\rightarrow+\infty} \int_{s}^{s/4} \int_{M_i} |\hess \tilde h_j|^2 d\nu_{(p_j,0),t} dt =  \int_{s}^{s/4} \int_{M_\infty} |\hess \tilde h_\infty|^2 d\nu_{\infty,t} dt,
\end{equation}
for any $s\in [-1/4,-\tau]$.

Hence, we may pass \eqref{eqn:earlier_from_later_hessian}  and \eqref{eqn:decay_not} to the limit to obtain that 
\begin{align}
\int_{-1/4}^{-1/16} \int_{M_\infty} |\hess \tilde h_\infty|^2 d\nu_{\infty,t} dt &\leq L(n,C_I) \int_{-\tau}^{-\tau/4}\int_{M_\infty} |\hess \tilde h_\infty|^2 d\nu_{\infty,t}  dt, \label{eqn:soliton_decay_not}\\
\int_{-1/4}^{-\tau/4} \int_{M_\infty} |\hess \tilde h_\infty|^2 d\nu_{\infty,t}  dt &\leq C(n,C_I,\tau) \int_{-\tau}^{-\tau/4}\int_{M_\infty} |\hess \tilde h_\infty|^2 d\nu_{\infty,t}  dt. \label{eqn:earlier_from_later_limit}
\end{align}

Now, we claim that
\begin{equation}\label{eqn:hessian_nontrivial}
\int_{-\tau}^{-\tau/4} \int_{M_\infty} |\hess \tilde h_\infty|^2 d\nu_{\infty,t} dt \geq c(n,C_I,\Lambda,C_1,\eta,\tau)>0.
\end{equation}
Suppose this is not true.  Then by \eqref{eqn:hessian_convergence} we know that 
$$\lim_{j\rightarrow +\infty}\int_{-\tau}^{-\tau/4} \int_{M_j} |\hess \tilde h_j|^2 d\nu_{(p_j,0),t} dt =0,$$
hence \eqref{eqn:grad_lower} implies that
$$\lim_{j\rightarrow +\infty}\int_{-1/4}^{-1/16} \int_{M_j} |\nabla \tilde h_j|^2 d\nu_{(p_j,0),t} dt =1.$$
Then, by \eqref{eqn:limit_non_trivial} and \eqref{eqn:earlier_from_later_limit}, $\tilde h_\infty$ is a non-constant linear function in $M_\infty \times [-1/4,-\tau/4]$.  Moreover,  by \eqref{eqn:tildehinfty_perp}, for any $a=1,\ldots,k$, $\nabla \tilde h_\infty$ is perpendicular to $\nabla v_\infty^a$ in $M_\infty \times [-\tau,-\tau/4]$. Since $\tilde h_\infty$ and $v^a_\infty$ solve the heat equation and their hessian vanishes, it follows that $\langle \nabla \tilde h_\infty,\nabla v_\infty^a\rangle$ is constant in $M_\infty\times [-1/4,-\tilde\tau/4]$, hence $\nabla \tilde h_\infty$ is perpendicular to  each $\nabla v_\infty^a$ in $M_\infty\times [-1/4,-\tau/4]$.

 Therefore, the $k$-selfsimilar Ricci flow $(M_\infty,g_\infty(t),q_\infty)_{t\in(-\infty,0)}$ splits in fact $k+1$-Euclidean factors. This contradicts that $(M_j,g_j (t),q_j)_{t\in (-\delta_j^{-1},0)}$ is not $(k+1,\eta)$-selfsimilar for every $j$.
 
Now, by the Hessian decay estimate  Lemma \ref{lemma:hessian_decay_solitons}, there is a $\tau=\tau(n,C_I,\Lambda,H,\eta)=4^{-m}$ so that
\begin{equation}\label{eqn:soliton_hessian_decay}
\int_{-1/4}^{-1/16} \int_{M_\infty} |\hess \tilde h_\infty|^2 d\nu_{\infty,t}  dt \geq 2 L(n,C_I) \int_{-\tau}^{-\tau/4} \int_{M_\infty} |\hess \tilde h_\infty|^2 d\nu_{\infty,t} dt.
\end{equation}
Therefore,  \eqref{eqn:soliton_decay_not} and  \eqref{eqn:soliton_hessian_decay} give
\begin{align*}
L(n,C_I) \int_{-\tau}^{-\tau/4} \int_{M_\infty} |\hess \tilde h_\infty|^2 d\nu_{\infty,t}  dt &\geq  \int_{-1/4}^{-1/16} \int_{M_\infty} |\hess \tilde h_\infty|^2 d\nu_{\infty,t}  dt\\
&\geq 2L(n,C_I) \int_{-\tau}^{-\tau/4} \int_{M_\infty} |\hess \tilde h_\infty|^2 d\nu_{\infty,t} dt.
\end{align*}
Hence, by \eqref{eqn:hessian_nontrivial}, we obtain a contradiction.

\end{proof}

Now that we have a control on the hessian of the non-linear part of a solution to the heat equation, the control of the linear part will be achieved by exploiting the hessian control available for the sharp splitting maps established by Theorem \ref{thm:sharp_splitting}.

\begin{lemma}\label{lemma:decay_pinching}
There is $D'=D'(n,H)<+\infty$ such that, given a scale $r>0$, $0<\mu\leq \frac{1}{6}$ and $R\geq \frac{D'}{\mu}$, the following holds. Let $(M,g(t),p)_{t\in (-\delta^{-2}r^2,0]}$ be a pointed smooth compact Ricci flow satisfying (RF1-5). Suppose that there is $q\in B(p,-1,R)$ such that $(M,g(t),q)_{t\in (-\delta^{-2}r^2,0]}$  is $(k,\delta^2)$-selfsimilar but not $(k+1,\eta)$-selfsimilar at scale $r$, and that if $q\not = p$
\begin{equation}\label{eqn:drop_decay_pinching}
\mathcal W_p(\delta)-\mathcal W_p(\delta^{-1})<\delta.
\end{equation}

Let $h: M \times [-r^2,0]\rightarrow \mathbb R$ be a solution to the heat equation with $\frac{4}{3r^2}\int_{-r^2}^{-r^2/4}\int_M |\nabla h|^2 d\nu_{(p,0),t} dt=1$. There is $\tau=\tau(n,C_I,\Lambda,C_1)=4^{-m}\in (0,1/4)$ such that if $0<\delta\leq \delta(n,C_I,\Lambda,H,K,\eta|R,\mu)$,  then 
\begin{equation*}
\begin{aligned}
\int_{-\tau r^2}^{-\frac{\tau r^2}{4}} \int_M &|\hess h|^2  d\nu_{(p,0),t} dt \leq\\
&\leq  \frac{1}{3}  \int_{-r^2}^{-\frac{r^2}{4}} \int_M |\hess h|^2 d\nu_{(p,0),t} dt + C(n,C_I,\Lambda,H,\eta|R,\mu)\mathcal E^{(k,\mu,\delta,R)}_r (p).
\end{aligned}
\end{equation*}
\end{lemma}

\begin{proof}
By scaling we can assume $r=1$. Let $D'=D'(n,H)<+\infty$ be provided by Theorem \ref{thm:sharp_splitting}. For any $\delta'>0$, if $0<\delta\leq\delta(n,C_I,\Lambda,C_1,C_2,K|R,\mu,\delta')$, $(M,g(t),q)_{t\in (-\delta^{-1},0]}$ is $(k,\delta^2)$-selfsimilar for $q\in B(p,-1,R)$, and \eqref{eqn:drop_decay_pinching} holds if $q\not=p$, then by Theorem \ref{thm:sharp_splitting} there exists a $(k,\delta')$-splitting map  $v: M\times [-1,0]\rightarrow \mathbb R^k$ around $p$, at scale $1$. Moreover,  $v$ satisfies for each $a=1,\ldots,k$
\begin{equation}
\int_{-1}^{-\delta'}\int_M  |\hess v^a|^2 d\nu_{(p,0),t} dt\leq C(n,C_I,\Lambda,C_1,C_2|R,\mu,\delta')\mathcal E^{(k,\mu,\delta,R)}_1 (p).
\end{equation}

Define $\tilde h=h-\xi_a v^a$ and $h_{\textrm{linear}} = \xi_a v^a$ such that  $h=\tilde h+h_{\textrm{linear}}$ and \eqref{eqn:definition_tilde_h} holds. Recall that, under the normalization of $h$, we have $|\xi|\leq C(n)$ by \eqref{eqn:xi_bound}.

By Lemma \ref{lemma:hessian_decay}, if $\delta'=\delta'(n,C_I,\Lambda,C_1,\eta)>0$ is small enough, we know that $\tilde h$ satisfies, for some $\tau=\tau(n,C_I,\Lambda,C_1)\in(0,1/4)$,
\begin{equation}\label{eqn:h_d}
 \int_{-\tau}^{-\frac{\tau}{4}} \int_M  |\hess \tilde h|^2 d\nu_{(p,0),t} dt \leq \frac{1}{4}  \int_{-1}^{-\frac{1}{4}} \int_M |\hess \tilde h|^2 d\nu_{(p,0),t} dt 
\end{equation}
Therefore, using  the Peter-Paul inequality, we obtain
\begin{align*}
& \int_{-\tau}^{-\frac{\tau}{4}} \int_M |\hess h|^2 d\nu_{(p,0),t} dt =  \int_{-\tau}^{-\frac{\tau}{4}} \int_M |\hess \tilde h+\hess h_{\textrm{linear}}|^2 d\nu_{(p,0),t} dt, \\
&\leq 1.1  \int_{-\tau}^{-\frac{\tau}{4}} \int_M  |\hess \tilde h|^2 d\nu_{(p,0),t} dt + C\int_{-\tau}^{-\frac{\tau}{4}} \int_M  |\hess h_{\textrm{linear}}|^2 d\nu_{(p,0),t} dt,\\
&\leq 1.1  \int_{-\tau}^{-\frac{\tau}{4}} \int_M  |\hess \tilde h|^2 d\nu_{(p,0),t} dt + C(n)  \int_{-\tau}^{-\frac{\tau}{4}} \int_M ||\hess v||^2 d\nu_{(p,0),t} dt,
\end{align*}
and similarly
\begin{align*}
&\int_{-1}^{-\frac{1}{4}} \int_M |\hess \tilde h|^2 d\nu_{(p,0),t} dt = \int_{-1}^{-\frac{1}{4}} \int_M |\hess h - \hess h_{\textrm{linear}}|^2 d\nu_{(p,0),t} d t,\\
&\leq 1.1 \int_{-1}^{-\frac{1}{4}} \int_M |\hess  h|^2 d\nu_{(p,0),t} d t +C(n)\int_{-1}^{-\frac{1}{4}} \int_M ||\hess v||^2 d\nu_{(p,0),t} dt.
\end{align*}

Hence, applying \eqref{eqn:h_d},
\begin{align*}
&\int_{-\tau}^{-\frac{\tau}{4}} \int_M |\hess h|^2 d\nu_{(p,0),t} dt \leq \\
&\leq 1.1 \int_{-\tau}^{-\frac{\tau}{4}} \int_M  |\hess \tilde h|^2 d\nu_{(p,0),t} dt + C(n)  \int_{-\tau}^{-\frac{\tau}{4}} \int_M ||\hess v||^2 d\nu_{(p,0),t} dt,\\
&\leq \frac{1.1}{4}   \int_{-1}^{-\frac{1}{4}} \int_M |\hess \tilde h|^2 d\nu_{(p,0),t} d\tau+C(n) \int_{-\tau}^{-\frac{\tau}{4}} \int_M  ||\hess v||^2 d\nu_{(p,0),t} dt,\\ 
&\leq \frac{1}{3}\int_{-1}^{-\frac{1}{4}} \int_M |\hess  h|^2 d\nu_{(p,0),t} dt +C(n)\int_{-1}^{-\frac{1}{4}} \int_M ||\hess v||^2 d\nu_{(p,0),t} dt, \\
&+C(n)  \int_{-\tau}^{-\frac{\tau}{4}} \int_{M}  |\hess v|^2 d\nu_{(p,0),t} dt,\\
&\leq \frac{1}{3}\int_{-1}^{-\frac{1}{4}} \int_M |\hess  h|^2 d\nu_{(p,0),t} dt +C(n)  \int_{-1}^{-\delta'} \int_M  ||\hess v||^2 d\nu_{(p,0),t} dt,\\
&\leq \frac{1}{3}\int_{-1}^{-\frac{1}{4}} \int_M |\hess  h|^2 d\nu_{(p,0),t} dt + C(n,C_I,\Lambda,C_1,C_2,\eta|R,\mu)\mathcal E^{(k,\mu,\delta,R)}_1 (p),
\end{align*}
which proves the required estimate.
\end{proof}

Finally, applying Lemma \ref{lemma:decay_pinching} to the splitting maps constructed by Proposition \ref{prop:transformations} leads to the following decay estimate for splitting maps.

\begin{proposition}\label{prop:decay_transformations}
Fix $\eta>0$, $0<\mu\leq \frac{1}{6}$,  $R\geq \frac{D'}{\mu}$, where $D'=D'(n,H)$ is given by Theorem \ref{thm:sharp_splitting}, and let $(M,g(t),p)_{t \in (-\delta^{-1},0]}$ be a pointed smooth compact Ricci flow satisfying (RF1-5). Suppose that there is $r\in (0,1]$ such that, for every $s\in [r,1]$ there is $q_s\in B(p,-s^2,Rs)$ such that $(M,g(t),q_s)_{(-2\delta^{-2} s^2,0]}$ is $(k,\delta^2)$-selfsimilar but 
not $(k+1,\eta)$-selfsimilar at scale $s$, and if $q_s\not = p$ then 
\begin{equation*}
\mathcal W_p( \delta s^2)-\mathcal W_p (\delta^{-1} s^2)<\delta.
\end{equation*}
 Let $v: M\times [-1,0]\rightarrow \mathbb R^k$ be a $(k,\delta)$-splitting map around $p$ at scale $1$. 

There is $\tau=\tau(n,C_I,\Lambda,C_1)\in (0,1/4)$ such that if $0<\delta\leq \delta(n,C_I,\Lambda,H,K,\eta|R,\mu,\varepsilon)$ then for every $s \in [ r, 1]$ there is a lower triangular $k\times k$ matrix $T_s$  such that
\begin{enumerate}
\item $v_s=T_s v$ is a $(k,\varepsilon)$-splitting map around $p$ at scale $s$, normalized as in \eqref{eqn:vr_normalization}.
\item If $s\in [r,1]$ and $s\sqrt{\tau}\in [r,1]$ then
\begin{align*}
\int_{-\tau s^2}^{-\frac{\tau s^2}{4}} \int_M & ||\hess v_{s\sqrt\tau}||^2 d\nu_{(p,0),t}  dt \leq\\
&\leq \frac{1}{2} \int_{-s^2}^{-\frac{s^2}{4}} \int_M ||\hess v_s||^2 d\nu_{(p,0),t}  dt + C(n,C_I,\Lambda,H,\eta|R,\mu)\mathcal E^{(k,\mu,\delta,R)}_{s}(p).
\end{align*}
\end{enumerate}
\end{proposition}
\begin{proof}

By Proposition \ref{prop:transformations}, if $\delta>0$ is small enough and $v$ is a $(k,\delta)$-splitting map around $p$ at scale $1$, then for each $s\in [r,1]$ there is a lower triangular $k\times k$ matrix $T_r$ such that $v_s=T_sv$ are $(k,\varepsilon)$-splitting maps, normalized as in \eqref{eqn:vr_normalization}.

Let $\tau=\tau(n,C_I,\Lambda,C_1)$ be the constant provided by Lemma \ref{lemma:decay_pinching}. Then, by Lemma \ref{lemma:normalize_splitting}, if  $0<\varepsilon\leq \varepsilon(n,C_I,\tau)=\varepsilon(n,C_I,\Lambda,C_1)$ we know that
\begin{equation}\label{eqn:dilation_control}
||T_{s\sqrt\tau}T_s^{-1} - I_k||\leq 0.1,
\end{equation}
since $v_s$ is a $(k,\varepsilon)$-splitting map and $v_{s\sqrt\tau}$ satisfies the normalization \eqref{eqn:vr_normalization}.

Applying Lemma \ref{lemma:decay_pinching} to each $v_s^a$, $a=1,\ldots,k$, we obtain, if $\delta$ is small enough, that
\begin{equation*}
\begin{aligned}
\int_{-\tau s^2}^{-\frac{\tau s^2}{4}} \int_M &||\hess v_s||^2 d\nu_{(p,0),t} dt \leq\\
&\leq \frac{1}{3} \int_{-s^2}^{-\frac{s^2}{4}} \int_M ||\hess v_s||^2 d\nu_{(p,0),t} dt + C(n,C_I,\Lambda,C_1,C_2,\eta|R,\mu)\mathcal E^{(k,\mu,\delta,R)}_s(p).
\end{aligned}
\end{equation*}
Since $v_{s\sqrt\tau}= T_{s\sqrt\tau} T_s^{-1} v_s$, by \eqref{eqn:dilation_control} this suffices to obtain the required estimate.
 \end{proof}

\subsection{Proof of the Geometric Transformation Theorem \ref{thm:transformations}}

To simplify our notation we will denote
$$\mathcal E_s(p):=\mathcal E^{(k,\mu,\delta,R)}_s(p),$$
since $k,\mu,\delta,R$ are fixed and $D'=D'(n,H)<+\infty$ is the constant given by Theorem \ref{thm:sharp_splitting}.  Moreover, observe that it suffices to prove the theorem for small enough $\varepsilon>0$.

Apply Proposition \ref{prop:decay_transformations} to obtain $\tau=\tau(n,C_I,\Lambda,C_1)=4^{-m}$ and lower triangular $k\times k$ matrices $T_s$, such that, for every $s\in [r,1]$, $v_s=T_s v$ are $(k,\varepsilon)$-splitting maps satisfying
\begin{equation*}
\begin{aligned}
\int_{-\tau s^2}^{-\frac{\tau s^2}{4}} \int_M &||\hess v_{s\sqrt{\tau}} ||^2 d\nu_{(p,0),t} dt \leq \\
&\leq \frac{1}{2}  \int_{-s^2}^{-\frac{s^2}{4}} \int_M ||\hess v_s ||^2 d\nu_{(p,0),t} dt + C(n,C_I,\Lambda,C_1,C_2,\eta|R,\mu) \mathcal E_s(p).
\end{aligned}
\end{equation*}

Choose $\theta=\theta(n,C_I,\Lambda,C_1)$ such that $\tau^{\theta/2}=\frac{1}{2}$ to obtain
\begin{align*}
 \int_{-\tau s^2}^{-\frac{\tau s^2}{4}} \int_M &||\hess v_{s\sqrt{\tau}} ||^2 d\nu_{(p,0)} dt \leq \\
 &\leq\tau^{\theta/2}  \int_{-s^2}^{-\frac{s^2}{4}} \int_M ||\hess v_s ||^2 d\nu_{(p,0)} dt + C(n,C_I,\Lambda,C_1,C_2,\eta|R,\mu)\mathcal E_s(p).
\end{align*}

We will first prove the estimate for any $s_i = 2^{-i}$. Note that since $v$ is a $(k,\delta)$-splitting map around $p$ at scale $1$,  by Lemma \ref{lemma:normalize_splitting}, if $\delta>0$ is small enough and $k=0,\ldots,m-1$
\begin{equation*}
\int_{-s_i^2}^{-s_i^2/4} \int_M ||\hess v_{s_k}||^2 d\nu_{(p,0),t} dt \leq 2\int_{-s_i^2}^{-\frac{s_i^2}{4}}\int_M ||\hess v||^2 d\nu_{(p,0),t} dt<2\delta,
\end{equation*}
so there is nothing to prove in this case. 

Suppose that $i\geq m$. Then, writing $i=l m +k$, for $k=0,\ldots,m-1$,
\begin{align*}
& \int_{-s_i^2}^{-\frac{s_i^2}{4}} \int_M ||\hess v_{s_i} ||^2 d\nu_{(p,0),t} dt  \leq \\
&\leq \left(\frac{s_i}{s_{i-m}}\right)^\theta  \int_{-s_{i-m}^2}^{-\frac{s_{i-m}^2}{4}} \int_M ||\hess v_{s_{i-m}} ||^2 d\nu_{(p,0),t} dt + C(n,C_I,\Lambda,C_1,C_2,\eta|R,\mu)\mathcal E_{s_{i-m}}(p),\\
&\leq \left( \frac{s_i}{s_k}\right)^\theta  \int_{-s_k^2}^{-\frac{s_k^2}{4}} \int_M  ||\hess v_{s_k} ||^2 d\nu_{(p,0),t} dt+  C(n,C_I,\Lambda,C_1,C_2,\eta|R,\mu)\sum_{j=0}^{l-1} \left(\frac{s_i}{s_{k+jm}} \right)^\theta \mathcal E_{s_{k+jm}}(p),\\
&\leq 2\left( 2^m s_i\right)^\theta  \delta+  C(n,C_I,\Lambda,C_1,C_2,\eta|R,\mu)\sum_{s_{i-1}\leq s_j\leq 1} \left(\frac{s_i}{s_j} \right)^\theta \mathcal E_{s_j}(p).
\end{align*}

Now, suppose that $s\in (s_{i+1},s_i)$, for some $i$. If $\varepsilon>0$ is small enough, we can use Lemma \ref{lemma:normalize_splitting} to estimate
\begin{equation}\label{eqn:s_decay_hessian_1}
\begin{aligned}
&\int_{-s^2}^{-\frac{s^2}{4}} \int_M ||\hess v_s||^2 d\nu_{(p,0),t} d t \leq \\
&\leq \int_{-s_{i+1}^2}^{-s_{i+1}^2/4} \int_M ||\hess v_s||^2 d\nu_{(p,0),t}dt +\int_{-s_i^2}^{-s_i^2/4} \int_M ||\hess v_s||^2 d\nu_{(p,0),t}dt,\\
&\leq 2\left( \int_{-s_{i+1}^2}^{-s_{i+1}^2/4} \int_M ||\hess v_{s_{i+1}}||^2 d\nu_{(p,0),t}dt +\int_{-s_i^2}^{-s_i^2/4} \int_M ||\hess v_{s_i}||^2 d\nu_{(p,0),t}dt\right).
\end{aligned} 
\end{equation}

Noting that $s_{i+1}<s$ and $s_i< s/\sqrt\tau$, \eqref{eqn:s_decay_hessian_1} gives
\begin{equation*}\label{eqn:s_decay_hessian_2}
\begin{aligned}
&\int_{-s^2}^{-\frac{s^2}{4}} \int_M ||\hess v_s||^2 d\nu_{(p,0)} dt \leq\\
&\leq C\left(s_{i+1}^\theta \delta +\sum_{s_{i}\leq s_j\leq 1} \left(\frac{s_{i+1}}{s_j} \right)^\theta \mathcal E_{s_j}(p)\right) + C\left(s_i^\theta \delta +\sum_{s_{i-1}\leq s_j\leq 1} \left(\frac{s_i}{s_j} \right)^\theta \mathcal E_{s_j}(p)\right), \\
&\leq C\left( s^\theta \delta +\sum_{s \leq s_j\leq 1} \left(\frac{s}{s_j}\right)^\gamma \mathcal E_{s_j}(p) \right).
\end{aligned}
\end{equation*}
This suffices to prove Theorem \ref{thm:transformations}.

\begin{proof}[Proof of Theorem \ref{intro_thm:GTT}]
By Proposition \ref{prop:RF35}, $(M,g(t),p)_{t\in (-2\delta^{-2},0]}$ satisfies (RF1-5) for some constants $C_I,\Lambda,C_1,C_2,K$ depending only on $n,C_I,\Lambda$. Thus Theorem \ref{thm:transformations} proves the result.

\end{proof}

\section{Non-degeneration of splitting maps}\label{sec:non_degeneration}

\begin{theorem}\label{thm:non_degen}
Fix $\varepsilon>0, \eta>0$, $0<\mu\leq \frac{1}{6}$, $R\geq \frac{D'}{\mu}$, where $D'=D'(n,H)<+\infty$,  $1\leq k \leq n$ and let $(M,g(t),p)_{t\in (-2\delta^{-2},0]}$ be a pointed smooth compact Ricci flow satisfying (RF1-5). Let $r\in (0,1)$ and suppose that for every $s\in [r,1]$ there is a $q_s\in B(p,-s^2,Rs)$ such that $(M,g(t),q_s)_{t\in (-2\delta^{-2},0)}$ is $(k,\delta^2)$-selfsimilar but not $(k+1,\eta)$-selfsimilar  at scale $s$, for $1\leq k\leq n$, and if $q_s\not = p$ then
\begin{equation*}
\mathcal W_p(\delta s^2)-\mathcal W_p(\delta^{-1} s^2) <\delta
\end{equation*}
holds.

Moreover, let  $v$ be a $(k,\delta)$-splitting map $v$ around $p$ at scale $1$, and suppose that for $s_j=2^{-j}$ we have
\begin{equation}
\sum_{r \leq s_j\leq 1} \mathcal E^{(k,\mu,\delta,R)}_{s_j}(p) <\delta.
\end{equation}
If $0<\delta\leq\delta(n,C_I,\Lambda,H,K,\eta|R,\mu,\varepsilon)$, then for every $r\leq s\leq 1$, $v: M\times [-s^2,0]\rightarrow \mathbb R^k$ is a $(k,\varepsilon)$-splitting map around $p$ at scale $s$.
\end{theorem}

\begin{proof}
Let $\varepsilon'>0$ which we will choose small enough towards the end of the proof. 

By the Geometric Transformation Theorem \ref{thm:transformations}, if $0<\delta\leq \delta(n,C_I,\Lambda,H,K,\eta|R,\mu,\varepsilon')$, there is $\theta=\theta(n,C_I,\Lambda,C_1)$ such that for each $r\leq s\leq 1$ there exist lower triangular $k\times k$ matrices  $T_s$ such that $v_s=T_s v$ are $(k,\varepsilon')$-splitting maps around $p$ at scale $s$ that satisfy
\begin{equation}\label{eqn:vs_norm}
\frac{4}{3s^2} \int_{-s^2}^{-s^2/4} \int_M \langle \nabla v_s^a,\nabla v_s^b\rangle d\nu_{(p,0),t} dt = \delta^{ab}, \quad \textrm{for every $a,b=1,\ldots,k$},
\end{equation}
and
\begin{equation}
\begin{aligned}
\int_{-s^2}^{-s^2/4} \int_M & ||\hess v_s||^2 d\nu_{(p,0),t} dt \leq \\
&\leq C(n,C_I,\Lambda,H,\eta|R,\mu)  \left( \sum_{s\leq s_j\leq 1} \left( \frac{s}{r_j}\right)^\theta \mathcal E_{s_j}^{(k,\mu,\delta,R)}(p) +\delta s^\theta\right).
\end{aligned}
\end{equation}

By the mean value theorem, for any $a,b=1,\ldots,k$ and $r\leq s_i\leq 1$, there is $ t^{ab}_i\in [s_{i+1}^2,s_i^2]$ and $\tilde t^{ab}_i\in [s_{i+2}^2,s_{i+1}^2]$ so that
\begin{align*}
\int_M \langle \nabla v_{s_i}^a,\nabla v_{s_i}^b\rangle d\nu_{(p,0),t^{ab}_i} &= \delta^{ab},\\
\int_M \langle \nabla v_{s_i}^a,\nabla v_{s_i}^b\rangle  d\nu_{(p,0),\tilde t^{ab}_i}&= \frac{4}{3 s_{i+1}^2} \int_{-s_{i+1}^2}^{s_{i+2}^2} \int_M \langle \nabla v_{s_i}^a,\nabla v_{s_i}^b\rangle d\nu_{(p,0),\tilde t^{ab}_i}.
\end{align*}

On the other hand, integrating 
$$\left(\frac{\partial}{\partial t} - \Delta\right) \langle \nabla v_{s_i}^a,\nabla v_{s_i}^b\rangle = - 2\langle\hess v_{s_i}^a,\hess v_{s_i}^b\rangle,$$
we can estimate
\begin{equation} 
\begin{aligned}
&\left|\frac{4}{3s_{i+1}^2}\int_{-s_{i+1}^2}^{-s_{i+2}^2}\int_M \langle \nabla v_{s_i}^a,\nabla v_{s_i}^b\rangle  d\nu_{(p,0),t} dt-\delta^{ab} \right| \leq \\
&=\left|\int_M \langle \nabla v_{s_i}^a,\nabla v_{s_i}^b\rangle d\nu_{(p,0),\tilde t^{ab}_i}-\int_M \langle \nabla v_{s_i}^a,\nabla v_{s_i}^b\rangle d\nu_{(p,0), t^{ab}_i} \right|,  \\
&\leq 2\int_{-s_i^2}^{-s_{i+2}^2} \int_M ||\hess v_{s_i}||^2 d\nu_{(p,0),t} dt,\\
&\leq C \left( \sum_{s_i\leq s_j\leq 1} \left( \frac{s_i}{s_j}\right)^\theta \mathcal E_{s_j}^{(k,\mu,\delta,R)}(p) +\delta s_i^\theta\right)+C \left( \sum_{s_{i+1}\leq s_j\leq 1} \left( \frac{s_{i+1}}{s_j}\right)^\theta \mathcal E_{s_j}^{(k,\mu,\delta,R)}(p) +\delta s_{i+1}^\theta\right),\\
&\leq C\left( \sum_{s_{i+1}\leq s_j\leq 1} \left( \frac{s_{i+1}}{s_j}\right)^\theta \mathcal E_{s_j}^{(k,\mu,\delta,R)}(p) +\delta s_{i+1}^\theta \right)=\chi_{i+1}.
\end{aligned}
\end{equation}

Therefore, by Gram-Schmidt, there is a lower triangular $k\times k$ matrix $A$ such that
$$\frac{4}{3s_{i+1}^2}\int_{-s_{i+1}^2}^{-s_{i+2}^2}\int_M \langle \nabla A v_{s_i}^a,\nabla A v_{s_i}^b\rangle  d\nu_{(p,0),t} dt=\delta^{ab},$$
and $||A_i-I_k||\leq C(n) \chi_{i+1}$. 

Since $v_{s_{i+1}}=T_{s_{i+1}}T_{s_i}^{-1} v_{s_i}$ also satisfies the normalization \eqref{eqn:vs_norm}, by the uniqueness of the Cholesky decomposition we conclude that $A_i=T_{s_{i+1}}T_{s_i}^{-1}$, hence 
$$||T_{s_{i+1}} T_{s_i}^{-1} - I_k||\leq C(n) \chi_{i+1}.$$

Therefore, since for every $l=1,\ldots$, $T_{s_{l}} T_{s_0}^{-1}=( T_{s_{l}}T_{s_{l-1}}^{-1}) \cdots (T_{s_1} T_{s_0}^{-1}) $, we obtain as in \cite{CJN}
\begin{equation}\label{eqn:tri}
\begin{aligned}
|| T_{s_{l}} T_{s_0}^{-1}- I_k|| &\leq  \prod_{j=0}^{l-1} \left(1+k ||T_{s_{j+1}}T_{s_j}^{-1} - I_k|| \right) -1,\\
&\leq  \prod_{j=0}^{l-1} \left(1+kC(n)  \chi_{j+1}\right) -1= e^{\sum_{j=0}^{l-1} \log(1+kC(n)\chi_{j+1})} -1,\\
&\leq e^{\sum_{j=0}^{l-1} C(n) \chi_{j+1}}-1 \leq C(n)\sum_{j=0}^{l-1}\chi_{j+1}.
\end{aligned}
\end{equation}

On the other hand
\begin{align*}
\sum_{j=0}^{l-1} \chi_{j+1} &= C \sum_{j=0}^{l-1} \left(  s_{j+1}^\theta \delta + \sum_{s_{j+1} \leq s_m \leq 1} \left( \frac{s_{j+1}}{s_m}\right)^\theta \mathcal E^{(k,\mu,\delta,R)}_{s_m}(p) \right),\\
&\leq C\delta \sum_{j=0}^{l-1}  s_{j+1}^\theta + C\sum_{j=0}^{l-1} \sum_{s_{j+1} \leq s_m \leq 1} \left( \frac{s_{j+1}}{s_m}\right)^\theta \mathcal E^{(k,\mu,\delta,R)}_{s_m}(p), \\
&\leq C\delta +  C \sum_{j=0}^{l-1} \sum_{s_{j+1} \leq s_m \leq 1} \frac{1}{2^{\theta(j+1-m)}} \mathcal E^{(k,\mu,\delta,R)}_{s_m}(p),\\
&\leq C\delta + C\sum_{s_l \leq s_m \leq 1} \mathcal E^{(k,\mu,\delta,R)}_{s_m}(p)
 \sum_{j=m}^l 2^{-\theta(j+1-m)}, \\
 &\leq C\left(\delta +\sum_{s_l \leq s_m \leq 1} \mathcal E^{(k,\mu,\delta,R)}_{s_m}(p)\right).
 \end{align*}
 Therefore, by \eqref{eqn:tri}, it follows that $||T_{s_l} T_{s_0}^{-1}-I_k ||<\varepsilon''$ for every $s_l\in[ r,1]$, if $\delta>0$ is small enough.
 
In general, namely if $\max(s_{l+1},r)\leq s\leq s_l$ for some $l$, using  Lemma \ref{lemma:normalize_splitting} we obtain that if $0<\varepsilon' \leq \varepsilon'(n|\varepsilon'')$ then $||T_s\circ T_{s_l}^{-1} - I_k ||\leq \varepsilon''$. Since $T_s T_{s_0}^{-1}=T_s T_{s_l}^{-1}T_{s_l}T_{s_0}^{-1}$ we can estimate
\begin{equation}
\begin{aligned}
||T_s T_{s_0}^{-1}-I_k|| &\leq ||T_s\circ T_{s_l}^{-1}-I_k|| + ||T_{s_l} T_{s_0}^{-1}-I_k|| + k||T_s\circ T_{s_l}^{-1}-I_k|| \; ||T_{s_l} T_{s_0}^{-1}-I_k|| \\
&\leq 2\varepsilon'' + k (\varepsilon'')^2.
\end{aligned}
\end{equation}

Therefore, since $T_s=T_s T_{s_0}^{-1} T_{s_0}$
\begin{equation*}
||T_s - I_k|| \leq ||T_s T_{s_0}^{-1} -I_k|| +||T_{s_0}^{-1}-I_k|| + k||T_s T_{s_0}^{-1} -I_k || ||T_{s_0}^{-1}-I_k ||,
\end{equation*}
 and by Lemma \ref{lemma:normalize_splitting}, $||T_{s_0}^{-1}-I_k||$ can be made arbitrarily small, if $\delta>0$ is small enough.
 
Thus, for $\varepsilon''$ small enough we conclude that $v$ is a $(k,\varepsilon)$-splitting map around $p$ at every scale $s\in [ r,1]$.
\end{proof}

\begin{proof}[Proof of Theorem \ref{intro_thm:non_degen}]
By Proposition \ref{prop:RF35}, $(M,g(t),p)_{t\in (-2\delta^{-2},0]}$ satisfies (RF1-5) for some constants $C_I,\Lambda,H,K$ depending only on $n,C_I,\Lambda$. Thus Theorem \ref{thm:non_degen} proves the result.

\end{proof}

\large{\bf Acknowledgments:}  The author would like to thank Andrea Mondino for bringing \cite{CJN} to his attention and Felix Schulze for many interesting discussions. The author would also like to thank  Konstantinos Leskas and George Zacharopoulos for their helpful comments on earlier versions of the paper.

{\bf Funding:} The research was supported by the Hellenic Foundation for Research
and Innovation (H.F.R.I.) under the “2nd Call for H.F.R.I. Research Projects to support
Faculty Members \& Researchers” (Project Number: HFRI-FM20-2958).

{\bf Data availability:} Data sharing is not applicable to this article as no datasets were generated or analyzed during the current study.

{\bf Declarations}\\
{\bf{Conflict of interest:}} The author declares that he has no conflict of interest.

\end{document}